\documentclass[aop]{imsart}
\usepackage{amsmath}
\usepackage{amsthm}
\usepackage{amsfonts}
\usepackage{amscd}
\usepackage{amssymb}
\usepackage{amstext}
\usepackage{hyperref}
\usepackage{epsfig}

\numberwithin{equation}{section}

\allowdisplaybreaks
\startlocaldefs
\newtheorem{theorem}{Theorem}[section]
\newtheorem{proposition}{Proposition}[section]
\newtheorem{lemma}{Lemma}[section]
\newtheorem{remark}{Remark}[section]
\newtheorem{definition}{Definition}[section]

\renewcommand{\epsilon}{\varepsilon}

\newcommand{\abs}[1]{\left\vert #1\right\vert}
\newcommand{\1}[1]{{\mathbf 1}{\{#1\}}}
\newcommand{\R}{\mathbb{R}}
\newcommand{\N}{\mathbb{N}}
\newcommand{\Z}{\mathbb{Z}}
 
\newcommand{\PR}{\mathbb{P}}
\newcommand{\ES}{\mathbb{E}}

\newcommand{\Hc}{\mathcal{H}}
\newcommand{\lv}{\vec{\ell}}

\endlocaldefs
\begin{document}


\begin{frontmatter}
\title{Scaling limits for sub-ballistic biased random walks in random conductances}
\runtitle{Biased random walks in random conductances}

\begin{aug}
\author{\fnms{Alexander} \snm{Fribergh}\ead[label=e1]{fribergh@dms.umontreal.ca}}
\and
\author{\fnms{Daniel} \snm{Kious}\ead[label=e2]{daniel.kious@nyu.edu}}
\runauthor{A.~Fribergh and D.~Kious}

\affiliation{Universit\'e de Montr\'eal and New York University Shanghai}
\address{Universit\'e de Montr\'eal, DMS\\
Pavillon Andr\'e-Aisenstadt\\     2920, chemin de la Tour Montréal,\\  H3T 1J4 (Qu\'ebec), Canada\\
\printead{e1}}
\address{New York University Shanghai,\\ 1555 Century Avenue,\\
Pudong, Shanghai\\
200122, China\\
\printead{e2}}
\end{aug}

\begin{abstract}
We consider biased random walks in positive random conductances on the d-dimensional lattice in the zero-speed regime and study their scaling limits. We obtain a functional Law of Large Numbers for the position of the walker, properly rescaled. Moreover, we state a functional Central Limit Theorem where an atypical process, related to the Fractional Kinetics, appears in the limit.
\end{abstract}



\end{frontmatter}



\section{Introduction}\label{sect_intro}

Random walks in random environments (RWRE) have been the subject of intense research for over fifteen years. We refer the reader to~\cite{Zeitouni}, \cite{S1}, \cite{Topics}, \cite{kumagai} and~\cite{benarous_fribergh} for different surveys of the field.

 One aspect that has attracted a lot of attention is the phenomenon of trapping. Trapping appears in several physical systems and it motivated  the introduction of an idealized model known as the Bouchaud trap model (BTM). The study of the BTM led to the discovery of several interesting anomalous limiting processes including the FIN diffusion (see~\cite{fin02}) and the Fractional Kinetics (FK) process (see \cite{mula}). Moreover, the BTM is a natural setting to witness \lq\lq aging\rq\rq, which is the phenomenon where the time it takes to witness a significant change in the system is of
the order of the \lq\lq age\rq\rq of the system. This behavior is common among dynamics in random media such as dynamics on spin glasses~\cite{BABC}  (see~\cite{Bouchaud}  for a physical overview of spin glasses) as well as in the random energy model under Glauber dynamics (see~\cite{bbgtwo}) or in parabolic Anderson model (see~\cite{MOS}). For an overview of the BTM we refer the reader to~\cite{BenarousCerny}.

Although the BTM was initially  used to study random walks which are symmetric in the sense that, up to a time-change, behaves like a brownian motion (see~\cite{BC} and~\cite{Mourrat}) it was subsequently used to study the behaviour of directionally transient RWREs which experience trapping. The initial series of works in this direction~\cite{ESZ1}, \cite{ESZ2} and~\cite{ESZ3} by Enriquez, Sabot and Zindy  were concerned with the one dimensional RWRE. The authors managed to obtain  the scaling limits in the zero-speed regime (a result already obtained in~\cite{KKS}). Furthermore, they also proved that this model experiences \lq\lq aging\rq\rq. This work also provided a robust method for studying directionally transient RWRE with trapping which served as an inspiration for future works on more complex graphs.

Following the study of one dimensional RWREs, several works were initiated to understand trapping for directionally transient RWREs on trees, see~\cite{FG}, \cite{BH}, \cite{H} and~\cite{Aidekon1}. These works confirmed that the picture provided by BTM was relevant for RWRE in general environment (up to some complication due to lattice effects, see~\cite{FG}).

 It is a central question to prove that directionally transient RWREs in $\Z^d$ can also be analyzed via the BTM analogy and to identify the limiting behaviors such models may have. So far the methods provided by the BTM have only been used to identify scaling exponents for biased random walks in positive random conductances~\cite{Fri11}, a model first studied in~\cite{Shen}, and on supercritical percolation clusters~\cite{FH}, which was first studied by~\cite{BGP} and~\cite{SznPerco}.
 
In this paper, we carry on the study of the zero-speed regime for biased random walks in positive random conductances in $\Z^d$ initiated in~\cite{Fri11}. Our main result is to find the limiting scaling processes appearing in those models. One of the scaling limits we identify is related to the Fractional Kinetics  in $\Z^d$. This constitutes the first scaling limit result that is rigorously proved for anisotropic random walks in random environments on $\Z^d$.

\subsection{Definition of the model}

We introduce $ {\bf P}[\,\cdot\,]=P_*^{\otimes E(\Z^d)} $, where $P_*$ is the law of a positive random variable $c_*\in (0,\infty)$. This measure gives a random environment usually denoted $\omega$.

In order to define the random walk, we introduce a bias $\ell=\lambda \vec \ell$ of strength $\lambda >0$ and  direction $\vec \ell$ which is in the unit sphere with respect to the Euclidian metric of $\R^d$. In an environment $\omega$, we consider the Markov chain of law $P_x^{\omega}$ on $\Z^d$ with $X_0=x$ $P_x^{\omega}$-a.s.~and  transition probabilities $p^{\omega}(x,y)$ for $x,y\in \Z^d$ defined by
\begin{align}\label{deftransition}
p^{\omega}(x,y) =\frac{c^{\omega}(x,y)}{\sum_{z \sim x} c^{\omega}(x,z)},
\end{align}
where $x\sim y$ means that $x$ and $y$ are adjacent in $\Z^d$ and also we set
\begin{equation}\label{def_conduct}
\text{for all $x\sim y \in \Z^d$,} \qquad \ c^{\omega}(x,y)=c_*^{\omega}([x,y])e^{(y+x)\cdot\ell}.
                                            \end{equation}

This Markov chain is reversible with invariant measure given by 
\begin{align}\label{defpi}
\pi^{\omega}(x)=\sum_{y \sim x} c^{\omega}(x,y).
\end{align}

The random variable $c^{\omega}(x,y)$ is called the conductance between $x$ and $y$ in the configuration $\omega$. This comes from the links existing between reversible Markov chains and electrical networks. We refer the reader to~\cite{DoyleSnell} and~\cite{LP} for a further background on this relation, which we will use extensively. Moreover for an edge $e=[x,y]\in E(\Z^d)$, we denote $c^{\omega}(e)= c^{\omega}(x,y)$.

Finally the annealed law of the biased random walk will be the semi-direct product $\PR = {\bf P}[\,\cdot\,] \times  P_0^{\omega}[\,\cdot\,]$.

In the case where $c_*\in( 1/K, K)$ for some $K<\infty$, the walk is {\it uniformly elliptic} and this model is the one previously studied in~\cite{Shen}. Later on, this work was generalized in \cite{Fri11}. Results of both papers can be stated in the following manner.
\begin{theorem}[\cite{Shen},\cite{Fri11}]
\label{the_theorem}
For $d\geq2$, we have 
\[
\lim \frac{X_n} n =v , \qquad \PR\text{-a.s.},
\]
where
\begin{enumerate}
\item if $E_*[c_*]<\infty $, then $v\cdot \vec{\ell} >0$,
\item if $E_*[c_*]=\infty$, then $v=\vec 0$.
\end{enumerate}

Moreover, if $\lim \frac{\ln P_*[c_*>n]}{\ln n}=-\gamma$ with $\gamma<1$ then
\[
\lim \frac{\ln X_n \cdot \vec{\ell}}{\ln n} =\gamma, \qquad \PR\text{-a.s.}.
\]
\end{theorem}

From this result, we see that a natural trapping regime occurs when $\gamma<1$.\\

In this paper, we are interested in this sub-ballistic regime. For the rest of the paper, we will naturally assume that
\begin{align}\label{assProba}
{\bf P}[c_*\ge t]=L(t)t^{-\gamma}\text{, for any }t\ge 0,
\end{align}
with $\gamma\in(0,1)$ and where $L$ is a slowly-varying function. We choose such a form for the tail of $c_*$ in order to, on one hand, be in the sub-ballistic regime (provided by $\gamma\in(0,1)$) and, on the other hand, in order to have some regularity, which is ensured by the slowly-varying function. If we do not assume this kind of regularity, this would not be possible to obtain full asymptotic results but we could prove some convergence along some sub-sequences of time. See for example \cite{FG} where the authors treat this kind of difficulties which arise when the distribution of $c_*$ is lattice.

%

\subsection{Main results}\label{sect_questions}

Our main results are a functional Law of Large Numbers and a functional Central Limit Theorem for the position of the walker.\\
For any time $T>0$, we denote $D^d([0,T])$ the space of c\`adl\`ag functions from $[0,T]$ to $\R^d$. The following results of convergence hold on the space $D^d([0,T])$ equipped with the uniform topology or the Skorokhod's $J_1$-topology, see \cite{Whitt} for details.
\begin{theorem}\label{mainth}
Consider the biased random walk among random conductances on $\mathbb{Z}^d$, $d\ge2$, with law given by \eqref{assProba}. There exist a deterministic unit vector $v_0\in\mathbb{S}^{d-1}$ with $v_0\cdot \vec\ell>0$ and a constant $C>0$ such that, for any $T\in\mathbb{R}_+$, under the annealed law $\PR$,
\[
\left(\frac{X_{\lfloor nt\rfloor}}{n^{\gamma}/L(n)}\right)_{t\in [0,T]} \xrightarrow{(d)}{} \left(C \mathcal{S}_\gamma^{-1}(t)v_0\right)_{t\in [0,T]},
\]
on $D^d([0,T])$ in the uniform topology, where  $S^{-1}_{\gamma}(\cdot)$ is the inverse of a stable subordinator with index $\gamma$. Moreover, there exists a deterministic $d\times d$ matrix $M_d$ of rank $d-1$ such that, for any $T\in\mathbb{R}_+$,
\[
\left(\frac{X_{\lfloor nt\rfloor}-\left(X_{\lfloor nt\rfloor}\cdot v_0\right)v_0}{\sqrt{n^{\gamma}/L(n)}}\right)_{t\in [0,T]} \xrightarrow{(d)}{} \left(M_dB_{\mathcal{S}_\gamma^{-1}(t)}\right)_{t\in [0,T]}.
\]
on $D^d([0,T])$ in the $J_1$-topology, where   $B_\cdot$ is a standard $d$-dimensional Brownian motion, independent of $S^{-1}_{\gamma}(\cdot)$.
\end{theorem}

\begin{remark}
Our main theorem does not give any information on the fluctuation in the direction $v_0$. In the course of this paper we will obtain a result on those fluctuations which appears in Theorem \ref{seinfeld}. The drawback of this result is that the re-centering is random, depending on regeneration times, and that, to state this result, we need to introduce a significant amount of notations. This is why we choose not to state it here.\\
Besides, the matrix $M_d$ is defined in \eqref{water} where we prove that it has rank $d-1$. We also point that $P_{v_0}M_d$ is the null matrix, where $P_{v_0}$ is the projection matrix on $v_0$.\\
Finally, note that we cannot extend the second result to the uniform topology, due to measurability issues, as explained in Section 11.5.3 of \cite{Whitt}.
\end{remark}

\begin{remark}
Previous scaling-limit results were obtained in the isotropic case, for example an annealed~\cite{DM} and quenched CLT was obtained for the simple random walk on the supercritical percolation cluster (\cite{BB}, \cite{MP} and \cite{SS}), the variable speed random walk in random conductances (\cite{SA}) and the random walk in bounded conductances (\cite{M} and~\cite{BP}). The only other limiting process that had appeared was the Fractional Kinetics in the case of the random walk in unbounded random conductances in~\cite{BC}.

Our main result is the first scaling limit result for anisotropic random walks and we believe that this type of result will prove to be universal for random walks in random environments with directional transience that experience trapping. In particular, we expect this type of result to appear for the biased random walk on supercritical percolation cluster. Several steps in our proof will be easily transferable to other models. However one key step (the description of the environment seen by the particle around a large trap) is extremely model dependent and will require substantial work in any other model that will be analyzed in the future.
\end{remark}

\begin{remark}
We believe that the estimates proved in this paper along with techniques from~\cite{ESZ3} should be sufficient to prove aging results in this model.
\end{remark}

\begin{remark}
If we were to consider the {\it variable speed random walk} defined in \cite{BC}, we believe a similar theorem would hold with $n^{\gamma}/L(n)$ replaced by $n$ and $\mathcal{S}_\gamma^{-1}(t)$ replaced by $C_0 t$, where $C_0$ is some constant.
\end{remark}


The process $B_{\mathcal{S}_\gamma^{-1}(\cdot)}$ is known as the Fractional Kinetics. This kind of process has already been found to be the scaling limit of symmetric processes with trapping (see \cite{BenarousCerny,BenarousCerny2,BC,Mourrat}).\\



\subsection{Sketch of proof}\label{sketch}

The proof of the main result is rather long and involved. For this reason, we start the paper by giving a sketch of proof which highlights the structure and the main estimates of the paper. \\

It is known, see~\cite{Fri11}, that the walk is slowed down by the presence of small trapping areas in the environment. The sub-ballisticity condition exhibited in this paper is equivalent to the fact that the annealed exit time of an edge  is infinite. \\
This leads one to believe that the most efficient trapping mechanism is to have one edge with large conductance surrounded by regular edges, and indeed this intuition will turn out to be correct. 

After having identified the geometry of efficient traps, we are going to implement a strategy for the analysis of anisotropic trapping models which was first developed in the works of Enriquez, Sabot and Zindy~\cite{ESZ1,ESZ2,ESZ3} on $\Z$ and later extended to trees in~\cite{FG}, \cite{BH}, \cite{H}.

Apart from identifying the geometry of efficient traps, one of the most difficult tasks is to analyze the tail of the time spent in large traps. One particularly difficult aspect is to describe the environment seen from the particle close to a deep trap.

\subsubsection{Regeneration times and independent regeneration blocks}
The standard approach to  study directionally transient walks is to use the {\it regeneration times} $\tau_1,\ldots,\tau_n$ associated to the walk, which is a particular increasing sequence of random times (see Section \ref{subsecregen}). Our main problem is then to study the scaling limit of a sum of i.i.d.~random variables, namely $\sum_{i=1}^n (\tau_{i+1}-\tau_i)$ where each term corresponds to one {\it regeneration period}.

The two key elements to prove are the following:
\begin{enumerate}
\item the only regeneration periods that matter are those where edges with large conductances are met, here large means above a certain specific cut-off $f(n)$. Furthermore, in those blocks, the time spent outside of the largest conductance does not matter (see Proposition \ref{justin}). This means that $T_i$, the time spent on the largest edge of the $i$-th regeneration block, is a good approximation for $\tau_{i+1}-\tau_i$ when the latter time matters.
\item the time spent during the $i$-th regeneration period on the edge with the largest conductance $c_{*,(i)}^{\text{max}}$ (conditionally on the event that this conductance is large) can asymptotically be written as $c_{*,(i)}^{\text{max}} W^{(i)}_{\infty}$ (see Proposition \ref{final_coupling}) where $W^{(i)}_{\infty}$ is an independent random variable with high enough moments (see Lemma \ref{moment_winf}). This allows us to compute the tail of $c_{*,(i)}^{\text{max}}W^{(i)}_{\infty}$ (see Lemma \ref{final_tail}) and thus, in some sense, the tail of the time spent on the edge with the largest conductance.
\end{enumerate}

This procedure is summed up as follows
\begin{align*}
\sum_{i=0}^{n-1} (\tau_{i+1}-\tau_i) \approx \sum_{i=0}^{n-1} T_i \1{c_{*,(i)}^{\text{max}}\geq f(n)} & \approx \sum_{i=0}^{n-1} W_{\infty}^{(i)}c_{*,(i)}^{\text{max}} \1{c_{*,(i)}^{\text{max}}\geq f(n)} \\ & \approx \sum_{i=0}^{n-1} W_{\infty}^{(i)}c_{*,(i)}^{\text{max}},
           \end{align*}
           the last line being a standard estimate on sums of heavy tailed random variables (see \cite{Durrett}). 
           
           Hence $\tau_n$  behaves like a sum of i.i.d.~random variables whose tails can be computed  and as such we can obtain scaling limit results (see Proposition \ref{justin2}).

Of the two key elements we need as an input, the first one can be proved using techniques similar to \cite{Fri11}, this is done in Section~\ref{pfff17}. The second point is the more difficult estimate, let us now discuss the difficulties arising to obtain that estimate.

\subsubsection{Analysis of the time spent in one edge with large conductance}
This is the most important step.
The key aspects to understand the time spent in an edge with large conductance are the following
\begin{itemize}
\item how likely are we to hit this edge?
\item after having hit it, how much time does it take to come out of it?
\item how likely are we to come back?
\end{itemize}

Let us start by addressing the second question. In a typical situation, once we have entered an edge with large conductance, we will perform a lot of back and forth crossings of that edge. The number of such crossings is roughly geometric (see Lemma \ref{coupl_exp}) and the exit probability of the large edge is almost proportional to the conductances of the adjacent edges (see Lemma \ref{exit_pro}).

The first and third questions are related to the asymptotic environment seen from the particle around a large edge at late time. It turns out that the analysis of this object can be done in a very intuitive manner. Indeed, seen from outside the trap, located at an edge $e$ say, the environment looks like the usual environment where $e$ has been collapsed into a vertex $x_e$, see Figure \ref{collapsing}. Hence the environment seen from the particle should be a weighted average of such environments, to factor in the likelihood of hitting such a vertex. This reweighting is done rigorously using regeneration times see Lemma \ref{big_trap_cond}.


\begin{figure}[h]   
\includegraphics{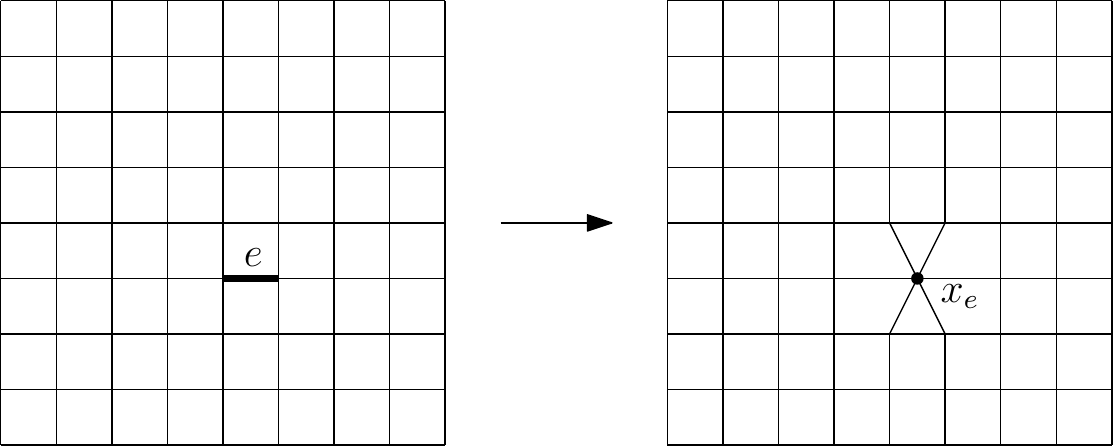}
   \caption{\label{collapsing} Collapsing and edge $e$ into a vertex $x_e$.}
\end{figure}

In fine, we will show that the time $T$ spent by the walker on some edge $e$, when $c_*(e)$ is large, will asymptotically behave like
\begin{equation}\label{linkin}
c_*(e)\times \left(\frac{2}{\overline{\pi}_\infty}\sum_{i=1}^{V_\infty}{\bf e_i}\right),
\end{equation}
where $\overline{\pi}_\infty$ is a random variable linked to the conductances surrounding $x_e$ (or $e$ equivalently) in the asymptotic environment, $V_\infty$ is the number of visits to $x_e$ and $({\bf e_i})$ is an independent sequence of i.i.d.~exponential random variables with mean $1$ (related to the large geometric number of back and forth crossings). The conductance $c_*(e)$ is independent of all these quantities. As we know the tail of $c_*(e)$ and because the sum in \eqref{linkin} behaves nicely, we are able to compute easily the tail of $T$ (see Lemma \ref{final_tail}).

\subsubsection{Conclusion}
We now have the two main ingredients to conclude. First, using the i.i.d.~structure on the trajectory of the walk, we easily obtain limit theorems at regeneration times. Second, we are able to analyze the time spent during one regeneration period. The conclusion will then follow using a classical inversion argument and theorems on the limits of stochastic processes (see Theorem \ref{seinfeld}).

\section{Notations}\label{sect_notations}
In this section, we define most of the necessary notations.\\
Let us denote $\{e_1,...,e_d\}$ an orthonormal basis of $\mathbb{Z}^d$ such that $e_1\cdot\lv\ge e_2\cdot\lv\ge...\ge e_d\cdot\lv\ge0$, and define, for any $i\in\{1,...,d\}$, $e_{i+d}:=-e_i$. The set $\{\pm e_1,\ldots,\pm e_d\}$ will be denoted
by $\nu$. In particular, we have that $e_1\cdot\lv\ge 1/\sqrt{d}$. Set $f_1:=\lv$ and complete it into an orthonormal basis $(f_i)_{1\le i \le d}$ of $\Z^d$.\\
We set, for any $z\in\R$,
\begin{equation}\label{johnlennonstar}
\Hc^+(z):=\left\{x\in\Z^d: x\cdot \lv>z\right\} \text{ and } \Hc^-(z):=\left\{x\in\Z^d: x\cdot \lv\le z\right\}.
\end{equation}
We also define the shorthand notations
\[
\Hc^+_x:=\Hc^+(x\cdot \lv) \text{ and } \Hc^-_x:=\Hc^-(x\cdot \lv).
\]

Two vertices $x,y\in \Z^d$ are called neighbours, or adjacent, denoted $x\sim y$, if $||x-y||=1$, where $||\cdot||$ is the euclidean distance. Besides, for a vertex $y$ and an edge $e$, we write $y\sim e$ if $y\sim e^+$ or $y\sim e^-$ and $y\notin \{e^+,e^-\}$. Two edges $e$ and $e'$ are said to be adjacent or neighbours, denoted $e\sim e'$, if they share exactly one endpoint. We use the notation $y\in e$ if $y\in\{e^+,e^-\}$. If $y$ is a neighbour of $e$ we denote $e_y$ the unique endpoint of $e$ which is adjacent to $y$.

For any pair of neighbouring vertices $x,y\in\Z^d$, we denote $[x,y]$ the unit non-oriented edge linking them. For a vertex $x\in\Z^d$ we denote $||x||_\infty$ the usual uniform norm and, for an edge $[x,y]$ of $\Z^d$, we denote
\[
||[x,y]||_\infty=||x||_\infty\vee||y||_\infty.
\]
Given a set $V$ of vertices of $\Z^d$, we denote by $|V|$ its cardinality, by $E(V):=\left\{[x,y] \text{ s.t.~} x,y\in V\right\}$; we also define its \emph{vertex-boundary}
\[
\partial V:= \left\{x\notin V:\exists y\in V, y\sim x\right\}
\]
as well as its \emph{edge-boundary}
\[
\partial_E V := \left\{[x,y]\in E\left(\Z^d\right): x\in V, \ y\notin V\right\}.
\]
Given an edge $e\in E(\Z^d)$ we denote by $e^+$ and $e^-$ its endpoints, in an arbitrary order if not precised otherwise. Given a set $E$ of edges in $\Z^d$, we denote by $V(E):=\left\{x\in \Z^d: \exists e\in E\text{ s.t.~}  x\in \{e^+,e^-\}\right\}$ its vertices.\\
For any subset
$A$ of vertices or edges of $\Z^d$, we define the {\it width} of $A$ to be
\begin{align}\label{defwith}
W(A)=\max_{1 \leq i \leq d} \Bigl(\max_{y\in A} y\cdot
e_i - \min_{y\in
A} y\cdot e_i \Bigr).
\end{align}
For any $L, L'\ge0$ and $y\in\Z$, define the tilted box
\[
B_y(L,L'):=\left\{x\in\Z^d: |(x-y)\cdot\lv|\le L\text{ and } |(x-y)\cdot f_i|\le L'\text{, for all } i\in\{2,...,d\}\right\},
\]
and its \emph{positive boundary}
\[
\partial^+B_y(L,L'):=\left\{x\in \partial B_y(L,L'): |(x-y)\cdot\lv|>L\right\}.
\]
Besides, we denote $B(L,L')=B_0(L,L')$. Also, we define the ball in the uniform norm $B_\infty(y,r)$ center at $y$ and of radius $r$.\\
For any graph on which a random walk $(X_n)_n$ is defined, let $A$ be some subset of its vertices and define its hitting times
\[
T_A:=\inf\left\{n\ge0:X_n\in A\right\}\text{ and } T^+_A:=\inf\left\{n\ge1:X_n\in A\right\},
\]
and its exit time
\[
T^{\text{ex}}_A:=\inf\left\{n\ge0: X_n\notin A\right\}.
\]
We will use the abuse of notation $T_x$ when $A$ is the singleton $\{x\}$.
We denote $\theta_n$ the time shift by $n$ units of times.\\
Besides, for any real numbers $a$ and $b$, we denote
\[
a\vee b:=\max\{a,b\}\text{ and } a\wedge b:=\min\{a,b\}.
\]
Throughout the paper, the letters $c$ and $C$ denote constants in $(0,\infty)$ that may depend on the dimension $d$, the strength of the bias $\lambda$ and the law $P_*$. Moreover their value may change from line to line.\\

Finally, we will define, later on,  several probability measures that we cannot define properly yet. Nevertheless, let us point out where they are defined and explain roughly their purpose. Firstly, we will define $\mathbb{P}_0^K$, in Definition \ref{defP0K}, which is an annealed measure in a special environment  which is key for understanding the environment such that $0$ is conditioned to be a regeneration time. Secondly, we will introduce the shorthand notation $\overline{\mathbb{P}}$, defined in \eqref{romeo}, which is the law of a regeneration block. Thirdly, we will define a probability measure $\mathbb{P}_n$ at \eqref{defPn} which provides the law of a regeneration block conditioned on meeting a large trap.

\section{Some previous results: estimates on the backtracking event}
As explained in the introduction, this work is the sequel of the work \cite{Fri11} done by the first author. Hence, we use several results from this paper. Let us state two of them and very briefly explain their content.\\

\begin{theorem}[Theorem $5.1$ of \cite{Fri11}]
\label{BL}
For $\alpha>d+3$
\[
\PR[T_{\partial B(L,L^{\alpha})}\neq T_{ \partial^+ B(L,L^{\alpha})}]
\leq Ce^{-cL}.
\]
\end{theorem}

\begin{lemma}[Lemma $7.3$ of \cite{Fri11}]
\label{thisisboring}
We have for any $n$,
\[
\PR[T_{\mathcal{H}^-(-n)}<\infty]\leq C\exp(-cn).
\]
\end{lemma}

These two results look alike. The first one states that the walk, started at zero, will  exit a large tilted box through the positive boundary (towards $\lv$) with large probability.\\
The second result states that the probability of ever backtracking at distance $n$ is exponentially small in $n$.

\section{Good/Bad areas decomposition}\label{sect_badareas}
Here, we will follow the idea already used in \cite{Fri11}, which consists in partitioning the space into \emph{good parts} where the walk is well-behaved and \emph{bad parts} where we have much less control.\\

Recall that, in the model we consider, we have no uniform ellipticity, but it will be, at some places, convenient to consider only edges which are \emph{typical}, in the sense that their conductances are neither too small nor too large. This will enable us to obtain several estimates.\\
For this purpose, we will define some vocabulary. These definitions depend on some real number $K\ge1$, which we will choose later to be large. 
\begin{definition}\label{defnormaledge} Fix a constant $K\ge1$.
We say that an edge $e\in E(\Z^d)$ is $K$-normal if $c_*(e)\in[1/K,K]$, otherwise the edge $e$ is said to be abnormal. 
\end{definition}
Note that we can choose the probability for an edge to be abnormal as small as we need by taking $K$ large enough: indeed, for any edge $e\in E(\Z^d)$, $c_*(e)\in(0,+\infty)$ almost surely, hence $P_*\left[e \text{ is abnormal}\right]=P_*\left[c_*(e)\notin[1/K,K]\right]$ goes to $0$ as $K$ goes to infinity.
\begin{definition}\label{defopenvertex} Fix a constant $K\ge1$.
A vertex $x\in\Z^d$ is $K$-open if, for any neighbouring site $y$, the edge $[x,y]$ is K-normal. A vertex that is not $K$-open is said to be $K$-closed.
\end{definition}
As before, the probability for a vertex to be $K$-closed can be made arbitrarily small be taking $K$ large enough.

\begin{remark}
\label{fakeUE}
We may notice that for any \emph{open} vertex $x$, using \eqref{def_conduct} and \eqref{defpi}, we have
\[
\frac{c}Ke^{2\lambda x\cdot\vec{\ell}} \leq c^{\omega}\bigl([x,y]\bigr)\leq
CKe^{2\lambda x\cdot\vec{\ell}},
\]
for any vertex $y$ adjacent to $x$ in $\Z^d$, and thus
\[
\frac{c}Ke^{2\lambda x\cdot\vec{\ell}}\leq\pi^{\omega}(x)\leq
CKe^{2\lambda x\cdot\vec{\ell}}.
\]
\end{remark}

\begin{definition}\label{defgoodvertex} Fix a constant $K\ge1$.
A vertex $x\in\Z^d$ is $K$-good if there exists an \emph{infinite directed $K$-open path} starting at $x$, that is a path $\{x_0,x_1,x_2,...\}$ with $x_0=x$ and such that, for all $i\ge0$,
\begin{itemize}
\item[(1)] we have $x_{2i+1} -x_{2i} = e_1$ and $x_{2i+2}-x_{2i+1} \in
\{e_1,\ldots,e_d\}$;
\item[(2)] $x_i$ is $K$-open.
\end{itemize}
If a vertex is not $K$-good, it is said to be $K$-bad. 
\end{definition}

\begin{remark}\label{remopenkey}
The key property
of a good point will be that there exists a open path $(x_i)_{i\geq0}$
such that, for all $i\ge0$, $x_{i}\cdot\lv\le x_{i+1}\cdot\lv$, and 
$(x_i-x_0)\cdot\vec{\ell} \geq c(d) i$.
\end{remark}

\begin{remark}
Note that $d\geq 2$ is crucial in the decomposition between good and bad points. Indeed in dimension $1$ there exists only one directed path. This means that $K$-good points exists only if the environment is uniformly elliptic which is not the context we are considering.
\end{remark}

For notational simplicity, we will often call edges and vertices open, closed, normal, etc, forgetting the dependence in $K$.\\

Let us state some results proved in \cite{Fri11}. 
For any environment $\omega$, let us denote $\mathrm{BAD}^\omega_K(x)$ the connected component of $K$-bad
vertices containing~$x$, in case $x$ is good then $\mathrm
{BAD}^\omega_K(x)=\emptyset$. We will often forget to indicate the dependence in $K$ or $\omega$ when it is clear from the context. Also, recall the definition \eqref{defwith} of $W(\cdot)$.

\begin{lemma}[Lemma $5.1$ of \cite{Fri11}]
\label{BLsizeclosedbox}
There exists $K_0<\infty$ such that, for any $K\geq K_0$ and for any $x\in\Z
^d$, we have that the cluster $\mathrm{BAD}_K(x)$ is finite ${\mathbf
P}$-a.s. and
\[
{\mathbf P}\bigl[ W\bigl(\mathrm{BAD}_K(x)\bigr) \geq n
\bigr] \leq C\exp\bigl(-\xi_1(K)n\bigr),
\]
where $\xi_1(K)\to\infty$ as $K$ tends to infinity. In particular, this implies that ${\bf P}\left[0\text{ is good}\right]>0$.
\end{lemma}

For $x \in\Z^d$, we define $\mathrm{BAD}^{\mathrm{s}}_x(K,\omega)=\{x\}\cup\bigcup_{y\sim x} \mathrm{BAD}_K^\omega(y)$ the union of all
bad areas adjacent to $x$. Notice that when $x\in\mathrm{BAD}_K(\omega)$, we have $\mathrm{BAD}^{\mathrm{s}}_x(K,\omega)=\mathrm{BAD}_K^\omega(x)$. Again, we will often drop the dependency on $\omega$ or on $K$ in the notation of  $\mathrm{BAD}^{\mathrm{s}}_x(K,\omega)$.\\
As a direct consequence of Lemma \ref{BLsizeclosedbox}, we have:
%
\begin{lemma}[Lemma 8.2 of \cite{Fri11}]\label{tgb}
There exists $K_0<\infty$ such that, for any $K\ge K_0$, 
$\mathrm{BAD}^{\mathrm{s}}_x(K)$ is finite ${\mathbf P}$-a.s.~and
\[
{\mathbf P}\bigl[ W\bigl(\mathrm{BAD}^{\mathrm{s}}_x(K)
\bigr) \geq n\bigr] \leq C\exp\bigl(-\xi_1(K)n\bigr),
\]
where $\xi_1(K)\to\infty$ as $K$ tends to infinity.
\end{lemma}

Let us define $\mathrm{BAD}_K=\bigcup_{x\in\Z^d} \mathrm{BAD}_K(x)$ which
is a
union of finite sets. Also we set $\mathrm{GOOD}_K=\Z^d \setminus
\mathrm{BAD}_K$. We may notice that
%
\begin{equation}
\label{disjoint} \mbox{for any $x\in\mbox{BAD}_K$,}\qquad\partial
\mathrm{BAD}_K(x) \subset\mathrm{GOOD}_K,
\end{equation}
since $\mathrm{BAD}_K(x)$ is a connected component of bad points.

In the sequel, $K$ will always be large enough so that $\mathrm{BAD}_K(x)$ is finite for any $x\in\Z^d$.

\begin{lemma}[Lemma 8.1 of \cite{Fri11}]
\label{backbone}Fix an environment $\omega$.
For any $x\in\mathrm{GOOD}_K(\omega)$, we have
\[
E^{\omega}_x \Biggl[\sum_{i=0}^{\infty}
{\mathbf1} {\{X_i=x\}} \Biggr]\leq C(K)<\infty.
\]
\end{lemma}

For any set of edges $A\subset E(\Z^d)$ and any $V\subset \Z^d$, let us define
\begin{align}\label{kurt}
T^+_{V,A}:= \sum_{e\in A} \left| \left\{k\in[1,T^+_{V}]:[X_{k-1},X_k]=e\right\}\right|,
\end{align}
and notice that $T^+_{V}=T^+_{V,A}+T^+_{V,A^c}$. Besides, for any $t\in\R_+\cup\{+\infty\}$, define
\begin{equation}\label{def_Et}
E_{< t}:=\left\{e\in E(\Z^d): c_*(e)< t\right\}.
\end{equation}

The statement and the proof of the following result are close to those of Lemma 8.3 in \cite{Fri11}. It will be used to prove that the time spent on edges with a conductance that is not too large is asymptotically negligible.
%
\begin{lemma}
\label{timetrap}
Fix an environment $\omega$, for any $x\in\Z^d$ which is open, and any $t\in\R_+\cup\{+\infty\}$ we have that
\[
E_x^{\omega}\bigl[T^+_{\mathrm{GOOD}(\omega)\cup\{x\},E_{< t}}\bigr]\leq C(K) \exp
\bigl(3\lambda W\left(\mathrm{BAD}^{\mathrm{s}}_x(
\omega)\right)\bigr) \biggl(1+\sum_{e\in
E(\mathrm{BAD}^{\mathrm{s}}_x)\cap E_{< t}}
c_*^{\omega}(e) \biggr).
\]
\end{lemma}

\begin{proof}
The first remark to be made is that since $x$ is open, then , for all $y\sim x $,
$c_*([x,y])\in[1/K,K]$. Besides, if $\mathrm{BAD}^{\mathrm{s}}_x\setminus \{x\}=\emptyset$, the result is obvious, thus we assume from now on that $\mathrm{BAD}^{\mathrm{s}}_x\setminus\{x\}\neq\emptyset$.\\
Let us introduce the notation $\mathrm{BAD}^{\mathrm{ss}}_x(K)=
\mathrm{BAD}^{\mathrm{s}}_x(K)\setminus\{x\}$.

Now, let us consider the finite network obtained by taking $
\mathrm{BAD}^{\mathrm{ss}}_x(\omega)\cup\partial\mathrm{BAD}^{\mathrm{ss}}_x(\omega)$ and merging all points of $ \partial
\mathrm{BAD}^{\mathrm{ss}}_x(\omega)$ (which contains $x$) to one point $\delta$ and removing all the ensuing loops. We denote
$\omega_{\delta}$ the resulting graph which is finite by
Lemma \ref{tgb} and Lemma \ref{BLsizeclosedbox}, and it is also connected, since the different connected components
$\mathrm{BAD}_K(y)$, $x\sim y$, are connected through $x$.\\
By Lemma \ref{psyche} from the Appendix, we have that, for any $y\in \mathrm{BAD}^{\mathrm{ss}}_x(K)$, $y\sim x$,
\begin{eqnarray*}
E^{\omega_\delta}_{y}[T_{\delta,E_{<t}}^+] 
&\le& \frac{2}{c^\omega(y,x)}  \sum_{e\in
E(\mathrm{BAD}^{\mathrm{ss}}_x)\cap E_{< t}} c^\omega(e),
\end{eqnarray*}
where we used that $[x,y]$ is an edge linking $y$ and $\delta$ in $\omega_\delta$.

Besides, using the fact the transition probabilities of the random walk in $\omega_{\delta}$ at
any point different from $\delta$ are the same as that of the walk in
$\omega$, Markov's property yields
\begin{align}\nonumber
E_x^{\omega}\bigl[T^+_{\mathrm{GOOD}(\omega)\cup\{x\},E_{< t}}\bigr]&\le1+\sum_{\substack{y\in\mathrm{BAD}^{\mathrm{ss}}_x(K):\\ y\sim x}} \frac{c^\omega(x,y)}{\pi^\omega(x)}E_y^{\omega_\delta}[T_{\delta,E_{<t}}^+]\\ \label{funkybaby}
&\le 1+\frac{c(d)}{\pi^\omega(x)} \sum_{e\in
E(\mathrm{BAD}^{\mathrm{ss}}_x)\cap E_{< t}} c^\omega(e).
\end{align}
Now, by \eqref{def_conduct}, and using the fact that $x$ is open and $x\in \partial\mathrm{BAD}^{\mathrm{ss}}_x$, we have that
\[
\pi^\omega(x)\ge C(K)\exp\Bigl(2 \lambda\min_{y\in\partial\mathrm{BAD}^{\mathrm{ss}}_x}y \cdot\vec{\ell} \Bigr)
\]
and, for $e\in E(\mathrm{BAD}^{\mathrm{ss}}_x)\cap E_{< t}$,
\[
c^\omega(e)\le c^\omega_*(e)\exp\Bigl(2 \lambda\max_{y\in\partial\mathrm{BAD}^{\mathrm{ss}}_x}y \cdot\vec{\ell} \Bigr).
\]
Moreover, recalling definition \eqref{defwith} of $W(\cdot)$ and since $\mathrm{BAD}^{\mathrm{ss}}_x\cup\partial\mathrm{BAD}^{\mathrm
{ss}}_x$ is connected, we have
\[
\max_{y\in\partial\mathrm{BAD}^{\mathrm{ss}}_x} y \cdot\vec{\ell
}-\min_{y\in\partial\mathrm{BAD}^{\mathrm{ss}}_x}
y \cdot\vec{\ell}\leq W\left(\mathrm{BAD}^{\mathrm{ss}}_x(
\omega)\right)+2\le W\left(\mathrm{BAD}^{\mathrm{s}}_x(
\omega)\right)+3.
\]
Therefore, we conclude from \eqref{funkybaby} that
\begin{align*}
E_x^{\omega}\bigl[T^+_{\mathrm{GOOD}(\omega)\cup\{x\},E_{< t}}\bigr]\le &C(K) \exp
\bigl(3\lambda W\left(\mathrm{BAD}^{\mathrm{s}}_x(
\omega)\right)\bigr) \biggl(1+\sum_{e\in
E(\mathrm{BAD}^{\mathrm{s}}_x)\cap E_{< t}}
c_*^{\omega}(e) \biggr).
\end{align*}
\end{proof}


\section{Regeneration times}\label{subsecregen}

A classical tool for analyzing directionally transient RWREs is to use a
regeneration structure, see \cite{SznZer}. We call ladder-point a new maximum of
the random walk in the direction $\vec{\ell}$. 

The standard way of
constructing regeneration times is to consider successive ladder points
and argue that there is a positive probability of never backtracking
again, i.e.~there exists a ladder-point $X_{n_0}$ such that $X_{n_0+k}\cdot\lv> X_{n_0}\cdot\lv$ for any $k\ge 1$. Such a ladder-point creates a separation between the past and
the future of the random walk leading to interesting independence
properties. We call this point a regeneration time.

There are two major issues in our case. Firstly, we do not have any type of uniform
ellipticity:
this has been addressed in \cite{Fri11} by considering \emph{open ladder-points}, so we will follow this strategy. Secondly, for the reversible model that we consider, the classical construction of regeneration times yields regeneration slabs (and quantities defined on them) that are ergodic but not independent. As we want to get limit theorems, it will be much more convenient to recover some independence. For this purpose, we will introduce an alternative construction of the model and define a slightly different version of regeneration times in order to obtain some independence properties (see Theorem \ref{thindep}). After the completion of the paper, we were made aware that a similar argument for creating independence of regeneration blocks had already been developped in~\cite{GMP}.


\subsection{Construction of an enhanced random walk}\label{sect_altcons}

In this section, we will  define an \emph{enhanced} walk $(\widetilde{X})_n=(X_n,Z_n)_n$, which will be such that, first, the marginal law of $(X_n)$ is the law of the original anisotropic walk that we study and, second, some extra information is encapsulated in the variables $Z_n$ concerning the last step of walk. We need to construct this enhanced walk in order to be able to define some regeneration times with nice independence properties.\\
As the classical anisotropic walk, the process $(\widetilde{X}_n)$ has two levels of randomness, one given by the environment, and one corresponding to the evolution of this enhanced walk in this environment. The definition of the environment is the same as before, that is a collection of random conductances with law ${\bf P}$.\\
Let us now explicit the law of the process $(\widetilde{X}_n)$. For this purpose, fix an environment $\omega$ and recall from \eqref{deftransition}, \eqref{def_conduct} and \eqref{defpi} that, for all $x\in\Z^d$ and $j\in\{1,...,2d\}$,
\begin{eqnarray*}
p^\omega(x,x+e_j)&=&\frac{c_*(x,x+e_j)e^{e_j\cdot\ell}}{\sum_{i=1}^{2d} c_*(x,x+e_i)e^{e_i\cdot\ell}}=\frac{c(x,x+e_j)}{\pi^\omega(x)},
\end{eqnarray*}
and define
\begin{eqnarray}
p^\omega_K(x,x+e_j)&:=&\frac{\left(c_*(x,x+e_j)\wedge K^{-1}\right)e^{e_j\cdot\ell}}{\sum_{i=1}^{2d} \left(c_*(x,x+e_i)\vee K\right)e^{e_i\cdot\ell}} \le p^\omega(x,x+e_j).\label{defpjK}
\end{eqnarray}
Moreover $p^\omega(x,y)=p^\omega_K(x,y)=0$ if $y$ is not a neighbour of $x$ in $\Z^d$.\\
In the environment $\omega$, we define, for any starting state $(x,z)\in\Z^d\times\{0,1\}$, the Markov chain $(\widetilde{X}_n)$ with law $\widetilde{P}^\omega_{(x,z)}$ on $\Z^d\times \{0,1\}$ and transition probabilities $\tilde{p}^\omega((y_1,z_1),(y_2,z_2))$ for $y_1,y_2\in\Z^d$ and $z_1,z_2\in\{0,1\}$ defined by:
\begin{enumerate}
\item $\widetilde{X}_0=(x,z),  \widetilde{P}^\omega_{(x,z)}$-a.s.,
\item $\tilde{p}^\omega((y_1,z_1),(y_2,1))=p^\omega_K(y_1,y_2)$,
\item $\tilde{p}^\omega((y_1,z_1),(y_2,0))=p^\omega(y_1,y_2)-p^\omega_K(y_1,y_2)$.
\end{enumerate}

\begin{remark}\label{rem_pkopen}
If $x$ is open and if $y\sim x$, then $\tilde{p}^\omega((x,z_1),(y,1))\ge \kappa$, where $\kappa>0$ is a constant depending only on $K$, $\ell$ and $d$.
\end{remark}

As before, we write $\widetilde{\PR}_{(x,z)}$ for the annealed law of the enhanced walk starting at $(x,z)$. Now, let us emphasize two facts
\begin{enumerate}
\item the evolution of the process starting at $(x,z)$ does not depend on $z$ (except $Z_0$),
\item it is easy to see that the marginal laws of the first coordinate of $(\widetilde{X}_n)$ match the laws $P^\omega_x$ and $\PR_x$ of the original biased random walk in random conductances.
\end{enumerate}

\begin{remark} We will often drop the $z$ in subscript and the tilde, simply writing $P^\omega_x$ and $\PR_x$ for $\widetilde{P}^\omega_{(x,z)}$ and $\widetilde{\PR}_{(x,z)}$ when the quantity observed does not depend on $z$,  voluntarily making the confusion with the original walk. For example, it is clear from the definition that the trajectories $(X_n)_{n\ge0}$ and $(Z_n)_{n\ge1}$, under $\widetilde{P}^\omega_{(x,z)}$ and $\widetilde{\PR}_{(x,z)}$, do not depend on $z$.
\end{remark}

From now on, when we write $X_n=x$ without specifying the second coordinate, we mean that $\widetilde{X}_n\in\{(x,0),(x,1)\}$.



\subsection{Definition of the regeneration times}\label{sect_defregen}


In this section, we will define the regeneration times in a slightly different way than the classical one in order to obtain  independence properties, see Theorem \ref{thindep}. In particular, we will use the \emph{enhanced random walk} defined in Section \ref{sect_altcons}.

Let us now introduce a variation on classical regeneration times where we ask for the regeneration point to be $K$-open and for a specific behaviour of the information encapsulated in the variables $Z_n$. One advantage of this construction is that it will allow us to obtain independence properties.

Let us now define the quantities we need.
First, define the random variable
\begin{eqnarray}
\mathcal{M}^{(K)}&:=&\inf\{i\geq2: X_{i} \mbox{ is $K$-open, } X_j\cdot\vec{\ell} < X_{i-2} \cdot\vec{
\ell}\text{ for any }j< i-2\nonumber
\\\label{def_Mcal}
&&\qquad \mbox{ and }X_i=X_{i-1}+e_1=X_{i-2}+2e_1
\}
\end{eqnarray}
%

Roughly, this corresponds to the time when we have reached a new maximum towards the direction $\vec{\ell}$, with two extra conditions. Let us comment these conditions.\\
%
%
%
%
%
%
%
%
%
%
Consider the time when the walker reaches a new maximum $x$, then the environment \emph{behind} him and the environment \emph{in front} of him have some common edges: we want the walker to jump twice in the direction $e_1$ (following \cite{Shen}) in order to have only a bounded number of common edges (namely those incident to $x$) and thus reduce correlations. Besides, we want $x$ to be open which will have two main advantages: this will firstly make it easier to escape to infinity and, secondly, this will enable us to   define regeneration times verifying independence properties.\\
We will state some results from~\cite{Fri11} in the Appendix about this random variable  where  the  same variable, exactly, is defined.\\

We define a random variable $D$ which is essentially the time it takes for the walk to go back beyond its starting point, with respect to the scalar product with $\lv$. Its definition is more complicated but we will explain the intuition below. In this particular definition, we will need the enhanced walk of Section \ref{sect_altcons}.  Define
\begin{eqnarray}
D&:=& \inf\left\{\left\{n>0:X_n\cdot \lv\le X_0\cdot\lv\right\} \cup\mathcal{I}_0\right.\label{defD}\\
&&\left. \cup\bigcup_{j=1}^{d}\Big\{ n>0: X_{n-1} =X_0+ e_j\text{ and }Z_{n}=0  \Big\}        \right\},\nonumber
\end{eqnarray}
where
\[
\mathcal{I}_0:=\left\{
    \begin{split}
   & \{1\} \text{ if } Z_1=0 ;\\ 
    & \emptyset \text{ otherwise.}
    \end{split}
  \right.
\]
We will be interested in the event $\{D=\infty\}$. The classic definition of $D$ is such that, on $\{D=\infty\}$, the walker never backtracks, i.e.~$X_n\cdot \lv> X_0\cdot\lv$ for all $n>0$. Here, we additionally impose that $Z_1=1$ and, if the walker is on a neighbour of $X_0$ at time $n-1$, then $Z_n=1$. This will again  reduce correlations, see Remark \ref{remindep}.
\begin{remark}\label{rem_measD}
Note that $D$ is measurable with respect to $\sigma\left(X_0,(\widetilde{X}_n)_{n\ge1}\right)$.
\end{remark}
\begin{remark}\label{remindep}
We will prove that $P^\omega_{x_0}[D=\infty]$ does not depend on the \emph{values} of the conductances of the edges adjacent to $x_0$, as long as $x_0$ is open. In fact, the whole future of the walk on the event $\{D=\infty\}$ does not depend on these conductances. See Proposition \ref{prop_idconf}.
\end{remark}

Also, we introduce the maximum (in the direction $\vec{\ell}$) of the
trajectory before $D$
\begin{equation}\label{defM}
M:=\sup_{n\leq D} X_n\cdot\vec{\ell}.
\end{equation}

We define the configuration dependent stopping times $S_k$, $k\geq0$
and the levels $M_k$, $k\geq0$,
%
\begin{eqnarray}
\label{defS} &\displaystyle S_0=0,\qquad M_0=X_0
\cdot\vec{\ell}\quad\mbox{and}&\nonumber\\[-8pt]\\[-8pt]
&\displaystyle \mbox{for $k\geq0$}\qquad
S_{k+1}:=\mathcal{M}^{(K)}\circ\theta_{T_{\mathcal
{H}^+(M_k)}}+T_{\mathcal{H}^+(M_k)},&\nonumber
\end{eqnarray}
where
%
\begin{equation}
\label{defrealM} M_k:=\sup\{X_m \cdot\vec{
\ell} \mbox{ with }0\leq m \leq R_k\}
\end{equation}
with
\[
R_{k} := D\circ\theta_{S_k}+S_k.
\]

These definitions imply that if $S_{i+1}<\infty$, then
%
\begin{equation}
\label{rightdir} X_{S_{i+1}}\cdot\vec{\ell} -
X_{S_i} \cdot\vec\ell\geq2 e_1\cdot\vec{\ell} \geq
\frac{2}{\sqrt d}.
\end{equation}

Finally we define the basic regeneration time
%
\begin{equation}
\label{deftau1} \tau_1^{(K)}=S_N\qquad \mbox{with } N:=\inf\{k \geq1\mbox{ with } S_k<\infty\mbox{ and }
M_k=\infty\}.
\end{equation}


\begin{remark}\label{remKdeptau}
As $\tau_1^{(K)}$ depends on $\mathcal{M}^{(K)}$, it also depends on the value of the constant $K$. However, to lighten notations we will drop the dependence in $K$, since this constant will be fixed later at a certain large value.
\end{remark}

\subsection{Uniformly bounded chance of never backtracking at open points}

The next result is natural. Starting from a good vertex and by following a directed open path from there, we
can bring the random walk far in the direction of the bias with a
positive probability, uniformly in the environment, and after this
point it will be unlikely by Lemma \ref{thisisboring} to backtrack
past your starting point. This means that there is always a positive
escape probability from a good point:  we will then be allowed to study events conditioned on $\{D=\infty\}$.

\begin{lemma}
\label{posescape}
Recall the Definition \ref{defopenvertex} of a $K$-good vertex. There exists $K_0<\infty$ such that, for any $K\ge K_0$, we have
\[
{\mathbf E} \Bigl[  P^{\omega
}_0[D<\infty] \mid0\text{ is }K\text{-good} \Bigr] <1-c(K).
\]
\end{lemma}

\begin{proof}
First, by Lemma \ref{BLsizeclosedbox}, ${\mathbf P}[\mbox{0 is good}]>0$ as soon as $K$ is large enough, hence the conditioning is properly defined.\\
Fix $n>0$. On the event that $\{0\mbox{ is good}\}$, we denote
$\mathcal
{P}(i)$ a directed path starting at $0$ where all points are open (including $0$). By definition, $\mathcal{P}(0)=0$, $\mathcal{P}(1)=e_1$, and denote $j_2\in\{1,...,d\}$ the integer such that $\mathcal{P}(2)=e_1+e_{j_2}$. We
denote $L_{\partial^+B(n,n^2)}=\inf\{ i, \mathcal{P}(i)\in
\partial^+B(n,n^2)\}$. 
Recall the definition of the enhanced walk $(X_n,Z_n)$ from Section \ref{sect_altcons}. Define the event $A=\cap_{i=1}^4A_i$, where
\begin{eqnarray*}
A_1&=&\left\{(X_1,Z_1)=(e_1,1)\right\};\\
A_2&=&\left\{(X_2,Z_2)=(e_1+e_{j_2},1)\right\};\\
A_3&=&\left\{X_i=\mathcal{P}(i)\text{ for }3\le i\leq L_{\partial
^+B(n,n^2)}\right\};\\
A_4&=&\left\{T_{\mathcal{H}^-(2)}\circ\theta_{T_{\mathcal{P}(L_{\partial^+B(n,n^2)})}} =\infty\right\}.
\end{eqnarray*}
Then, we have $A\subset\{D=\infty\}$, as an immediate consequence of the definitions.

As $\{0\mbox{ is good}\}$, then $L_{\partial
^+B(n,n^2)}\leq Cn$, and any vertex $y$ on the trajectory $\mathcal{P}(i)$ is open, hence ${p}^\omega_K(y,y+e_j)=\left( K^{-1}e^{e_j\cdot\ell}\right)\left(\sum_{i=1}^{2d}  Ke^{e_i\cdot\ell}\right)^{-1}$, for any $j\in\{1,...,d\}$, and where $p^\omega_K$ is defined in \eqref{defpjK}. Therefore, we have
\begin{eqnarray*}
P^\omega_0\bigl[A_1\cap A_2\cap A_3\bigr]
\geq \left[ \left(\min_{j\in\{1,...,d\}}e^{e_j\cdot\ell}\right)\left(\sum_{i=1}^{2d}  K^2e^{e_i\cdot\ell}\right)^{-1}\right]^{Cn}=c^{n},
\end{eqnarray*}
for some constant $c>0$ that only depends on $K$, $d$ and $\ell$.
In particular, we have
\begin{eqnarray*}
&& {\mathbf E} \Bigl[ P^\omega
_0[D=\infty]\mid0
\mbox{ is good} \Bigr]
\\
&&\qquad\geq  {\mathbf E} \Bigl[P^\omega_0\bigl[A_1\cap A_2\cap A_3\bigr]
\times P^{\omega}_{\mathcal{P}(L_{\partial^+B(n,n^2)})}[
T_{\mathcal{H}^-(2)} =\infty]\mid0\mbox{
is good} \Bigr]
\\
&&\qquad\geq  c^{n} {\mathbf E} \Bigl[  P^{\omega}_{\mathcal{P}(L_{\partial
^+B(n,n^2)})}[ T_{\mathcal
{H}^-(2)} =
\infty]\mid0\mbox{ is good} \Bigr].
\end{eqnarray*}
%

Besides, we have
\begin{eqnarray*}
&& {\mathbf E} \bigl[ P^{\omega}_{\mathcal{P}(L_{\partial^+B(n,n^2)})}[ T_{\mathcal{H}^-(2)} <\infty]
\mid\mbox{0 is good} \bigr]
\\
&&\qquad\leq {\mathbf P}[\mbox{0 is good}]^{-1} {\mathbf E} \bigl[
P^{\omega
}_{\mathcal
{P}(L_{\partial^+B(n,n^2)})}[ T_{\mathcal{H}^-(2)} <\infty]
\bigr]
\\
&&\qquad\leq {\mathbf P}[\mbox{0 is good}]^{-1} {\mathbf E} \bigl[
\max_{x\in\partial^+B(n,n^2)}P^{\omega
}_x[ T_{\mathcal{H}^-(2)} <\infty]
\bigr]\\
&&\qquad\leq  Cn^{c(d)} \PR[T_{\mathcal{H}^-(-n+2)} <\infty]\leq  Cn^{c(d)} \exp(-cn),
\end{eqnarray*}
where we use translation invariance, the fact that ${\mathbf P}[\mbox{0 is good}]>0$ (see Lemma \ref{BLsizeclosedbox}) and Lemma \ref{thisisboring}.

We see that the previous quantity
is less than $1/2$ for $n\geq n_0$, for some $n_0$ depending on $K$ and $d$. Hence combining the last two equations,
\[
{\mathbf E} \Bigl[ P^\omega
_0[D=\infty]\mid0
\mbox{ is good} \Bigr] \geq(1/2) c^{n_0} >0,
\]
which implies the result.
\end{proof}

\subsection{Tails of regeneration times}

Here, we state the following theorem on the tails of regeneration times. We postpone its proof to the Appendix because it is extremely similar to the proof of Theorem 4.1 in \cite{Fri11}.\\
\begin{theorem}
\label{tailtau}
For any $M\in(0,+\infty)$, there exists $K_0<\infty$ such that, for any $K\geq K_0$
we have $\tau_1^{(K)}< \infty$ $\PR$-a.s.~and
\[
\PR_0[X_{\tau_1^{(K)}}\cdot\vec{\ell} \geq n] \leq C(M)n^{-M}.
\]
\end{theorem}

\subsection{Fundamental property of regeneration times}\label{fundprop}

We are going to define the sequence $\tau_0:=0<\tau_1<\tau_2< \cdots
<\tau_k< \cdots$ of successive regeneration times. Moreover, we state and prove Theorem \ref{thindep} which is the key result about the independence of regeneration blocks.\\

Using a slight abuse
of notation by viewing $\tau_k(\cdot,\cdot)$ as a function of a walk
and an environment, we can define the sequence of successive regeneration times via the following procedure:
%
\begin{equation}
\label{regenstruct} \tau_{k+1}=\tau_1+
\tau_k\bigl((X_{\tau_1+ \cdot}-X_{\tau_1},Z_{\tau_1+\cdot}), \omega( \cdot
+X_{\tau_1})\bigr),\qquad k\geq0,
\end{equation}
meaning that the $(k+1)$-th regeneration time is the $k$-th regeneration
time after the first one. We will denote by $\mathcal{F}_n$ the canonical filtration of the enhanced walk $\widetilde{X}=(X,Z)$, see Section \ref{sect_altcons}.

We set
\begin{align} \nonumber
\mathcal{E}_x &:=\left\{[x,x+e_j], \ j\in\{1,...,2d\}\right\},\\
\label{defleft} \mathcal{L}^x &:=\bigl\{[y,z]\in E
\bigl(\Z^d\bigr), y\cdot\ell\leq x \cdot\ell\mbox{ and } z\cdot\ell
\leq x \cdot\ell\bigr\}\cup\mathcal{E}_x\\
\label{defright} \mathcal{R}^x &:=\bigl\{[y,z]\in E
\bigl(\Z^d\bigr), y\cdot\ell> x \cdot\ell\mbox{ or } z\cdot\ell> x
\cdot\ell\bigr\}\cup\mathcal{E}_x\\ \nonumber
\mathcal{G}_{k} &:=\sigma\bigl\{ \tau_1,\ldots,
\tau_k; (\widetilde{X}_{\tau_k\wedge
m})_{m\geq0}; c_*(e) \mbox{ with } e
\in\mathcal{L}^{X_{\tau
_{k}}}\bigr\}.
\end{align}
We will denote $t_x$ the canonical shift on $\mathbb{Z}^d$.
For $a\in[1/K,K]^{\mathcal
{E}_0}$, we set
\[
{\mathbf P}_x^a =\delta_a \bigl(\bigl(c_*
\bigl(e\bigr)\bigr)_{e\in\mathcal{E}_x} \bigr)
\otimes\int_{e\in E(\Z^d)\setminus\mathcal{E}_x} \otimes\, d {\mathbf P}\bigl
(c_*(e)\bigr),
\]
where $\otimes$ denotes the product of measures. We introduce the
associated annealed measure
\[
\PR_x^a ={\mathbf P}_x^a \times
P_x^{\omega}.
\]

In words, $\PR_x^a$ denotes the annealed measure for the walk started
at $x$ but where the conductances of the edges in $\mathcal{E}_x$ are
fixed and given by $a$. We will use the notation $\PR^a$ (resp., ${\mathbf
P}^a$) for $\PR_0^a$ (resp., ${\mathbf P}_0^a$). 
\begin{definition}\label{defP0K}
We denote $\PR_x^K$ and ${\mathbf P}_x^K$ the annealed law of the walk started at $x$ and the law of the environment, respectively, when the configuration at $x$ is fixed such that $c_*(e)=K$ for any $e\in\mathcal{E}_x$. We also use the heavier notation $\PR^{x,K}_y$ when the configuration at $x$ is fixed such that $c_*(e)=K$ for any $e\in\mathcal{E}_x$, and the walk starts at $y$.
\end{definition}

As a particular case of a result in \cite{Fri11}, the following result holds.
\begin{theorem}[Theorem 7.3 of \cite{Fri11}]
\label{BLK}
For $\alpha>d+3$
\[
\PR_0^K[T_{\partial B(L,L^{\alpha})}\neq T_{ \partial^+ B(L,L^{\alpha})}]
\leq Ce^{-cL}.
\]
\end{theorem}

We can also state the following variant of Theorem \ref{tailtau}, whose proof is postpone in the Appendix.
\begin{theorem}
\label{tailtauK}
For any $M\in(0,+\infty)$, there exists $K_0<\infty$ such that, for any $K\geq K_0$
we have $\tau_1< \infty$ $\PR$-a.s.~and
\[
\PR_0^K[X_{\tau_1}\cdot\vec{\ell} \geq n] \leq C(M)n^{-M}.
\]
\end{theorem}

%
%
%

The fundamental properties of regeneration times are that:
\begin{itemize}
\item[(1)] the past and the future of the random walk that has arrived at
$X_{\tau_k}$ are independent;
\item[(2)] the law of the future of the random walk has the same law as
a random walk under $\PR_0^{K}[ \cdot\mid D=\infty]$.
\end{itemize}

Let us first state and prove the following result which is important for the independence of regeneration blocks: with our non-classical definition of $D$, the environment at $0$ is irrelevant for the evolution of the walk, on the event $\{D=\infty\}$ and when $0$ is open.

\begin{proposition}\label{prop_idconf} Fix a vertex $x_0\in\Z^d$ and fix an environment $\omega$ such that  $x_0$ is $K$-open. Define the environment $\omega_K$ such that $c^{\omega_K}_*(g)=c^{\omega}_*(g)$ if $g\notin\mathcal{E}_{x_0}$, and  $c^{\omega_K}_*(g)=K$ if $g\in\mathcal{E}_{x_0}$.\\
Then, for any bounded and $\sigma((X_n,Z_n),n\ge0)$-measurable function $f$ and for any $z_0\in\{0,1\}$, we have
\[
E_{(x_0,z_0)}^\omega\left[f(X_\cdot,Z_\cdot)\1{D=\infty}\right]=E_{(x_0,z_0)}^{\omega_K}\left[f(X_\cdot,Z_\cdot)\1{D=\infty}\right].
\]
\end{proposition}

\begin{proof}
The proof essentially comes from the definition \eqref{defD} of $D$ and the construction of the enhanced walk $\widetilde{X}=(X,Z)$ in Section \ref{sect_altcons}. Let us prove the result by a coupling argument.

Recall the construction of Section \ref{sect_altcons}.
Given the environment $\omega$, let us define a new probability $P_{(x_0,z_0)}^{\omega,\omega_K}$ such that the process $(\widetilde{X}^{(1)}_\cdot,\widetilde{X}^{(2)}_\cdot)$ under $P_{(x_0,z_0)}^{\omega,\omega_K}$ is such that the marginal law of $\widetilde{X}^{(1)}_\cdot$ (resp. $\widetilde{X}^{(2)}_\cdot$) matches the law of $\widetilde{X}_\cdot$ under $P_{(x_0,z_0)}^{\omega}$ (resp. $P_{(x_0,z_0)}^{\omega_K}$).\\
Notice that, from \eqref{defD}, we obtain that $D>0$ almost surely and, for any integer $N>0$,
\begin{eqnarray*}
\{D>N\}&=& \bigcap_{n=1}^N\left\{X_n\cdot \lv> X_0\cdot\lv\right\} \cap \left\{Z_1=1\right\}\\
&& \cap\bigcap_{j=1}^{d}\left\{ \text{for all }0<n\le N\text{ s.t. } X_{n-1} =X_0+ e_j: Z_{n}=1  \right\}.
\end{eqnarray*}
Moreover, we naturally define the quantities $D^{(1)}$ and $D^{(2)}$ respectively associated to $\widetilde{X}^{(1)}_\cdot$ and $\widetilde{X}^{(2)}_\cdot$.\\
Now, let us define the law of $(\widetilde{X}^{(1)}_\cdot,\widetilde{X}^{(2)}_\cdot)$  under $P_{(x_0,z_0)}^{\omega,\omega_K}$. The important point of the coupling is that if $D^{(1)}\vee D^{(2)}=\infty$, then the two walks remain coupled for ever (and in particular $D^{(1)}= D^{(2)}=\infty$). 

To do this, we couple the walks in the following manner
\begin{enumerate} 
\item if at time $n$ the two walks are still coupled and if the trajectories are still compatible with $D^{(1)}= D^{(2)}=\infty$, then we let $\widetilde{X}^{(1)}$ make a step according to $P^{\omega}$,
\item  if this step is again compatible with $D^{(1)}=\infty$, then $\widetilde{X}^{(2)}$ takes the same step,
\item otherwise, it means that $D^{(1)}=n+1$ and we impose $\widetilde{X}^{(2)}$ to move such that $D^{(2)}=n+1$ (and the walks are considered as decoupled).
\end{enumerate}

Let us do this rigourously. First, recall the definition \eqref{defpjK} of $p^\omega_K$ and let us notice that, as $\omega$ and $\omega_K$ coincide everywhere except on $\mathcal{E}_{x_0}$ and because $x_0$ is open in $\omega$, we have that $p^{\omega}_K(x,\cdot)=p^{\omega_K}_K(x,\cdot)$ for any $x\in\Z^d$, and $p^{\omega}(x,\cdot)=p^{\omega_K}(x,\cdot)$ as soon as $x\notin\{{x_0}\}\cup\{{x_0}+e_j:j\in\{1,...,2d\}\}$.\\
We fix $(\widetilde{X}^{(1)}_0,\widetilde{X}^{(2)}_0)=((x_0,z_0),(x_0,z_0))$, $P_{(x_0,z_0)}^{\omega,\omega_K}$-almost surely. We define the process by induction. Given the trajectory up to time $n\ge0$, the conditional law of $(\widetilde{X}^{(1)}_{n+1},\widetilde{X}^{(2)}_{n+1})$ is given by  the following rules:
\begin{enumerate}
\item if  $\{D^{(1)}>n\}\cap\{D^{(2)}>n\}$ holds, and if $X^{(1)}_n=X^{(2)}_n=x$ for some $x\in \Z^d$, then let $\widetilde{X}^{(1)}$ make a step according to $P^{\omega}$, and
\begin{enumerate}
\item if $x\neq{x_0}$ and $x\notin\{{x_0}+e_j:j\in\{1,...,2d\}\}$, then $\widetilde{X}^{(2)}_{n+1}=\widetilde{X}^{(1)}_{n+1}$;
\item if $x={x_0}$ or $x\in\{{x_0}+e_j:j\in\{1,...,2d\}\}$, if $Z^{(1)}_{n+1}=1$ and regardless of $X^{(1)}_{n+1}$, then $\widetilde{X}^{(2)}_{n+1}=\widetilde{X}^{(1)}_{n+1}$;
\item if $x={x_0}$ or $x\in\{{x_0}+e_j:j\in\{1,...,2d\}\}$, if $Z^{(1)}_{n+1}=0$ and regardless of  $X^{(1)}_{n+1}$, then, for any $j\in\{1,...,2d\}$, $\widetilde{X}^{(2)}_{n+1}=(x+e_j,0)$ with probability
\[
\frac{p^{\omega_K}(x,x+e_j)-p^{\omega_K}_K(x,x+e_j)}{\displaystyle{1-\sum_{1\le i\le 2d} p^{\omega}_K(x,x+e_i)}}.
\]
In particular we have $D^{(1)}=D^{(2)}=n+1$ in this case;
\end{enumerate}

\item if  $\{D^{(1)}\le n\}\cup\{D^{(2)}\le n\}$ holds, then $\widetilde{X}^{(1)}$ and $\widetilde{X}^{(2)}$ move independently according to $P^{\omega}$ and $P^{\omega_K}$ respectively.

\end{enumerate}
In order to end this construction properly, note that, if $\{D^{(1)}\wedge D^{(2)}>n\}\cap\{X^{(1)}_n=X^{(2)}_n\}$ holds, then either $\{D^{(1)}\wedge D^{(2)}>n+1\}\cap\{X^{(1)}_{n+1}=X^{(2)}_{n+1}\}$ or $\{D^{(1)}=D^{(2)}=n+1\}$ holds. Then, as $\{D^{(1)}\wedge D^{(2)}>0\}\cap\{X^{(1)}_0=X^{(2)}_0\}$ holds $P^{\omega,\omega_K}_{(x_0,z_0)}$-a.s., this implies by a simple induction:
\begin{enumerate}
\item the event $\{D^{(1)}\wedge D^{(2)}>n\}\cap\{X^{(1)}_n\neq X^{(2)}_n\}$ never occurs $P^{\omega,\omega_K}_{(x_0,z_0)}$-a.s., therefore the construction is complete;
\item as long as $\{D^{(1)}>n\}\cap\{D^{(2)}>n\}$ holds, we have $X^{(1)}_k=X^{(2)}_k$ for any $0\le k\le n$;
\item  $D^{(1)}=D^{(2)}$ $P_{(x_0,z_0)}^{\omega,\omega_K}$-almost surely.
\end{enumerate}
 Using this, we have 
\begin{align*}
&P_{(x_0,z_0)}^{\omega,\omega_K}\left[\{D^{(1)}=\infty\}\cap\left\{\{\exists n\ge0: \widetilde{X}^{(1)}_n\neq \widetilde{X}^{(2)}_n\}\cup\{D^{(2)}<\infty\}\right\}\right]\\
\le&P_{(x_0,z_0)}^{\omega,\omega_K}\left[D^{(1)}\neq D^{(2)}\right]+\sum_{n\ge0}P_{(x_0,z_0)}^{\omega,\omega_K}\left[D^{(1)}\wedge D^{(2)}>n+1,\widetilde{X}^{(1)}_{n+1}\neq \widetilde{X}^{(2)}_{n+1}\right]\\
=&0.
\end{align*}


Besides, recalling that $p^{\omega}_K(x,\cdot)=p^{\omega_K}_K(x,\cdot)$ for any $x\in\Z^d$, and $p^{\omega}(x,\cdot)=p^{\omega_K}(x,\cdot)$ as soon as $x\notin\{{x_0}\}\cup\{{x_0}+e_j:j\in\{1,...,2d\}\}$, it is easy to check that the law of $\widetilde{X}^{(1)}_\cdot$ (resp. $\widetilde{X}^{(2)}_\cdot$) under $P_{(x_0,z_0)}^{\omega,\omega_K}$ is the law of $\widetilde{X}_\cdot$ under $P_{(x_0,z_0)}^{\omega}$ (resp. $P_{(x_0,z_0)}^{\omega_K}$).\\
Finally, for  any integer $n\ge0$ and any set $A=A_0\times...\times A_n$ with $A_0,...,A_n$ in the $\sigma$-algebra generated by the subsets of $\Z^d\times\{0,1\}$, and on the event $\{D=\infty\}$, we have that
\begin{align*}
 &E_{(x_0,z_0)}^{\omega}\left[\1{(\widetilde{X}_0,...,\widetilde{X}_n)\in A}\1{D=\infty}\right]\\\
=&E_{(x_0,z_0)}^{\omega,\omega_K}\left[\1{(\widetilde{X}^{(1)}_0,...,\widetilde{X}^{(1)}_n)\in A}\1{D^{(1)}=\infty}\right]\\
=&E_{(x_0,z_0)}^{\omega,\omega_K}\left[\1{(\widetilde{X}^{(2)}_0,...,\widetilde{X}^{(2)}_n)\in A}\1{D^{(2)}=\infty}\right]\\
=&E_{(x_0,z_0)}^{\omega_K}\left[\1{(\widetilde{X}_0,...,\widetilde{X}_n)\in A}\1{D=\infty}\right],
\end{align*}
and we conclude the proof using the Monotone Class Theorem.

%
%
%
\end{proof}
\begin{remark}\label{suissenergie}
The last result implies, for example, that $\PR_0^K[D=\infty]=\PR_0[D=\infty|0\text{ is open}]$.
\end{remark}

Let us now prove the fundamental property of regeneration times which provides  an i.i.d.~structure on the trajectory of the walk.

\begin{theorem}
\label{thindep}
Let us fix $K$ large enough, and $k\ge 1$. First, for any $k\geq1$, we have $\tau
_k^{(K)}<\infty$ $\PR$-a.s. (or $\PR^a$-a.s. for any $a\in
[1/K,K]^{\mathcal{E}}$).\\
Second, let $f$, $g$, $h_k$ be bounded functions
which are measurable with respect to $\sigma\{X_0,\widetilde{X}_n:  n\geq1\}$, $\sigma
\{
c_*(e),e\in\mathcal{R}^0\setminus \mathcal{E}_0\}$ and $\mathcal{G}_k$, respectively. Then,
we have
\[
\ES_{(0,z_0)}\bigl[f(X_{\tau_k+\cdot}-X_{\tau_k},Z_{\tau_k+\cdot})g\circ
t_{X_{\tau_k}}h_k\bigr]=\ES_{(0,z_0)}[h_k]\times 
\ES_0^{K}[f(\widetilde{X}_\cdot)g\mid D=\infty].
\]
\end{theorem}

\begin{remark}
According to Remark \ref{rem_measD}, the event $\{D=\infty\}$ is $\sigma\{X_0, \widetilde{X}_n:  n\geq1\}$-measurable but not $\sigma\{{X}_n:  n\geq0\}$-measurable, this is why we need to deal with the enhanced walk in Theorem \ref{thindep}. Note that the function that we consider for the future of the walk should not depend on $Z_0$ in order to have the independence between the future and the past.
\end{remark}

\begin{proof}
The proof follows the blueprint of \cite{SznZer}, precisely Proposition 1.3 and Theorem 1.4. We only need to  adapt it slightly in order to have, in our case, the independence of the future and the past of the walk at regeneration times. We will not give details on the part that are identical (see also \cite{Shen} for more detailed arguments). We will prove the theorem for $k=1$ and the conclusion will then follow by induction. As we will need it to make the induction step, we also need to prove that,
for any $a\in[1/K,K]^{\mathcal{E}_0}$ and $z_0\in\{0,1\}$,
\[
\ES_{(0,z_0)}^a\bigl[f(X_{\tau_k+\cdot}-X_{\tau_k},Z_{\tau_k+\cdot})g\circ
t_{X_{\tau_k}}h_k\bigr]=\ES_{(0,z_0)}^a[h_k]\times 
\ES_0^{K}[f(\widetilde{X}_\cdot)g\mid D=\infty].
\]
We give the argument for this case but the proof is the same for $\ES$. First, let us point out that $\mathcal{G}_1$ is generated by the sets $\{\tau_1=k\}\cap\{\widetilde{X}_{\tau_1}=(x,z)\}\cap A$, with $A\in\sigma\left(c_*(e),\ e\in\mathcal{L}^{X_{\tau_{1}}}\right)\otimes{\mathcal{F}_{\infty}}$, and $(x,z)\in\Z^d\times\{0,1\}$. Besides, recall that $\tau_1<\infty$ a.s.~by Theorem \ref{tailtau}.\\
Now, recalling the construction of the enhanced random walk $\widetilde{X}$ in Section \ref{sect_altcons}, we have
\begin{eqnarray*}
&&\ES_{(0,z_0)}^a\bigl[f(X_{\tau_1+\cdot}-X_{\tau_1},Z_{\tau_1+\cdot})g\circ t_{X_{\tau_1}}h_1\bigr]\\
&&= \sum_{k\ge0} \ES_{(0,z_0)}^a\bigl[f(X_{\tau_1+\cdot}-X_{\tau_1},Z_{\tau_1+\cdot})g\circ t_{X_{\tau_1}}h_1, S_k<\infty, R_k=\infty\bigr]\\
&&= \sum_{k\ge0,(x,z)} \ES_{(0,z_0)}^a\bigl[f(X_{S_k+\cdot}-x,Z_{S_k+\cdot})g\circ t_{x}h_1, S_k<\infty, \widetilde{X}_{S_k}=(x,z), R_k=\infty\bigr].
\end{eqnarray*}

Following the arguments of \cite{SznZer}, there exists a random variable $h_1^{k,(x,z)}$, measurable with respect to $\sigma\left(c_*(e),\ e\in\mathcal{L}^{x}\right)\otimes{\mathcal{F}_{S_k}}$, which coincides with $h_1$ on the event $\{\widetilde{X}_{\tau_1}=(x,z)\}\cap\{\tau_1=S_k\}$. Therefore,
\begin{eqnarray*}
&&\ES_{(0,z_0)}^a\bigl[f(X_{\tau_1+\cdot}-X_{\tau_1},Z_{\tau_1+\cdot})g\circ t_{X_{\tau_1}}h_1\bigr]\\
&&=\sum_{k\ge0,(x,z)} \ES_{(0,z_0)}^a\left[E_0^\omega\left[f(X_{S_k+\cdot}-x,Z_{S_k+\cdot})h_1^{k,(x,z)}, S_k<\infty,\right.\right.\\
&&\qquad\qquad\qquad\qquad \left. \left.\widetilde{X}_{S_k}=(x,z), D\circ \theta_{S_k} =\infty\right]g\circ t_{x}\right]\\
&&=\sum_{k\ge0,(x,z)} \ES_{(0,z_0)}^a\left[E_0^\omega\left[h_1^{k,(x,z)}, S_k<\infty, \widetilde{X}_{S_k}=(x,z)\right]  \right.\\
&&    \qquad\qquad\qquad\qquad  \left. \times E_{(x,z)}^\omega\left[    f({X}_{\cdot}-x,Z_\cdot), D =\infty   \right]g\circ t_{x}\right]
\end{eqnarray*}
where we used the strong Markov property at time $S_k$.\\
Let $E^K_{(x,z)}$ denote the quenched law in the environment $\omega^x_K$ which is identical to $\omega$ everywhere except that we fix the conductances $c^{\omega^x_K}_*([x,x+e_j])=K$ for all $j\in\{1,...,2d\}$.\\
Now, notice that if $X_{S_k}=x$ then $x$ is open and recall that $f$ is $\sigma\{X_0, \widetilde{X}_n:  n\geq1\}$-measurable. Hence, by Proposition \ref{prop_idconf}, $E_{(x,z)}^\omega\left[    f(X_{\cdot}-x,Z_\cdot), D =\infty   \right]=E_{x}^{K}\left[    f(X_{\cdot}-x,Z_\cdot), D =\infty   \right]$, which does not depend on $z$. Moreover, this last quantity is bounded and measurable w.r.t.~$\sigma\left(c_*(e),\ e\in\mathcal{R}^{x}\setminus \mathcal{E}_x\right)$, hence it is $\PR$-independent of $E_0^\omega\left[h_1^{k,(x,z)}, S_k<\infty, \widetilde{X}_{S_k}=(x,z)\right]$ which is bounded and $\sigma\left(c_*(e),\ e\in\mathcal{L}^{x}\right)$-measurable. Hence, we have
\begin{align}
&\ES_{(0,z_0)}^a\bigl[f(X_{\tau_1+\cdot}-X_{\tau_1})g\circ t_{X_{\tau_1}}h_1\bigr]\nonumber\\
 &=\sum_{k\ge0,(x,z)} \ES_0^a\left[h_1^{k,(x,z)}, S_k<\infty, \widetilde{X}_{S_k}=(x,z)\right] \ES_{0}^K\left[    fg, D =\infty   \right]\label{eqf1g1}\\
&=\ES_0^K\left[   \left. fg\right|D =\infty   \right]\sum_{k\ge0,(x,z)} \ES_{(0,z_0)}^a\left[h_1^{k,(x,z)}, S_k<\infty, \widetilde{X}_{S_k}=(x,z)\right] \PR_{0}^K[D=\infty] \nonumber\\
&=\ES_0^K\left[   \left. fg\right|D =\infty   \right] \ES_{(0,z_0)}^a\left[h_1\right], \nonumber
\end{align}
where we used \eqref{eqf1g1} in the case $f=1$ and $g=1$ in order to obtain the last equality. This concludes the proof the theorem in the case $k=1$.\\

Now, using this result, \eqref{regenstruct} and Theorem \ref{tailtau}, we have, by induction, that $\tau_k<\infty$ almost surely, for any $k$.\\
We conclude the proof and obtain the result for general $k$ by induction, closely following the arguments of the proof Theorem 1.4 of \cite{SznZer}, or the more detailed proof of Theorem 3.5 of \cite{Shen}. The only modification to make to this last one is to define $\bar{\mathcal{G}_{k}}$ as
\[
\bar{\mathcal{G}_{k}}:=\sigma\bigl\{ \tau_1,\ldots,
\tau_k; (\widetilde{X}_{\tau_k\wedge
m})_{m\geq0}; c_*(e) \mbox{ with } e
\in \left(\mathcal{R}_0\setminus \mathcal{E}_0\right)\cap\mathcal{L}^{X_{\tau
_{k}}}\bigr\},
\]
and to the turn, at the very end of the proof,  $E_0^{a_{X_{\tau_{k+1}}}}$ into $E_0^{K}$, using Proposition \ref{prop_idconf}.
\end{proof}


\section{The time spent outside abnormally large edges is negligible}\label{pfff17}

In this section, the goal is to prove that the time spent, during one regeneration period, on edges with a conductance that is not too large is asymptotically negligible. This result is given by Lemma \ref{decaytau}. We need first to have estimates on the size of the regeneration blocks.\\

Recall that the regeneration times $(\tau_i)$ depend on the constant $K$, see Remark \ref{remKdeptau}. Also, recall that, for any $x\in\Z^d$, $T_x$ is the hitting of the vertex $x$. Fix some constant $\alpha>d+3$ and define
\begin{equation}\label{def_chi}
\chi:= \chi^{(K)}=\inf\{m\in\N: \{X_i, i\in[0,\tau_1]\} \subset B(m,m^{\alpha})\}.
\end{equation}

\begin{lemma}\label{tail_chi_tau}
For any $M\in(0,+\infty)$, there exists $K_0<\infty$ such
that, for any $K\geq K_0$,
\[
\PR_0[\chi^{(K)} \geq k]\leq Ck^{-M}.
\]

This implies that for any $M<\infty$, there exists $K_0<\infty$ such
that, for any $K\geq K_0$ and for any $x\in\Z^d$,
\[
\PR_0[T_{x} \leq \tau_1]  \leq  C \abs{\abs{x}}_{\infty}^{-M/\alpha}.
\]
The same results hold for $\PR_0^K$.
\end{lemma}

\begin{proof}
We can follow line by line the proof of Lemma 8.7 of \cite{Fri11} except that $\max_{a\in[1/K,K]^{\mathcal{E}}}\PR^a$ has to be replaced by $\PR_0$ (resp.~$\PR_0^K$) and, instead of using Theorems 7.2 and 7.3 from \cite{Fri11}, we need to use the analog Theorems \ref{tailtau} and \ref{BL} (resp.~\ref{tailtauK} and \ref{BLK}) from this paper.
\end{proof}

When observing the random variables $(\tau_i)$, we want to distinguish the time spent on abnormally large edges (traps) and the time spent on the other edges which will be negligible.\\
For this purpose, recall the definitions \eqref{def_Et} of $E_{< t}$ and \eqref{defright} of $\mathcal{R}$, and let us define,
\begin{eqnarray}
\tau^{\ge t}_1&:= &\sum_{e\in E_{< t}^c\cap \left\{\mathcal{R}^0\setminus\mathcal{E}_{0}\right\}} \left| \left\{k\in[1,\tau_1]:[X_{k-1},X_k]=e\right\}\right|,\label{def_tausup}
\end{eqnarray}
as well as
\begin{eqnarray}
\tau^{< t}_1&:=&\tau_1-\tau^{\ge t}_1.\label{def_tauinf}
\end{eqnarray}
\begin{remark}\label{rem_regensup}
We are careful about the definition of $\tau^{\ge t}_1$ in order to make sure that this quantity does not depend on the conductances outside $\mathcal{R}^0\setminus\mathcal{E}_0$, so that we can later apply Theorem \ref{thindep}. Note that, under $\PR_0^K[\cdot|D=\infty]$, one edge of $\mathcal{E}_0$ is crossed once and the other edges of $\left\{\mathcal{R}^{0}\setminus\mathcal{E}_{0}\right\}^c$ are not crossed at all. Moreover, as soon as $t\ge K$, $\mathcal{E}_{0}\subset E_{< t}$.
\end{remark}
As for the regeneration times, we can define, for $k\ge0$:
\begin{eqnarray}
\tau_{k+1}^*=\tau_1^*+
\tau_k^*\bigl((X_{\tau_1+ \cdot}-X_{\tau_1},Z_{\tau_1+\cdot}), \omega( \cdot
+X_{\tau_1})\bigr),\label{def_tausupn}
\end{eqnarray}
where $*$ stands for $<$ or $\ge$.\\

We now give an upper-bound on the $\PR^K[\cdot|D=\infty]$-probability that $\tau^{< t}_1$ is large, when $t$ is large.

%
%
%
%

\begin{lemma}\label{decaytau}
For any
$\delta\in(0,1)$, there exists $K_0<\infty$ such that, for any $K\geq K_0$ and for any constant $a>0$,
\[
\PR^K_0\bigl[{
\tau}_1^{< n^{\delta}}>an\mid D=\infty\bigr]\leq C(K,\delta,a)
n^{-\gamma-\frac{(1-\delta)(1-\gamma)}{2}}.
\]
\end{lemma}

\begin{proof}

Let us introduce
\begin{eqnarray*}
\widetilde{\tau}_1^{< n^\delta}:= &\sum_{e\in E_{\le n^\delta}\cap \left\{\mathcal{R}^{0}\setminus\mathcal{E}_{0}\right\}} &\left| \left\{k\in[
T_{\mathrm{GOOD}},T_{\mathrm{GOOD}}\circ\tau_{1})\right.\right.\\
&&\left.\left.\text{ such that }[X_{k},X_{k+1}]=e\right\}\right|.
\end{eqnarray*}
Define $T_\tau:=T^+_{      \mathrm{GOOD}\cup\{0\}, E_{< n^\delta}     }+\widetilde{\tau}_1^{< n^\delta}$. It is clear that, under $\PR_0^K[\cdot|D=\infty]$, $0$ is open and $\tau_1^{< n^\delta}\le T_\tau$. Therefore, to prove the lemma it will be enough to prove that
\[
\PR_0^K\bigl[T_\tau>an\bigr]\leq C(K)
n^{-\gamma-\frac{(1-\delta)(1-\gamma)}{2}}.
\]

In this proof, we will point out the $K$ dependence of constants,
since the proof requires us to be careful with this dependence. On the other hand, we drop the dependence on $\delta$ and $a$ since these quantities will be fixed throughout the proof. Also, we fix $n>3$ for convenience and note that the statement is obvious in the other case.

Fix $\delta>0$ and $a>0$. Let us work in an environment $\omega$ which is such that $0$ is open.
We have
\begin{eqnarray*}
T_\tau &\leq&T^+_{      \mathrm{GOOD}\cup\{0\}, E_{\le n^\delta}     }+\sum_{x\in\mathrm{GOOD}(\omega)} {
\mathbf1} {\{T_x<\tau_1\}} \sum_{i=0}^\infty\1{X_i=x} \\
&&+ \sum
_{x\in\partial\mathrm{BAD}(\omega) } {\mathbf1} {\{T_x<\tau_1\}}\sum_{i=1}^{\infty
} {\mathbf1} {
\{X_i=x\}} T^+_{\mathrm
{GOOD},E_{\le n^{\delta}}}\circ\theta_i.
\end{eqnarray*}

Recalling the definition \eqref{def_chi} of $\chi$, for any $\epsilon>0$, we see that
\begin{eqnarray*}
E^{\omega}_0\bigl[{\mathbf1} {\bigl\{\chi\leq n^{\epsilon}\bigr\}}
T_\tau\bigr]&\leq&  E_0^\omega\left[T^+_{      \mathrm{GOOD}\cup\{0\}, E_{< n^\delta}     }\right]
\\
&+&\sum_{x\in B(n^{\epsilon}, n^{C\epsilon})} \Biggl[{\mathbf1}
{\bigl\{x\in
\mathrm{GOOD}(\omega)\bigr\}}E^\omega_0\left[ \sum_{i=0}^\infty\1{X_i=x}\right]\\
&+& {\mathbf1} {\bigl\{x\in\partial\mathrm
{BAD}(\omega)
\bigr\}} E^{\omega
}_0 \Biggl[ \sum_{i=1}^{\infty}{
\mathbf1} {\{X_i=x\}} T^+_{\mathrm{GOOD},E_{< n^{\delta}}}\circ\theta_i
\Biggr] \Biggr].
\end{eqnarray*}
Using Markov's property and Lemma \ref{backbone}, we obtain
\begin{eqnarray*}
E^{\omega}_0\bigl[{\mathbf1} {\bigl\{\chi\leq n^{\epsilon}\bigr\}}
T_\tau\bigr]&\leq&  E_0^\omega\left[T^+_{      \mathrm{GOOD}\cup\{0\}, E_{< n^\delta}     }\right]\\
&&+ C(K) \sum_{x\in B(n^{\epsilon}, n^{C\epsilon})} \bigl[{\mathbf1}
{\bigl\{ x\in
\mathrm{GOOD}(\omega)\bigr\}}\\
&&+ {\mathbf1} {\bigl\{x\in\partial\mathrm
{BAD}(\omega)
\bigr\}} E^{\omega
}_x\bigl[T^+_{\mathrm{GOOD},E_{< n^{\delta}}}\bigr]\bigr].
\end{eqnarray*}
%

Recall that $\partial\mathrm{BAD}(\omega)\subset \mathrm{GOOD}(\omega)$ and notice that, if $x$ is good, then $T^+_{\mathrm{GOOD},E_{< n^{\delta}}}=T^+_{\mathrm{GOOD}\cup\{x\},E_{< n^{\delta}}}$. We may now apply Lemma \ref{timetrap}: 
\begin{eqnarray*}
 E_0^{\omega}\bigl[{\mathbf1} {\bigl\{\chi\leq n^{\epsilon}\bigr\}}
T_\tau\bigr]&\le& C(K)\biggl[ \exp
\bigl(3\lambda\left|\partial\mathrm{BAD}^{\mathrm{s}}_0(
\omega)\right|\bigr) \biggl(1+\sum_{e\in
E(\mathrm{BAD}^{\mathrm{s}}_0)\cap E_{< n^\delta}}
c_*^{\omega}(e) \biggr)
\\
&+& \sum_{x\in B(n^{\epsilon}, n^{C\epsilon})} {\mathbf1}
{\bigl\{x
\in\mathrm{GOOD}(\omega)\bigr\}}
\\
&+&\sum_{x\in B(n^{\epsilon},n^{C\epsilon})} {\mathbf1} {\bigl\{x\in
\partial\mathrm
{BAD}(\omega)\bigr\}} \times\exp
\bigl(3\lambda W\left(\mathrm{BAD}^{\mathrm{s}}_x(
\omega)\right)\bigr)\\
&&\qquad\qquad\qquad\times \biggl(1+\sum_{e\in
E(\mathrm{BAD}^{\mathrm{s}}_x)\cap E_{< n^{\delta}}}
c_*^{\omega}(e) \biggr)\biggr]
\\
&\leq& C(K)n^{C\epsilon}\left[1+ \max_{x\in B(n^{\epsilon},n^{C\epsilon})}
{\mathbf1} {\bigl\{x
\in\partial\mathrm{BAD}(\omega)\cup\{0\}\bigr\}}\right.\\
&&\hspace*{60pt}{}\times \left| E\bigl(\mathrm
{BAD}^{\mathrm{s}}_x(\omega)\bigr)\right|\times
\exp\bigl(3\lambda W\left(\mathrm{BAD}^{\mathrm{s}}_x(
\omega)\right)\bigr)
\\
&&\left.\hspace*{90pt}\qquad\quad{} \times\Bigl(1+\max_{e\in E(\mathrm{BAD}^{\mathrm{s}}_x)\cap E_{< n^{\delta}}}
c_*(e) \Bigr)\right]
\\
&&\qquad\leq C(K) n^{C\epsilon}\left[ \max_{x\in B(n^{\epsilon},n^{C\epsilon})}
 \exp\bigl(4\lambda W\left(\mathrm{BAD}^{\mathrm{s}}_x(\omega)\right)\bigr)\right]
\\
&&\hspace*{90.5pt}\qquad\quad{} \times\left[1+\max_{e\in F^\omega(B(n^{\epsilon},n^{C\epsilon}))}
c_*(e) \right],
\end{eqnarray*}
where we used that $\left| E(\mathrm{BAD}^{\mathrm{s}}_x(\omega
))\right|\leq C
W\left(\mathrm{BAD}^{\mathrm{s}}_x(\omega)\right)^d \leq C \exp\left(\lambda W\left(\mathrm{BAD}^{\mathrm{s}}_x(\omega)\right)\right)$ and defined
\begin{equation}\label{blablablablou}
F^\omega(B(n^{\epsilon},n^{C\epsilon})):=\bigcup_{x\in B(n^{\epsilon},n^{C\epsilon})}E(\mathrm{BAD}^{\mathrm{s}}_x)\cap E_{< n^{\delta}}.
\end{equation}
Now, fix some integers $I\in\N^*$ and $i\in[0,I-1]$. Using the previous inequality, on the event $\{\max_{e\in F^\omega(B(n^{\epsilon},n^{C\epsilon}))}c_*(e) \le n^{\delta\frac{i+1}{I}}\}\cap\{E^{\omega}\bigl[{\mathbf1} {\bigl\{\chi\leq n^{\epsilon}\bigr\}}T_\tau\bigr]>n^{\delta\frac{i+2}{I}}\}$, we have that
\[
\max_{x\in B(n^{\epsilon},n^{C\epsilon})}
 \exp\bigl(4\lambda W\left(\mathrm{BAD}^{\mathrm{s}}_x(\omega)\right)\bigr)>\frac{n^{\frac{\delta}{I}-C\epsilon}}{C(K)}.
\]
Using Lemma \ref{tgb} and the previous remark, we have
\begin{eqnarray}
&\PR_0^K&\left[E^{\omega}\bigl[{\mathbf1} {\bigl\{\chi\leq n^{\epsilon}\bigr\}}\widetilde{\tau}_1^{\le n^{\delta}}\bigr]>n^{\delta\frac{i+2}{I}},\max_{e\in F^\omega(B(n^{\epsilon},n^{C\epsilon}))}c_*(e) \le n^{\delta\frac{i+1}{I}}\right]\nonumber\\
&&\le \PR_0^K\left[\max_{x\in B(n^{\epsilon},n^{C\epsilon})}
 \exp\bigl(4\lambda W\left(\mathrm{BAD}^{\mathrm{s}}_x(\omega)\right)\bigr)>\frac{n^{\frac{\delta}{I}-C\epsilon}}{C(K)}\right]\nonumber\\
&&\le Cn^{C\epsilon}\frac{C(K)}{In^2}\le\frac{C(K)}{In},\label{inegIn}
\end{eqnarray}
as soon as $\epsilon>0$ is small enough (depending on $\delta$, $I$ and $C$ but not $K$) and $K$ is large enough (depending on $\delta$, $I$ and $C$).\\
Now, using Lemma \ref{tail_chi_tau} and since every edge $e\in F^\omega(B(n^{\epsilon},n^{C\epsilon}))$ is such that $c_*(e)< n^{\delta}$, we have
\begin{eqnarray*}
\PR_0^K\left[T_\tau\ge an\right]&\le& \PR_0^K[\chi> n^\epsilon]+\PR_0^K\left[{\mathbf1} {\bigl\{\chi\leq n^{\epsilon}\bigr\}}T_\tau\ge an\right]\\
&\le& \frac{C(K)}{n}+\sum_{i=0}^{I-1}\PR_0^K\left[{\mathbf1} {\bigl\{\chi\leq n^{\epsilon}\bigr\}}T_\tau\ge an,\right.\\
&&\qquad\qquad\qquad\left.\max_{e\in F^\omega(B(n^{\epsilon},n^{C\epsilon}))}c_*(e) \in [n^{\delta\frac{i}{I}},n^{\delta\frac{i+1}{I}})\right].
\end{eqnarray*}
By \eqref{inegIn}, this implies
\begin{eqnarray*}
\PR_0^K\left[T_\tau\ge an\right]&\le& \frac{C(K)}{n}+\sum_{i=0}^{I-1}\ES_0^K\Bigg[P^\omega\left[{\mathbf1} {\bigl\{\chi\leq n^{\epsilon}\bigr\}}T_\tau\ge an\right]\\
&&\qquad\qquad\qquad\times\1{\max_{e\in F^\omega(B(n^{\epsilon},n^{C\epsilon}))}c_*(e) \in [n^{\delta\frac{i}{I}},n^{\delta\frac{i+1}{I}})}\\
&&\qquad\qquad\qquad\times \1{E^{\omega}\bigl[{\mathbf1} {\bigl\{\chi\leq n^{\epsilon}\bigr\}}T_\tau\bigr]\le n^{\delta\frac{i+2}{I}}}\Bigg].
\end{eqnarray*}
By Markov's inequality, we obtain
\begin{eqnarray}
&&\PR_0^K\left[T_\tau\ge an\right]\nonumber\\
&&\le \frac{C(K)}{n}+\sum_{i=0}^{I-1}\frac{n^{\delta\frac{i+2}{I}}}{an}{\bf P}\left[\max_{e\in F^\omega(B(n^{\epsilon},n^{C\epsilon}))}c_*(e) \in [n^{\delta\frac{i}{I}},n^{\delta\frac{i+1}{I}})\right].\label{noinspiration}
\end{eqnarray}
Now, for any $i\in[0,I-1]$, using that $\left| E(\mathrm{BAD}^{\mathrm{s}}_x(\omega
))\right|\leq C
\left|\partial\mathrm{BAD}^{\mathrm{s}}_x(\omega)\right|^d\le\left|W(\mathrm{BAD}^{\mathrm{s}}_x)\right|^{dC}$ and Lemma \ref{tgb}, we have
\begin{eqnarray*}
&&\PR_0^K\left[\max_{e\in F^\omega(B(n^{\epsilon},n^{C\epsilon}))}c_*(e) \in [n^{\delta\frac{i}{I}},n^{\delta\frac{i+1}{I}})\right]\\
&&\qquad\qquad \le\sum_{x\in B(n^{\epsilon},n^{C\epsilon})} {\bf P}\left[\max_{e\in E(\mathrm{BAD}^{\mathrm{s}}_x)}c_*(e) \in [n^{\delta\frac{i}{I}},n^{\delta\frac{i+1}{I}})\right]\\
&&\qquad\qquad \le n^{C\epsilon}\sum_{e\in E(B(n^{C\epsilon},n^{C\epsilon}))} {\bf P}\left[c_*(e) \in [n^{\delta\frac{i}{I}},n^{\delta\frac{i+1}{I}})\right] \\
&&\qquad\qquad\qquad\qquad+n^{C\epsilon}\sum_{x\in B(n^{\epsilon},n^{C\epsilon})}{\bf P}\left[\left|E(\mathrm{BAD}^{\mathrm{s}}_x)\right|\ge n^\epsilon\right]\\
&&\qquad\qquad \le C(K)\frac{n^{C\epsilon} }{n^{\gamma\delta\frac{i}{I}}}L_{\max}(n),
\end{eqnarray*}
as soon as $K$ is large enough (depending on $\delta$, $I$, $C$ and $\epsilon$), and where $L_{\max}(n)=\max\{L(n^{\gamma\delta\frac{i}{I}}), i=0,...,I-1\}$, with $L$ the slowly-varying function from the tail $c_*$ introduced in \eqref{assProba}. From \eqref{noinspiration}, we then deduce
\begin{eqnarray*}
\PR_0^K\left[T_\tau\ge n\right]
&\le& \frac{C(K)}{n}+ C(K)L_{\max}(n)\frac{n^{\delta\frac{2}{I}+C\epsilon}}{n}\sum_{i=0}^{I-1} n^{   \frac{ \delta(1-\gamma)  }{I}i   }\\
&\le& \frac{C(K)}{n}+ C(K)L_{\max}(n)n^{\delta\frac{2}{I}+C\epsilon}\frac{ n^{ \delta(1-\gamma)   }}{n}\\
&\le &\frac{C(K)}{n}+ \frac{C(K) L_{\max}(n)}{n^{     1-\delta(1-\gamma)-   \delta\frac{2}{I} -C\epsilon      }        }\\
&\le& \frac{C(K) }{n^{     \gamma+\frac{(1-\delta)(1-\gamma)}{2}        }},
\end{eqnarray*}
as soon as $I$ is large enough (depending on $\gamma$ and $\delta$, but not $K$), $\epsilon$ small enough (depending on $\delta$, $\gamma$, $I$ and $C$), and for $n$ large enough (depending on $\delta$, $\gamma$ and $L$). We used the fact that, as $L$ is slowly varying, for any $\epsilon'>0$, $x^{-\epsilon'}L(x)\to0$ as $x\to+\infty$.\\
Recall that $K$ depends on $\delta$, $I$, $C$ and $\epsilon$, so the proof is consistent and we can conclude.

\end{proof}

A simple variation of the previous lemma is the following result.

\begin{lemma}\label{decaytau2}
Fix $e\in E(\Z^d)$.
For any
$\delta\in(0,1)$ and $m>\delta$, there exists $K_0<\infty$ such that, for any $K\geq K_0$ and for any constant $a>0$,
\begin{eqnarray*}
&&\PR^K_0\bigl[{
\tau}_1^{< n^{\delta}}>an, c_*(e)\ge n^m\mid D=\infty\bigr]\\
&&\qquad\qquad\qquad\leq C(K,\delta,a,m)
n^{-\gamma-\frac{(1-\delta)(1-\gamma)}{2}}{\bf P}[c_*(e)\ge n^m].
\end{eqnarray*}
\end{lemma}
\begin{proof}
The proof of this lemma is very similar to the proof of Lemma \ref{decaytau}. Here are the only modifications to be made:
\begin{enumerate}
\item in \eqref{blablablablou}, as $e\notin E_{< n^{\delta}}$, we can define $F^\omega(B(n^{\epsilon},n^{C\epsilon}))$ as
\[
F^\omega(B(n^{\epsilon},n^{C\epsilon})):=\left(\bigcup_{x\in B(n^{\epsilon},n^{C\epsilon})}E(\mathrm{BAD}^{\mathrm{s}}_x)\cap E_{< n^{\delta}}\right)\setminus \{e\},
\]
excluding the edge $e$ will also ensure that events measurable with respect to $F^\omega(B(n^{\epsilon},n^{C\epsilon}))$ are independent of those measurable with respect to $e$.
\item in the bound corresponding to \eqref{inegIn}, we will obtain the same upper-bound multiplied by ${\bf P}[c_*(e)\ge n^m]$ as soon as $\epsilon$ is small enough and $K$ is large enough, depending on $m$, by Lemma \ref{tgb} and using that ${\bf P}[c_*(e)\ge n^m]\ge cn^{-\gamma m-\epsilon}$ for any $\epsilon>0$ by \eqref{assProba};
\item similarly, we can obtain a bound $\PR_0^K[\chi>n^\epsilon]\le C(K){\bf P}[c_*(e)\ge n^m]/n$ by Lemma \ref{tail_chi_tau};
\item therefore, instead of \eqref{noinspiration}, we obtain
\begin{eqnarray*}
&&\PR_0^K\left[T_\tau\ge an, c_*(e)\ge n^m\right]\\
&&\le \frac{C(K)}{n}{\bf P}[c_*(e)\ge n^m]+\sum_{i=0}^{I-1}\frac{n^{\delta\frac{i+2}{I}}}{an}{\bf P}\left[\max_{e\in F^\omega(B(n^{\epsilon},n^{C\epsilon}))}c_*(e) \in [n^{\delta\frac{i}{I}},n^{\delta\frac{i+1}{I}}), \right.\\
&& \hspace*{300pt}c_*(e)\ge n^m\bigg]\\
&&\le \frac{C(K)}{n}{\bf P}[c_*(e)\ge n^m]+\sum_{i=0}^{I-1}\frac{n^{\delta\frac{i+2}{I}}}{an}{\bf P}\left[\max_{e\in F^\omega(B(n^{\epsilon},n^{C\epsilon}))}c_*(e) \in [n^{\delta\frac{i}{I}},n^{\delta\frac{i+1}{I}})\right]\\
&&\qquad\qquad\qquad\qquad\qquad\qquad\qquad\times{\bf P}[c_*(e)\ge n^m],
\end{eqnarray*}
and the conclusion follows easily as in the proof of Lemma \ref{decaytau}.
\end{enumerate}
\end{proof}


\section{First estimates on the number of large traps in regeneration times}


As explained in Section \ref{sketch}, in order to understand the time spent by the walker during one regeneration period, it is important to understand the time it spends in large traps it meets. In this section, we start by studying the number of such traps the walker can meet. In particular, we prove that, with overwhelming probability, the walker meets at most one edge with large conductance during one regeneration period, see Proposition \ref{not_2_traps}.

First let us introduce some notations. We will write
\begin{align}\label{romeo}
\overline{\PR}[\ \cdot \ ]:=\PR_0^K[\ \cdot\ | D=\infty].
\end{align}
 We call a {\it large trap} (resp.~{\it medium trap}) an edge with a conductance greater than $n$ (resp.~$n^\delta$), where $n$ and $\delta$ are some variables we will be using later. Also, we define
\begin{equation}\label{def_exist_1_trap}
LT(K,n)=\{\text{there exists $e\in E(\Z^d)$ such that, } c_*^{\omega}(e)\geq n \text{ and } T_e<\tau_1^{(K)}\},
\end{equation}
and define the measure of a regeneration block in which we encounter a large trap.
\begin{eqnarray}\label{defPn}
\PR_n[\ \cdot \ ]:=\PR_0^K[\ \cdot \ | LT(n),\ D=\infty].
\end{eqnarray}

Further, we set
\begin{align}\label{def_exist_2_trap}
 SLT(\delta,K,n) &=\{\text{there exists $e\in E(\Z^d)$ such that, } c_*^{\omega}(e)\geq n,\ T_e<\tau_1, \\  \nonumber
&   \text{and there exists } e'\neq e,\ e'\in B(2\chi,2\chi^{\alpha})\text{, such that }\ c_*^{\omega}(e')\geq n^{\delta} \},
\end{align}
and 
\begin{equation}\label{def_bn_trap}
{NLT}(\delta,K,n)=\text{card}\{e\in E(\Z^d),\ c^{\omega}_*(e)\geq n^{\delta} \text{ and } e\in B(2\chi,2\chi^{\alpha})\},
\end{equation}
a random variable which upper-bounds the number of medium traps seen in a regeneration time.

Let us introduce the event on which the walker meets only one large conductance during one regeneration period:
 \begin{align}\nonumber
 OLT(\delta,K,n)=\{\exists e\in E(\Z^d): c^{\omega}_*(e)\geq n, T_e<\tau_1\text{ and for any }e'\neq e\\
 \text{ with }T_{e'}<\tau_1 \text{ or }e'\sim e,
 \text{ we have }c_*^{\omega}(e')< n^{\delta}\}.\label{def_OLT}
\end{align}
Moreover, for any $e\in E(\Z^d)$, we define the event  
 \begin{align}\nonumber
 OLT_e(\delta,K,n)=\{ c^{\omega}_*(e)\geq n, T_e<\tau_1\text{ and for any }e'\neq e\text{ with }T_{e'}<\tau_1\\
\text{or }e'\sim e, \text{ we have }c_*^{\omega}(e')< n^{\delta}\},\label{def_OLTe}
\end{align}
so that $ OLT(\delta,K,n)=\bigcup_{e\in E(\Z^d)}  OLT_e(\delta,K,n)$.

\begin{remark}\label{elodie}
Note that if $e$ is hit before $\tau_1$, then $e'\in B(2\chi,2\chi^{\alpha})$ for any $e'\sim e$. Therefore, on $SLT(\delta,K,n)^c\cap \{T_e<\tau_1,c_*(e)\ge n\}$, we have that $c_*(e')<n^\delta$ for any $e'\sim e$.
\end{remark}
The goal of this section is to show that 
\begin{proposition}\label{not_2_traps}
Fix $\delta\in (0,1)$. There exists $K_0<\infty$ such that, for any $K\ge K_0$, there exists some $\epsilon>0$ such that we have
\[
\ES_n[{NLT}(\delta,K,n)\1{SLT(\delta,K,n)}] =o(n^{-\epsilon}),
\]
in particular
\[
\PR_n[SLT(\delta,K,n)]=o(n^{-\epsilon}).
\]
\end{proposition}

This proposition will follow easily from an upper bound on the quantity {$\ES[{NLT}(\delta,K,n) \1{SLT(\delta,K,n)}]$} (see Lemma~\ref{not_2_trapsb}) and a lower bound on $\PR[LT(n)]$ (see Lemma~\ref{LB_pn}). This is the focus of the remainder of this section.

\subsection{It is unlikely to see medium traps close to large traps}

\begin{lemma}\label{not_2_trapsb}
For any $\delta\in (0,1)$ and any $\epsilon>0$ there exists $K_0<\infty$ such that, for any $K\ge K_0$,
\[
\overline{\ES}[{NLT}(\delta,K,n)\1{SLT(\delta,K,n)}]\leq  C n^{\epsilon} {\bf P}[c_*\geq n^{\delta}] {\bf P}[c_*\geq n],
\]
in particular
\[
\overline{\PR}[SLT(\delta,K,n)] \leq C n^{\epsilon} {\bf P}[c_*\geq n^{\delta}] {\bf P}[c_*\geq n].
\]
\end{lemma}

\begin{proof}
Note that we have 
\begin{equation} \label{ub_nlt_chi}
{NLT}(\delta,K,n)\leq C(d)\chi^{2d\alpha},
\end{equation}
 where $\alpha>d+3$ is the constant from the definition \eqref{def_chi} of $\chi$. Using this, we see that for any $\epsilon'>0$
 \[
\overline{\ES}[{NLT}(\delta,K,n)\1{\chi\geq n^{\epsilon'}}]\leq 
C(K)n^{-2/\gamma}=o({\bf P}[c_*\geq n^{\delta}] {\bf P}[c_*\geq n]),
\]
by Lemma~\ref{tail_chi_tau}, and choosing $K\ge K_0$ with $K_0<\infty$ large enough (depending on $\epsilon'$, $d$, $\gamma$ and $\alpha$). Recalling~\eqref{ub_nlt_chi}, we can see that for any $\epsilon'>0$, there exists $K_0<\infty$ such that, for any $K\ge K_0$,
\begin{align*}
& \overline{\ES}[{NLT}(\delta,K,n)\1{SLT(\delta,K,n)}] \\
\leq & C(d)n^{ 2d \alpha \epsilon'}\overline{\PR}[SLT(\delta,K,n),{\chi\leq n^{\epsilon'}}]+o({\bf P}[c_*\geq n^{\delta}] {\bf P}[c\geq n]) \\
\leq & C(d)n^{2d \alpha \epsilon'} \sum_{e,e'\in  B(n^{\epsilon'},n^{\alpha\epsilon'})} {\bf P}[c_*^{\omega}(e)\geq n] {\bf P}[c_*^{\omega}(e')\geq n^{\delta}] +o({\bf P}[c_*\geq n^{\delta}] {\bf P}[c\geq n]) \\
\leq & C(K,d)n^{c(d,\alpha)\epsilon'}{\bf P}[c_*\geq n^{\delta}] {\bf P}[c_*\geq n],
\end{align*}
and the result follows from choosing $\epsilon'$ small enough.
\end{proof}

\subsection{Lower bound on the probability of meeting a large trap in a regeneration time}

\begin{lemma}\label{LB_pn}
We have
\[
c(K,d) {\bf P}[c_*\geq n] \leq \overline{\PR}[LT(n)].
\]
\end{lemma}
\begin{proof}
First, notice that 
\begin{align}\nonumber
\overline{\PR}[LT(n)]&=\PR_0^K[\exists e\in E(\Z^d)\text{ such that, } c_*^{\omega}(e)\geq n \text{ and } T_e<\tau_1\mid D=\infty] \\ \label{zf}
 &=  \frac 1{\PR_0^K[D=\infty]}\PR_0^K[\text{there exists $e\in E(\Z^d)$ such that, } c_*^{\omega}(e)\geq n \\\nonumber
 &\qquad\qquad\qquad\qquad\qquad\text{ and } T_e<\tau_1, D=\infty] \\ \nonumber
  & \geq C(K) \PR_0^K[\text{there exists $e\in E(\Z^d)$ such that, } c_*^{\omega}(e)\geq n \\\nonumber
  &\qquad\qquad\qquad\qquad\qquad\text{ and } T_e<\tau_1, D=\infty].
  \end{align}
  
  Let us describe a way to construct the event appearing in the last probability. 
  \begin{figure}[h] 
\includegraphics[width=250pt]{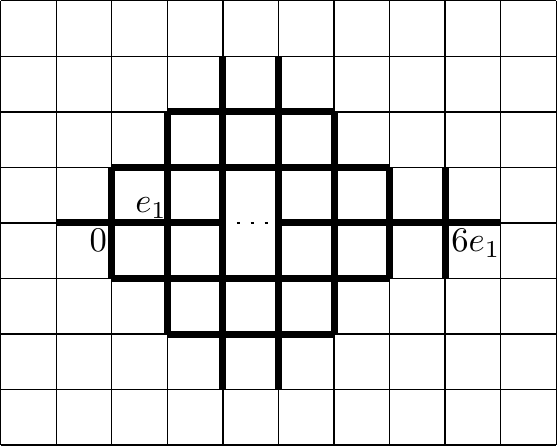}
   \caption{\label{setA} The edges in bold are those incident to some vertex of $\mathcal{A}$ and the dotted edge is $[e_2,e_3]$.}
\end{figure}
  As depicted in Figure \ref{setA}, we set $\mathcal{A}$ to be the set of vertices: $0$, $e_1$, $e_1\pm e_i$, $2e_1 \pm e_i$, $2e_1\pm 2e_i$, $3e_1\pm 2e_i$, $3e_1\pm e_i$, $4e_1$, $4e_1\pm e_i$,  for all $i\in [2,d]$, $5e_1$, $6e_1$ and the events
\[
A=\{\text{any $x\in \mathcal{A}$ is $6e_1$-open}\} \text{ and } B=\{6e_1 \text{ is good}\},
\]
where  a vertex is called $x$-open if it would be open in $\omega$ after all edges adjacent to $x$ are turned normal.

Note that  $A$ and $B$ are independent and independent of $c_*([2e_1,3e_1])$. Also, recall that, from the definition \eqref{deftau1}, we have that $\tau_1\ge 3$.

We may notice that on $A\cap B$, if
\begin{enumerate}
\item $(X_1,Z_1)=(e_1,1), (X_2,Z_2)=(2e_1,1)$ (hence $\tau_1> T_{[2e_1,3e_1]}=2$),
\item $T_{6e_1}\circ \theta_{T_{\Z^d\setminus \{2e_1,3e_1\}}\circ \theta_2}\leq T_{\partial (\mathcal{A}\setminus\{0\})}\circ \theta_{T_{\Z^d\setminus \{2e_1,3e_1\}}\circ \theta_2}$, and $Z_i=1$ for any $3\le i \le T_{6e_1}$ such that $X_i\sim0$,
\item $D\circ \theta_{T_{6e_1}} =\infty$,
\item $c^{\omega}_*([2e_1,3e_1])\geq n$.
\end{enumerate}
then we have $D=\infty$ and $\tau_1> T_{[2e_1,3e_1]}$ and thus there exists $e\in E(\Z^d)$ such that,  $c_*^{\omega}(e)\geq n$  and $T_e<\tau_1$. To provide a lower-bound on $\overline{\PR}[LT(n)]$, we aim at estimating the four different events which will give us a lower bound on $\overline{\PR}[LT(n)]$.

On $A\cap B$, we see, by Remark~\ref{rem_pkopen}, that we have 
\begin{equation}\label{poiq}
P^{\omega}_0[(X_1,Z_1)=(e_1,1), (X_2,Z_2)=(2e_1,1)]\geq \kappa^2,
\end{equation}
moreover on $A\cap B$
\begin{align}\label{poiw}
 P^{\omega}_{2e_1}&[T_{6e_1}\circ \theta_{T_{\Z^d\setminus \{2e_1,3e_1\}}}\leq T_{\partial (\mathcal{A}\setminus\{0\})}\circ \theta_{T_{\Z^d\setminus \{2e_1,3e_1\}}},\nonumber\\
& Z_i=1\text{ for any }3\le i \le T_{6e_1}\text{ such that }X_i\sim0] \geq \kappa^7,
\end{align}
which follows from Strong Markov's property, Remark~\ref{rem_pkopen} and the fact that, on $A\cap B$,  from any neighbour of $2e_1$ or $3e_1$, there exists an open nearest-neighbour path of length at most $7$ in $\mathcal{A}\setminus  \{0\}$ to $6e_1$.

Using Markov's property (at times $2$, $T_{\Z^d\setminus \{2e_1,3e_1\}}\circ \theta_2$ and $T_{6e_1}$) along with~(\ref{poiq}) and~(\ref{poiw}) we may see
\begin{align}\label{tie}
 &\PR_0^K[A,B,\text{there exists $e\in E(\Z^d)$ such that, } c^{\omega}_*(e)\geq n \text{ and } T_e<\tau_1, D=\infty] \\ \nonumber
\geq &c {\bf E}_0^K[\1{A,B}\1{c_*([2e_1,3e_1])\geq n}P_{6e_1}^{\omega}[D=\infty] ].
\end{align}

Recalling that $\1{A}$, $c_*([2e_1,3e_1])$ and $\1{B}P_{6e_1}^{\omega}[D=\infty]$ are ${\bf P}_0^K$-independent, we have
\begin{align*}
& {\bf E}_0^K[\1{A,B}\1{c_*([2e_1,3e_1])\geq n}P_{6e_1}^{\omega}[D=\infty] ]\\
\geq & {\bf P}^K_0[A] {\bf P}[c_*([2e_1,3e_1])\geq n]{\bf E}[\1{B}P_{6e_1}^{\omega}[D=\infty]].
\end{align*}

We have $ {\bf P}^K_0[A] \geq c>0$ and by translation invariance 
\[
{\bf E}[\1{B}P_{6e_1}^{\omega}[D=\infty]]={\bf E}[\1{0\text{ is good}}P^{\omega}_0[D=\infty]]>0,
\]
by Lemma~\ref{posescape} and the fact that ${\bf P}[0\text{ is good}]>0$ (see Lemma \ref{BLsizeclosedbox}). This means that, by~\eqref{tie},
\begin{align}\label{tyu}
 &\PR[A,B,\text{there exists $e\in E(\Z^d)$ such that, } c^{\omega}_*(e)\geq n \text{ and } T_e<\tau_1, D=\infty] \\ \nonumber
 \geq & c {\bf P}[c_*\geq n].
\end{align}

This and~(\ref{zf}) imply the result.
\end{proof}

\section{The environment seen by the particle close to a large edge using a coupling}\label{timberlake}

We know that with high probability a large edge will only be surrounded by relatively small edges by Proposition \ref{not_2_traps}. Because of this fact, the random walk (once it hits this edge) will typically make a large number of back and forth crossings of this edge. This has several important consequences.

Firstly, the exit probabilities from the edge $e$ are almost proportional to the conductances leaving $e$ (see Lemma~\ref{exit_pro}). This indicates that we should be able to couple, with high probability, the random walk with a random walk $Y^e$ in an environment $\omega_e$ where the large edge $e$ is collapsed into one point, see Figure \ref{collapsing}.

We will then argue that for this coupling
\begin{enumerate}
\item there exists a regeneration $\tau_1^{Y^e}$ for the random walk in $\omega_e$ which coincides with high probability with $\tau_1$. This regeneration time will take into account that with overwhelming probability the random walker will do a back and forth crossing of $e$ when it reaches it (see Lemma~\ref{coupling}).
\item The time change induced by the  back and forth crossings of $e$ will be described in terms of exponential random variables in the limit (see Lemma~\ref{coupl_exp}).
\end{enumerate}

The strength of this coupling is that both the modified regeneration time $\tau_1^{Y^e}$ and the time change, normalized by $c_*(e)$, will essentially be independent of the value of $c_*(e)$. This will allow us to describe the environment seen from the particle around a large conductance independently of the precise value of $c_*(e)$ as long as it is large. Hence, this coupling will be useful to describe the behaviour of the random walk in a regeneration block where it meets a large conductance.

\subsection{Approximation of exit probabilities of the  large edge}

Recall that, for a vertex $y$ and an edge $e$, we write $y\sim e$ if $y\sim e^+$ or $y\sim e^-$ and $y\notin \{e^+,e^-\}$. Moreover, if $y$ is a neighbour of $e$ we denote $e_y$ the unique endpoint of $e$ which is adjacent to $y$.\\
Finally, we denote $\omega_e$ the environment inherited from $\omega$ but where the edge $e$ is collapsed into one point denoted $x_e$, see Figure \ref{collapsing}. In this environment $\omega_e$, we have $c^{\omega_e}(e')=c^{\omega}(e')$ for any $e'\neq e$. Therefore, $\pi^{\omega_e}(x)=\pi^{\omega}(x)$ for any $x\neq x_e$ and
\[
\pi^{\omega_e}(x_e)=\sum_{\substack{e'\in E(\Z^d):\\ e'\sim e}} c^{\omega_e}(e')=\pi^{\omega}(e^+)+\pi^{\omega}(e^-)-2c^{\omega}(e).
\]

\begin{lemma}\label{exit_pro}
Fix $\delta\in(0,1)$, an environment $\omega$ and an edge $e\in E(\Z^d)$. Assume that $c_*^{\omega}(e)\geq n$ and $c^{\omega}_*(e')\leq n^{\delta}$ for any $e'\sim e$. For any $x\in \{e^+,e^-\}$ and $y\sim e$, we have
\[
(1-C(d)n^{\delta-1}) \frac{c^{\omega_e}(y,x_e)}{\pi^{\omega_e}(x_e)}\leq  P^{\omega}_{x}[X_{T^{\text{ex}}_e}=y] \leq 
(1+C(d)n^{\delta-1}) \frac{c^{\omega_e}(y,x_e)}{\pi^{\omega_e}(x_e)}.
\]
\end{lemma}

\begin{proof}
For $y \sim e^+$ such that $y \neq e^-$, we can apply Markov's property at time $1$ and $2$ to see that
\[
 P^{\omega}_{e^+}[X_{T^{\text{ex}}_e}=y] =P^{\omega}_{e^+}[X_1=y]+P^{\omega}_{e_+}[X_1=e_-]P^{\omega}_{e_-}[X_1=e_+]P^{\omega}_{e^+}[X_{T^{\text{ex}}_e}=y] ,
 \]
 which yields
 \[
  P^{\omega}_{e^+}[X_{T^{\text{ex}}_e}=y] = \frac{c^{\omega}(e^+,y)}{\pi^{\omega}(e^+)} \frac 1 {1-P^{\omega}_{e_+}[X_1=e_-]P^{\omega}_{e_-}[X_1=e_+]}.
 \]
 
 A similar computation will yield that, for $z \sim e^-$ such that $z \neq e^+$,
 \[
   P^{\omega}_{e^+}[X_{T^{\text{ex}}_e}=z] = \frac{c^{\omega}(e^-,z)}{\pi^{\omega}(e^-)} \frac {P^{\omega}_{e_+}[X_1=e_-]} {1-P^{\omega}_{e_+}[X_1=e_-]P^{\omega}_{e_-}[X_1=e_+]}.
 \]

Recalling~\eqref{def_conduct} and \eqref{defpi}, we have that
\[
1-C(d)n^{\delta-1}\leq \frac{c^{\omega}(e) }{\pi^{\omega}(e^\mp)}\le \frac{\pi^{\omega}(e^\pm) }{\pi^{\omega}(e^\mp)}\leq 1+C(d)n^{\delta-1},
\]
where ``$\pm$'' can stand for $+$ or $-$ and $\mp$ is the opposite sign.
Then, we can see that, for any $y,y'\sim e$,
\[
(1-C(d)n^{\delta-1}) \frac{c^{\omega}(y',e_{y'})}{c^{\omega}(y,e_y)} \leq \frac{   P^{\omega}_{e^+}[X_{T^{\text{ex}}_e}=y'] }{   P^{\omega}_{e^+}[X_{T^{\text{ex}}_e}=y] }  \leq (1+C(d)n^{\delta-1}) \frac{c^{\omega}(y',e_{y'})}{c^{\omega}(y,e_y)} ,
\]
and a similar inequality holds under $P^{\omega}_{e^-}$.

Since $\sum_{y'\sim e }P^{\omega}_{e^+}[X_{T^{\text{ex}}_e}=y'] =1$ and $\sum_{y'\sim e}c^{\omega}(y',e_{y'})=\sum_{e'\sim e}c^{\omega}(e')=\pi^{\omega}(e^+)+\pi^{\omega}(e^-)-2c^{\omega}(e)$, a simple computation allows us to see that for any $y\sim e $
\[
(1-C(d)n^{\delta-1}) \frac{c^{\omega}(y,e_y)}{\pi^{\omega}(e^+)+\pi^{\omega}(e^-)-2c^{\omega}(e)}\leq  P^{\omega}_{e^+}[X_{T^{\text{ex}}_e}=y] ,
\]
and 
\[
P^{\omega}_{e^+}[X_{T^{\text{ex}}_e}=y] \leq 
(1+C(d)n^{\delta-1}) \frac{c^{\omega}(y,e_y)}{\pi^{\omega}(e^+)+\pi^{\omega}(e^-)-2c^{\omega}(e)},
\]
and a similar inequality holds under $P^{\omega}_{e^-}$.

The last two lines are a restatement of the lemma.
\end{proof}

\subsection{Definition of two walks on a modified graph} \label{iwontdo}

For $e\in E(\Z^d)$, we denote $(\Z^d_{e},E(\Z^d_e))$ the graph obtained by collapsing the edge $e$ into one vertex $x_e$, see Figure \ref{collapsing}. Given an environment $\omega$ of conductances on $\Z^d$, we denote $\omega_e$ the conductances induced on $\Z^d_e$ by $\omega$. Moreover, we denote ${\bf P}^e$ the law of the environment $\omega_e$ on $\Z^d_e$.

We will define two random walks  on $\Z^d_{e}$. We will show later that it is possible to couple those random variables with high probability. The first walk will simply correspond to the trace on $\Z^d_{e}$ of the random walk $X$ on $\Z^d$ in the environment $\omega$, which still depends on $c_*(e)$. The second walk will be a random walk on $\Z^d_{e}$ in the environment $\omega_e$ and thus will be independent of $c_*(e)$. 

Provided $c_*(e)$ is large, we will be able to couple those random walks with high probability, which intuitively means that the trace on $\Z^d_{e}$ of the random walk $X$ is essentially independent of the value of $c_*(e)$ when the latter is large.

We define and compare these walks in order to understand the influence of $c_*(e)$ on the trajectory when this conductance is large. In particular, we will prove that when  $c_*(e)$ is large, even if the walk $X$ spends a lot of time on $e$, the trace of $X$ outside $e$ will be essentially independent of $c_*(e)$.
 
 \subsubsection{The random walk in the environment where $e$ is collapsed}\label{sectwalkcollapsed}

First we define the random walk obtained from the conductances induced on $\omega_e$ and we denote the corresponding walk by $(Y_n^e)_{n\in \N}$. We extend to $\Z^d_e$ all the vocabulary and the definitions  introduced in Section \ref{sect_badareas}, where, in particular, we defined an open point, a good point and the subsets $\mathrm{GOOD}(\omega_e)$, $\mathrm{BAD}(\omega_e)$ and $\mathrm{BAD}^{\omega_e}(x)$ for $x\in\Z^d$. The only difference is that, for an environment $\omega_e$ on $\Z^d_e$, we always consider the vertex $x_e$ to be closed and consequently bad.\\
Besides, given an environment $\omega_e$, we define an enhanced version $(Y^e,Z^{Y^e})$ of $Y^e$, with law $P^{\omega_e}$, in the exact same way as in Section \ref{sect_altcons}, except for the transition probabilities on $x_e$ where, for any $y\sim x_e$,
\begin{enumerate}
\item $\tilde{p}^{\omega_e}((x_e,z_1),(y,1))=0$,
\item $\tilde{p}^{\omega_e}((x_e,z_1),(y,0))=p^{\omega_e}(x_e,y)=\frac{c([x_e,y])}{\pi^{\omega_e}(x)}$.
\end{enumerate}
Finally, we define ${\bf P}^e$ the law of the environment $\omega_e$ on $\Z^d_e$ and we denote $\PR^e_0$ the annealed law of $(Y^e,Z^{Y^e})$.

 \subsubsection{The trace of $(X_n)_{n\geq 0}$ outside of $e$}

Now, we are going to introduce $({X}^e_n, Z^{X^e}_n)_{n\in \N}$, an enhanced walk on $\Z^d_e$, as the trace of $(X_n, Z_n)_{n\geq 0}$ outside of $e$.

For this introduce the time change $(A^e_n)_{n\in \N}$ defined by $A^e_k=k$ for $k\leq T_{e}$ and for larger $k$ we set $A^e_{k+1}=\inf \{ j>A_k,\ \{X_j,X_{j-1}\} \neq e\}$. 

Now set
\begin{align}\label{defAek}
({X}^e_n,Z^{X^e}_n) =\begin{cases} (X_{A^e_n},Z_{A^e_n}) & \text{ if } X_{A^e_n-1},X_{A^e_n}\notin e \\
							(X_{A^e_n},0)& \text{ if } X_{A^e_n-1}\in e, \ X_{A^e_n}\notin e \\
                                 (x_e,Z_{A^e_n}) & \text{otherwise.}
                                 \end{cases}
\end{align}

In words  ${X}^e_n$ follows the transitions that $X_n$ makes outside of $e$. $Z^{X^e}$ does as well, except when the walk is jumping out of the edge $e$. 

\begin{remark} We may notice that $({X}^e_n)_{n\in \N}$ is not a Markov chain, because its transition probabilities at $x_e$ depends on which vertex it used to enter $x_e$. It is however a ``two-step'' Markov chain.
\end{remark}

\subsection{Coupling $(Y^e_n)_{n\in \N}$ and $({X}^e_n)_{n\in \N}$}

The goal of this section is to show that there exists a version of $(Y^e_n,Z^{Y^e}_n)_{n\in \N}$ which is, with high probability, equal to $({X}^e_n,Z^{X^e}_n)_{n\in \N}$. This will allow us to say that the behaviour of $(X_n)_{n\in \N}$ outside of the edge $e$ is, with high probability, independent of the value of $c_*({e})$.\\

For any set $A$ of vertices or edges of $\Z^d_e$, we denote $T_A^{X^e}$ (resp.~$T_A^{Y^e}$) the first time $X^e$ (resp.~$Y^e$) hits $A$.\\
Recall that $X^e$ is the trace of $X$ on $\mathbb{Z}^d_e$ and, as the two walks $X$ and $X^e$ coincide until $T_e$, we have $T_e=T^{X^e}_{x_e}$. Thus, later on, when dealing with $X^e$ and by a slight abuse, we will use the stopping time $T_e$ without any further specification.

\begin{lemma}\label{coupling0}
Fix $\omega$, $e\in E(\Z^d)$ and $\delta\in(0,1/3)$. Assume that $c^{\omega}_*(e)\geq n$ and $c^{\omega}_*(e')\leq n^{\delta}$ for any $e'\sim e$. There exists a coupling $\widehat{P}^\omega$ of $({X}^e_k,Z^{X^e}_k)_{k}$ and $(Y^e_k,Z^{Y^e}_k)_{k}$ such that
\[
\widehat{P}^{\omega}[({X}^e_k,Z^{X^e}_k)_{k\le T_{e}+n^{2\delta}}\neq (Y^e_k,Z^{Y^e}_k)_{k\le T_{e}+n^{2\delta}}] \leq C(d)n^{3\delta-1},
\]
and $T^{X^e}_e=T^{Y^e}_{x_e}=T_e$, $\widehat{P}^{\omega}$-almost surely.
We also define $\widehat{\PR}^K$, the annealed version of $\widehat{P}^{\omega}$.
\end{lemma} 

\begin{proof} Note first that it is plain that the transition probabilities of $(Y^e_k,Z^{Y^e}_k)_{k\in \N}$ and $({X}^e_k,Z^{X^e}_k)_{k\in \N}$ at any point $x\in \Z^d_e$ are given by those of $(X_k,Z_k)_{k\in \N}$ except for $x=x_e$.\\
At the point $x_e$ we have that for any $y,y' \sim x_e$
\begin{align}\nonumber
P^{\omega}[({X}^e_{k+1},Z^{X^e}_{k+1})=(y,1)\mid X^e_k=x_e \text{ and  } {X}^e_{k-1}=y']=&0,\\
P^{\omega}[({X}^e_{k+1},Z^{X^e}_{k+1})=(y,0)\mid X^e_k=x_e \text{ and  } {X}^e_{k-1}=y']=&P^{\omega}_{e_{y'}}[X_{T^{\text{ex}}_e}=y],\label{bosselarecherche}
\end{align}
and
\begin{align}\nonumber
P^{\omega}[Y^e_{k+1}=(y,1)\mid Y^e_k=x_e]=&0,\\
P^{\omega}[Y^e_{k+1}=(y,0)\mid Y^e_k=x_e]=&\frac{c^{\omega_e}(y,x_e)}{\pi^{\omega_e}(x_e)}.\label{bosselarecherche2}
\end{align}

So, under $\widehat{P}^{\omega}$ we define the process $(({X}^e_k,Z^{X^e}_k),(Y^e_k,Z^{Y^e}_k))_k$ with the following transition probabilities:
\begin{enumerate}
\item if ${X}^e_k\neq Y^e_k$, then $(X^e,Z^{X^e})$ and $(Y^e,Z^{Y^e})$ move independently according their respective original laws;
\item if ${X}^e_k=Y^e_k=x$, with $x\neq x_e$, then  $({X}^e_{k+1}k,Z^{X^e}_{k+1})=(Y^e_{k+1},Z^{Y^e}_{k+1})$ almost surely and $({X}^e_{k+1},Z^{X^e}_{k+1})=(y,z')$ with probability
\[
P^\omega_{(x,z)}[X_1=(y,z')];
\]
\item if ${X}^e_k=Y^e_k=x_e$ and $X^e_{k-1}=y'\sim x_e$, then 

\begin{enumerate}
\item[(a)] $({X}^e_{k+1},Z^{X^e}_{k+1})=(Y^e_{k+1},Z^{Y^e}_{k+1})=(y,0)$ with probability
\[
p_{y',y}:=P^{\omega}_{e_{y'}}[X_{T^{\text{ex}}_e}=y]\wedge P^{\omega}_{x_e}[Y^e_{k+1}=(y,0)];
\]
\item[(b)] $(({X}^e_{k+1},Z^{X^e}_{k+1}),(Y^e_{k+1},Z^{Y^e}_{k+1}))=((y_1,0),(y_2,0))$ with probability
\[
\frac{\left(P^{\omega}_{e_{y'}}[X_{T^{\text{ex}}_e}=y_1]-p_{y',y_1}\right)\left(P^{\omega}_{x_e}[Y^e_{k+1}=(y_2,0)]-p_{y',y_2}\right)}{1-\sum_{y\sim x_e}p_{y',y}}.
\]
\end{enumerate}
\end{enumerate}

With this coupling, the marginal laws of $X^e$ and $Y^e$ are their original laws. Moreover, as long as the two walks are coupled, if they are not on $x_e$ then they stay coupled for at least one more step almost surely, and if they are on $x_e$ (and if they were on $y'$ one time step before), they decouple with probability
\[
1-\sum_{y\sim x_e}p_{y',y}\le 2(2d-1)\max_{y,y'\sim x_e}\left|P^{\omega}_{e_{y'}}[X_{T^{\text{ex}}_e}=y]- P^{\omega}_{x_e}[Y^e_{k+1}=(y,0)]\right|.
\]

Hence, it is sufficient to show that the transition probabilities at $x_e$ are close for $({X}^e_k,Z^{X^e}_k)$ and $(Y^e_k,Z^{Y^e}_k)$.

In particular, using \eqref{bosselarecherche}, \eqref{bosselarecherche2} and Lemma~\ref{exit_pro}, we see that if we assume that $c^{\omega}_*(e')\leq n^{\delta}$ for any $e'\sim e$, then for any $y,y' \sim x_e$, we have 
\begin{eqnarray*}
&&\left|P^{\omega}_{e_{y'}}[X_{T^{\text{ex}}_e}=y]- P^{\omega}_{x_e}[Y^e_{k+1}=(y,0)]\right| \leq C(d)n^{\delta-1}.
\end{eqnarray*}

Until $T_{e}$ we can keep $({X}^e_k,Z^{X^e}_k)$ and $(Y^e_k,Z^{Y^e}_k)$ coupled with probability 1. After that  at every point in time, we can keep the walks coupled except with probability at most $C(d)n^{\delta-1}$, it is clear that we can keep $({X}^e_k,Z^{X^e}_k)$ and $(Y^e_k,Z^{Y^e}_k)$ coupled for $n^{2\delta}$ units of time with probability at least $1-C(d)n^{3\delta-1}$.

\end{proof}

\subsection{The coupling of regeneration times of the new processes}

Let us define new regeneration times ${\tau}^{Y^e}_1$ in $\Z^d_e$ associated to $({Y}^e_k,Z^{Y^e}_k)$. This is done in the same way we defined $\tau_1$ except that
\begin{enumerate}
\item we define ${D}^{Y^e}$ as in \eqref{defD}, using the convention that $x_e\cdot \vec{\ell}=e_+\cdot \vec{\ell} \wedge e_-\cdot \vec{\ell}$,
\item we define ${M}^{Y^e}_{k+1}$ as in \eqref{defrealM}, using the convention that $x_e\cdot \vec{\ell}=e_+\cdot \vec{ \ell} \vee e_-\cdot \vec{\ell}$,
\item we replace the definition \eqref{def_Mcal} of $\mathcal{M}$ by
\begin{eqnarray}\nonumber
\mathcal{M}^{Y^e}&:=&\inf\{i\geq2: X_{i} \mbox{ is $K$-open, } X_i,X_{i-1},X_{i-2}\neq x_e\text{, for }j< i-2, \\
&&X_j\cdot\vec{\ell} < X_{i-2} \cdot\vec{
\ell}
 \mbox{ and }X_i=X_{i-1}+e_1=X_{i-2}+2e_1
\}\label{def_McalY}.
\end{eqnarray}
\end{enumerate}
Recall that the times corresponding to $\mathcal{M}^{Y^e}$ are potential times for regeneration. The definition \eqref{def_McalY} is chosen to into account that $x_e$ corresponds to an edge of high conductance, hence a regeneration is unlikely to occur there because a back and forth crossing will occur with overwhelming probability.

Also we define $\tau_1^{X^e}$, the regeneration time associated to $(X^e_k,Z^{X^e}_k)$ which is defined as the greatest number $k$ such that $A^e_k\le \tau_1$, where $(A^e_k)$ is the time-change introduced in \eqref{defAek}. Note that if $X_{\tau_1}\notin e$, then $A^e_{\tau_1^{X^e}}=\tau_1$. Furthermore, as soon as $c_*(e)>K$, it is guaranteed that $X_{\tau_1}\notin e$, since, in this case, $e^+$ and $e^-$ are not open.
%
%

\begin{lemma}\label{coupling}
Fix any environment  $\omega$, an edge $e\in E(\Z^d)$ and $\delta\in(0,1/3)$. Assume that $c^{\omega}_*(e)\geq n>K$ and $c^{\omega}_*(e')\leq n^{\delta}$ for any $e'\sim e$.  For the coupling $\widehat{P}^{\omega}$ of $(({X}^e_k,Z^{X^e}_k),{\tau}^{X^e}_1)$ and $((Y^e_k,Z^{Y^e}_k),\tau_1^{Y^e})$ from Lemma \ref{coupling0}, we have, for any $z\in\Z^d_e$,
\begin{eqnarray*}
\widehat{P}_z^{\omega}[({X}^e_k)_{k\in\N}\neq (Y^e_k)_{k\in\N},T_{e}<\tau_1^{X^e}\le n^{2\delta}] &\leq& C(d,\lv,K)n^{3\delta-1},\\
\widehat{P}_z^{\omega}[D=\infty, D^{Y^e}<\infty,T_{e}<\tau_1^{X^e}\le n^{2\delta}] &\leq& C(d,\lv,K)n^{3\delta-1},
\end{eqnarray*}
and
\[
\widehat{P}_z^{\omega}[{\tau}^{X^e}_1\neq \tau_1^{Y^e}, T_{e}<\tau_1^{X^e}\le n^{2\delta}] \leq C(d,\lv,K)n^{3\delta-1}.
\]
\end{lemma} 
\begin{proof}
For simplicity, we choose $z=0$, which could stand for $z=x_e$ if $0\in e$.\\

By Lemma~\ref{coupling0}, we already have a coupling of $({X}^e_k,Z^{X^e}_k)$ and $({Y}^e_k,Z^{Y^e}_k)$ such that $T_e=T_{x_e}^{X^e}=T_{x_e}^{Y^e}$. We simply have to verify the three equations for that coupling. To do this, we need first to control the events on which the walks decouple. Second, it could happen that the two walks stay couple forever but $\tau^{X^e}_1\neq \tau^{Y^e}_1$, thus we also want to control these events.\\

Let us define some events:
\begin{enumerate}
\item[(A1)] $({X}^e_k,Z^{X^e}_k)$ and $({Y}^e_k,Z^{Y^e}_k)$ decouple before time $n^{2\delta}$;
\item[(A2)] $X_{\tau_1^{X}-2}$ or $X_{\tau_1^{X}-1}$ belongs to $e$ and there has not been a back and forth crossing of $e$; 
\item[(A3)] $e$ is such that $(e_--X_{\tau_1})\cdot\lv\le0$ and $(e_+-X_{\tau_1})\cdot\lv>0$ (or switching $e_+$ and $e_-$) and there has not been a crossing of $e$ right after $T_e<\infty$;
\item[(A4)] $e$ is such that one of its ends is a neighbour of $0$, $e_+$ say, $T_e<\infty$ and either $T_{e_+}=\infty$ or $Z_{T_{e_+}+1}=1$ and $X_{T_{e_+}+1}\neq 0$;
\item[(A5)] $e$ is such that $e_-\cdot\lv\le0$ and $e_+\cdot\lv>0$ (or switching $e_+$ and $e_-$) and there has not been a crossing of $e$ right after $T_e<\infty$.
\end{enumerate}
Recall that, under the coupling of Lemma \ref{coupling0}, $X^e$ and $Y^e$ can decouple only at times when they are on $x_e$.\\
Firstly, on $T_{e}<\tau_1^{X^e}\le n^{2\delta}$, if $({X}^e_k,Z^{X^e}_k)$ and $({Y}^e_k,Z^{Y^e}_k)$ decouple at some finite time then we are either in the situation $(A1)$ or in $(A1)^c$ and it means that $X$ goes back to $e$ after $\tau_1$ and we are in the situation $(A3)$.\\
Secondly, on $T_{e}<\tau_1^{X^e}\le n^{2\delta}$, if the two walks remain coupled forever, and if $D=\infty$ and $D^{Y^e}<\infty$, then we are in the situation  $(A4)$ or $(A5)$.\\
Thirdly, on $T_{e}<\tau_1^{X^e}\le n^{2\delta}$, if the two walks remain coupled forever, and if $\tau_1^{Y^e}\neq {\tau}^{X^e}_1$ then we are either in the situation  $(A2)$ or $(A3)$.\\

By Lemma \ref{coupling0} and using Markov's property at time $T_e$, we have
\begin{eqnarray*}
\widehat{P}_0^\omega\left[(A1)\cup(A2)\cup(A3)\cup(A4)\cup(A5), T_{e}<\tau_1^{X^e}\le n^{2\delta}\right]&\le& C(d,\lv)n^{3\delta-1},
\end{eqnarray*}
where, recalling \eqref{defpjK}, we used for $(A4)$ that ${P}^\omega_{e_+}[Z_1=1]\le C(d,\lv)K^{-1}/c_*(e)$. Now, the conclusion follows easily.

\end{proof}

\subsection{Tail estimates on the new regeneration times}

In this section we prove some estimates on the tail of the newly introduced regeneration times.

\subsubsection{Regeneration times for $Y^e$ cannot be very large  when a large conductance is met }

We start by proving the following technical lemma.
\begin{lemma}\label{close_tau}
Fix $e\in E(\Z^d)$ and $\delta\in(0,1/3)$.   We have, for $n>K$,
\begin{eqnarray*}
\widehat{\PR}_0^K[{\tau}^{X^e}_1\neq \tau_1^{Y^e},\ \ T_{e}<\tau_1^{X^e}\le n^{2\delta},\ c_*(e)\ge n, SLT(\delta,K,n)^c]\\
\leq  C(d,\lv,K)n^{3\delta-1}{\bf P}[c_*(e)\ge n],
\end{eqnarray*}
where $\widehat{\PR}^K$ is the coupling of $(({X}^e_k,Z^{X^e}_k),{\tau}^{X^e}_1)$ and $((Y^e_k,Z^{Y^e}_k),\tau_1^{Y^e})$ from Lemma \ref{coupling0}.
\end{lemma}
\begin{proof}
For an environment $\omega$, recall that, under $\widehat{P}^\omega$, $T_e=T_{x_e}^{X^e}=T_{x_e}^{Y^e}$ by Lemma \ref{coupling0}. On $\{T_{e}<\tau_1^{X^e}\le n^{2\delta}\}\cap\{c_*(e)\ge n\}\cap SLT(\delta,K,n)^c$, we know that the hypothesis of Lemma~\ref{coupling} is verified  as, by Remark \ref{elodie}, there exists no other edge $e'\sim e$ such that $c_*^{\omega}(e')\geq n^{\delta}$.

Hence, by applying Lemma~\ref{coupling}
\begin{eqnarray*}
&&{\bf E}_0^K\left[\1{c_*(e)\ge n}P^{\omega}_0\left[T_{e}<\tau_1^{X^e}\le n^{2\delta}, SLT(\delta,K,n)^c,\ {\tau}^{X^e}_1\neq \tau_1^{Y^e}\right]\right] \\
&&\le {\bf E}_0^K\left[\1{c_*(e)\ge n, c_*(e')< n^\delta, \forall e'\sim e}P^{\omega}_0\left[T_{e}<\tau_1^{X^e}\le n^{2\delta}, {\tau}^{X^e}_1\neq \tau_1^{Y^e}\right]\right]\\
&&\leq C(d,\lv,K)n^{3\delta-1}{\bf P}[c_*(e)\ge n].
\end{eqnarray*}
%


\end{proof}




We proceed to prove another intermediate result.
\begin{lemma}\label{small_time2}
Fix $e\in\Z^d$. For any $\delta\in(0,1/(\gamma+3))$, there exists $K_0<\infty$ such that, for all $K\ge K_0$ and any $n>0$,
\[
\widehat{\PR}^K_0[\left.\tau_1^{Y^e}\geq n^{2\delta}, c_*(e)\ge n, SLT(n,K,\delta)^c, T_{e}<\tau_1\right|D=\infty]\leq Cn^{-\gamma\delta}{\bf P}[c_*\geq n],
\]
where $\widehat{\PR}^K$ is the coupling of $(({X}^e_k,Z^{X^e}_k),{\tau}^{X^e}_1)$ and $((Y^e_k,Z^{Y^e}_k),\tau_1^{Y^e})$ from Lemma \ref{coupling0}.
\end{lemma}

\begin{proof}
Firstly, recall that if $n>K$ then $e$ is closed, $X_{\tau_n}\notin e$ and  $\{T_{e}<\tau_1\}=\{T_{e}^{X^e}<\tau_1^{X^e}\}$. Using Lemma \ref{close_tau}, we have
\begin{eqnarray*}
&&\widehat{\PR}^K_0[\left.\tau_1^{Y^e}\geq n^{2\delta}, c_*(e)\ge n, SLT(n,K,\delta)^c, T_{e}<\tau_1, \tau_1^{X^e}\le n^{2\delta}\right|D=\infty]\\
&&\le C(K)\widehat{\PR}_0^K[{\tau}^{X^e}_1\neq \tau_1^{Y^e},\ \ T_{e}<\tau_1^{X^e}\le n^{2\delta},\ c_*(e)\ge n, SLT(\delta,K,n)^c]\\
&&\leq  C(d,\lv,K)n^{3\delta-1}{\bf P}[c_*(e)\ge n]. 
\end{eqnarray*}

Secondly, as $\tau_1^{X^e}\le\tau_1$ by the time-change \eqref{defAek}, Lemma \ref{decaytau2} yields
\begin{eqnarray}\nonumber
&&\widehat{\PR}^K_0\left[\left. c_*(e)\ge n, SLT(n,K,\delta)^c,  \tau_1^{X^e}\ge n^{2\delta}\right|D=\infty\right]\\ \nonumber
&&\le {\PR}_0^K\left[\left.{\tau}^{\le n^\delta}_1\ge n^{2\delta},\ c_*(e)\ge n\right|D=\infty\right]\\ \label{dodo}
&&\leq  C(d,\lv,K,\delta)n^{-\gamma\delta}{\bf P}[c_*(e)\ge n],
\end{eqnarray}
where ${\tau}^{\le n^\delta}_1$, defined in  \eqref{def_tauinf}, is the time spend before $\tau_1$ on the edges with a conductance less than $n^\delta$.

%
%
\end{proof}

We now prove the main result of this section.
\begin{lemma}\label{tail_regen_new}
Fix $e\in E(\Z^d)$ and fix $\delta\in(0,1/(\gamma+3))$, there exists $K_0<\infty$ such that for any $K\ge K_0$, there exists $\epsilon'>0$ such that
\[
\widehat{\PR}_0^K\left[\left.\tau_1^{Y^e}\geq n^{2\delta}, T_{e}<\tau_1,c_*(e)\ge n\right|D=\infty\right]\leq Cn^{-\epsilon'} {\bf P}[c_*\geq n]=o(\overline{\PR}[LT(n)]),
\]
and
\[
\widehat{\PR}_0^K\left[\left.{\tau}^{X^e}_1\geq n^{2\delta}, T_{e}<\tau_1,c_*(e)\ge n\right|D=\infty\right]\leq Cn^{-\epsilon'} {\bf P}[c_*\geq n]=o(\overline{\PR}[LT(n)]),
\]
where $\widehat{\PR}^K$ is the coupling of $(({X}^e_k,Z^{X^e}_k),{\tau}^{X^e}_1)$ and $((Y^e_k,Z^{Y^e}_k),\tau_1^{Y^e})$ from Lemma \ref{coupling0}.
\end{lemma}
\begin{proof}
For the first inequality, we just need to use Lemma~\ref{small_time2} together with Lemma \ref{not_2_trapsb} and Lemma \ref{LB_pn}.
The second inequality is given by \eqref{dodo} together with Lemma \ref{not_2_trapsb} and Lemma \ref{LB_pn}.
\end{proof}

\subsubsection{Backtracking probabilities for the random walk $Y^e$}

Recall, from Section \ref{iwontdo}, the definition of ${\bf P}^e$, the law of the environment $\omega_e$ where the edge $e\in E(\Z^d)$ is collapsed. Note that, in order to define $\omega_e$, as the environment outside $e$ is independent of $c^{\omega}(e)$,  it is equivalent to pick $\omega$ under ${\bf P}[\cdot|e\text{ is closed}]$ and then to collapse the edge $e$.\\
Besides, recall that we extended, in Section \ref{sectwalkcollapsed}, the definition of {\it good} and {\it bad} vertices to the environment $\omega_e$ keeping the exact same definition except that the vertex $x_e$ is always considered as closed, whatever are the conductances of the surrounding edges. In words, a vertex $x\in\Z^d_e$ is $\omega_e$-good if there exists an infinite directed open path starting from $x$ (hence, this path does not go through $x_e$).\\
Now, notice that for ${\bf P}[\cdot|e\text{ is closed}]$-a.e.~environment $\omega$, the vertices $e^+$ and $e^-$ are closed. Thus, under ${\bf P}[\cdot|e\text{ is closed}]$, this path does not go through $e^+$ or $e^-$. This means that, for ${\bf P}[\cdot|e\text{ is closed}]$-a.e.~environment $\omega$, a vertex is good in $\omega$ if and only if it is good in $\omega_e$. In other words, 
\begin{equation}\label{eleanorrigby}
\mathrm{GOOD}(\omega_e)=\mathrm{GOOD}(\omega)\qquad {\bf P}[\cdot|e\text{ is closed}]\text{-a.s.,}
\end{equation}
where each vertex $x\in\Z^d_e\setminus\{x_e\}$ is naturally associated to a unique vertex in $\Z^d\setminus\{e^+,e^-\}$.\\
Also, under ${\bf P}[\cdot|e\text{ is closed}]$,  $\mathrm{BAD}(\omega_e)\setminus\{x_e\}=\mathrm{BAD}(\omega)\setminus\{e^+,e^-\}$ a.s., $x_e\in\mathrm{BAD}(\omega_e)$ and $e^+,e^-\in\mathrm{BAD}(\omega)$. Moreover, for any $x\neq x_e$, we have $\mathrm{BAD}^{\omega_e}(x)\setminus\{x_e\}=\mathrm{BAD}^{\omega}(x)\setminus\{e^+,e^-\}$ and $x_e\in\mathrm{BAD}^{\omega_e}(x_e)$ if and only if $e^+,e^-\in\mathrm{BAD}^{\omega}(e^+)$. Besides, $\mathrm{BAD}^{\omega_e}(x_e)\setminus\{x_e\}=\mathrm{BAD}^{\omega}(e^+)\setminus\{e^+,e^-\}$. Then, we extend the definition of the width $W(\cdot)$ such that
\begin{align*}
W(\mathrm{BAD}^{\omega_e}(x))&=W(\mathrm{BAD}^{\omega}(x)),\text{ for any }x\neq x_e,\\
W(\mathrm{BAD}^{\omega_e}(x_e))&=W(\mathrm{BAD}^{\omega}(e^+))
\end{align*}

The two following results are the analog of Lemma \ref{backbone} and Lemma \ref{thisisboring} for the walk $Y^e$.

\begin{lemma}\label{backboneY}
Fix some edge $e\in E(\Z^d)$ and some environment $\omega_e$ on $\Z^d_e$. For any $x\in\mathrm{GOOD}(\omega_e)$, we have
\[
E^{\omega_e}_x \Biggl[\sum_{i=0}^{\infty}
{\mathbf1} {\{Y^e_i=x\}} \Biggr]\leq C(K)<\infty.
\]
\end{lemma}
\begin{proof}
This proof is identical to the proof of Lemma 8.1 in \cite{Fri11} but, as it is short, we give it again. First, we have that
\[
E^{\omega_e}_x \Biggl[\sum_{i=0}^{\infty}{\mathbf1} {\{Y^e_i=x\}} \Biggr]=\frac{1}{P^{\omega_e}[T_x^+=\infty]}=\frac{\pi^{\omega_e}(x)}{C^{\omega_e}(x\leftrightarrow \infty)},
\]
where $C^{\omega_e}(x\leftrightarrow \infty)$ is the effective conductance between $x$ and infinity in $\omega_e$. Since $x\in\mathrm{GOOD}(\omega_e)$, using Remark \ref{fakeUE}, we have
\[
\pi^{\omega_e}(x)\le C(K)e^{2x\cdot \ell}
\]
Moreover, using Rayleigh's Monotonicity Principle (see \cite{LP}) and Remark \ref{remopenkey}, we obtain that
\[
C^{\omega_e}(x\leftrightarrow \infty)\ge \frac{1}{\sum_{i\ge0} \frac{1}{c^{\omega_e}(g_i)}}\ge  \frac{C(K)e^{2x\cdot \ell}}{\sum_{i\ge0} e^{-c(d)i}}\ge C(K)e^{2x\cdot \ell},
\]
where $(g_i)_{i\ge0}$ is the sequence of the edges of an infinite directed open path starting at $x$. The result now follows easily.
\end{proof}

For any $e\in E(\Z^d)$, we define
\begin{equation*}
d_e:=e^+\cdot \vec{\ell}\wedge e^-\cdot \vec{\ell}.
\end{equation*}

Let us define, for any $k\ge0$, the following half-space in $\Z^d_e$:
\[
\Hc_e^-(-k)=\left\{x\in\Z^d_e\setminus\{x_e\}: x\cdot \lv-d_e\le -k\right\}.
\]
We also naturally extend the definition of $\Hc_e^-(-k)$ in $\Z^d$.

\begin{lemma}\label{thisisboringY}
Fix some edge $e\in E(\Z^d)$. We have, for any $k\ge 0$,
\[
\PR^e_{x_e}[T^{Y^e}_{\mathcal{H}_e^-(-k)}<\infty]\leq C\exp(-ck).
\]
\end{lemma}
\begin{proof}
Recall that the law ${\bf P}^e$ of $\omega_e$ is given by the law of $\omega$ under ${\bf P}[\cdot|e\text{ is closed}]$ where $e$ is collapsed into $x_e$, and recall the equality \eqref{eleanorrigby}.
We then have, for any constant $c>0$, as soon as $K$ is large enough,
\begin{align}\nonumber
{\bf P}^e\left[W\left(\mathrm{BAD}^{\omega_e}(x_e)\right)\ge \frac{k}{8\sqrt{d}}\right]&= {\bf P}\left[\left.W\left(\mathrm{BAD}^{\omega}(e^+)\right)\ge \frac{k}{8\sqrt{d}}\right| e\text{ is closed}\right]\\ \label{loooser1}
&\le C(K) \exp(-ck),
\end{align}
where we used Lemma \ref{BLsizeclosedbox}.\\
Now, when $Y^e$ is started from $x_e$, we have the following inclusion:
\begin{align*}
&\left\{T^{Y^e}_{\mathcal{H}_e^-(-k)}<\infty\right\}\cap\left\{W\left(\mathrm{BAD}^{\omega_e}(x_e)\right)< \frac{k}{8\sqrt{d}}\right\}\subset A,
\end{align*}
where
\begin{align}\nonumber
&A:=\left\{  \exists i\in\N \text{ s.t. }X_i\in\partial\mathrm{BAD}^{\omega_e}(x_e),\  T^{Y^e}_{\mathcal{H}_e^-(-k)}\circ \theta_i < T^{Y^e}_{\mathrm{BAD}^{\omega_e}(x_e)}\circ \theta_i\right\}\\ \label{loooser2}
& \qquad \qquad \cap\left\{W\left(\mathrm{BAD}^{\omega_e}(x_e)\right)< \frac{k}{8\sqrt{d}}\right\}.
\end{align}
Besides, recall that $\partial\mathrm{BAD}^{\omega_e}(x_e)\subset \mathrm{GOOD}(\omega_e)$ and that $x_e\notin  \mathrm{GOOD}(\omega_e)$ by definition. We then have
\begin{align*}
\PR^e_{x_e}[A]\le& \ES^e_{x_e}\left[\1{W\left(\mathrm{BAD}^{\omega_e}(x_e)\right)< \frac{k}{8\sqrt{d}}}\right.\\
&\left.\sum_{i\ge0}\sum_{x\in\partial\mathrm{BAD}^{\omega_e}(x_e)}\1{X_i=x,\  T^{Y^e}_{\mathcal{H}_e^-(-k)}\circ \theta_i < T^{Y^e}_{\mathrm{BAD}^{\omega_e}(x_e)}\circ \theta_i}\right]\\
\le& \sum_{x\in B_\infty(e^+,k/8\sqrt{d})\setminus\{x_e\}}\ES^e_{x_e}\left[\1{x\in\mathrm{GOOD}(\omega_e)}\sum_{i\ge0}\1{X_i=x}  \right.\\
&\left.\qquad \qquad \qquad \qquad \qquad \qquad \times\1{  T^{Y^e}_{\mathcal{H}_e^-(-k)}\circ \theta_i < T^{Y^e}_{\mathrm{BAD}^{\omega_e}(x_e)}\circ \theta_i } \right]\\
\le& \sum_{x\in B_\infty(e^+,k/8\sqrt{d})\setminus\{x_e\}}\ES^e_{x_e}\left[\1{x\in\mathrm{GOOD}(\omega_e)}\sum_{i\ge0}\1{X_i=x}  \right.\\
&\left.\qquad \qquad \qquad \qquad \qquad \qquad \times P_x^{\omega_e}\left[{  T^{Y^e}_{\mathcal{H}_e^-(-k)} < T^{Y^e}_{\mathrm{BAD}^{\omega_e}(x_e)} } \right]\right]\\
\le& \sum_{x\in B_\infty(e^+,k/8\sqrt{d})\setminus\{x_e\}}C(K)\ES^e_{x_e}\left[  P_x^{\omega_e}\left[{  T^{Y^e}_{\mathcal{H}_e^-(-k)} < T^{Y^e}_{\mathrm{BAD}^{\omega_e}(x_e)} } \right]\right],
\end{align*}
where we used Markov's property and Lemma \ref{backboneY}. Moreover, notice that for ${\bf P}[\cdot|e\text{ is closed}]$-a.e.~environment $\omega$, $\omega_e$ coincides with $\omega$ outside $x_e$, and the transition probabilities of $Y^e$ outside $x_e$ are equal to those of $X$ outside $\{e^-,e^+\}$, hence, for any $x\neq x_e$,
\[
P_x^{\omega_e}\left[{  T^{Y^e}_{\mathcal{H}_e^-(-k)} < T^{Y^e}_{\mathrm{BAD}^{\omega_e}(x_e)} } \right]=P_x^{\omega}\left[{  T^{X}_{\mathcal{H}_e^-(-k)} < T^{X}_{\mathrm{BAD}^{\omega_e}(e^+)} } \right].
\]

Therefore, using Lemma \ref{thisisboring}, we have
\begin{align*}
\PR^e_{x_e}[A] \le& \sum_{x\in B_\infty(e^+,k/8\sqrt{d})}C(K)\PR_x\left[  \left.{  T^{X}_{\mathcal{H}_e^-(-k)} < T^{X}_{\mathrm{BAD}^{\omega_e}(e^+)} }  \right|e\text{ is closed}\right]\\
 \le& \sum_{x\in B_\infty(e^+,k/8\sqrt{d})}C(K)\PR_x\left[  {  T^{X}_{\mathcal{H}_e^-(-k)} < \infty}  \right]\\
 \le& C(K)k^d\PR_0\left[  {  T^{X}_{\mathcal{H}^-(-k/2)} < \infty}  \right]
\\
\le& C(K)\exp(-ck).
\end{align*}
This implies the conclusion, together with \eqref{loooser1} and \eqref{loooser2}.
\end{proof}

The next lemma deals with the original walk $X$ and improves Lemma  \ref{thisisboring}.

\begin{lemma} \label{thisisboring+}
Fix an edge $e\in E(\Z^d)$. For any $n\ge0$ and any $k\ge0$, we have
\[
\max_{z\in\{e^+,e^-\}} \PR_z\left[T^{X}_{\mathcal{H}_e^-(-k)}<\infty,c_*(e)\ge n\right]\le C\exp(-ck){\bf P}[c_*(e)\ge n].
\]
\end{lemma}

\begin{proof}
This proof is similar to the previous one. Fix $z\in\{e^+,e^-\}$.\\
Recall that the sets $\mathrm{BAD}(x)=\mathrm{BAD}_K^\omega(x)$ depend on the value of some constant $K$. Fix some constant $c>0$. There exists a constant $K_0<\infty$ such that for any $K\ge K_0$ and any $n>K$, we have
\begin{align*}
&{\bf P}\left[\mathrm{BAD}(e^+)\cup\mathrm{BAD}(e^-)\ge k/8\sqrt{d},c_*(e)\ge n\right]\\
=&{\bf P}\left[\left.\mathrm{BAD}(e^+)\ge k/8\sqrt{d}\right|c_*(e)\ge n\right]{\bf P}\left[c_*(e)\ge n\right]\\
=&{\bf P}\left[\left.\mathrm{BAD}(e^+)\ge k/8\sqrt{d}\right|e\text{ is closed}\right]{\bf P}\left[c_*(e)\ge n\right]\\
\le& C(K){\bf P}\left[\mathrm{BAD}(e^+)\ge k/8\sqrt{d}\right]{\bf P}\left[c_*(e)\ge n\right]\\
\le&C(K)\exp(-ck){\bf P}\left[c_*(e)\ge n\right],
\end{align*}
where we used Lemma \ref{BLsizeclosedbox}.\\
Now, in a way that is very similar to the proof of Lemma \ref{thisisboringY}, we obtain
\begin{align*}
&\PR_z\left[T^{X}_{\mathcal{H}_e^-(-k)}<\infty,c_*(e)\ge n, \mathrm{BAD}(e^+)\cup\mathrm{BAD}(e^-)< k/8\sqrt{d}\right]\\
\le& \sum_{x\in B_\infty(e^+,k/8\sqrt{d})}C(K)\ES_{z}\left[  P_x^{\omega}\left[{  T^{X}_{\mathcal{H}_e^-(-k)} < T^{X}_{\mathrm{BAD}(e^+)} } \right]\1{c_*(e)\ge  n}\right].
\end{align*}
Moreover,  consider the environment $\widetilde{\omega}$ which coincides with $\omega$ everywhere expect on $e$ for which we independently resample a conductance $\tilde{c}_*(e)$ under ${\bf P}[\cdot|e\text{ is closed}]$. We then have, as soon as $c_*(e)\ge n>K$,
\[
P_x^{\omega}\left[{  T^{X}_{\mathcal{H}_e^-(-k)} < T^{X}_{\mathrm{BAD}(e^+)} } \right]=P_x^{\widetilde{\omega}}\left[{  T^{X}_{\mathcal{H}_e^-(-k)} < T^{X}_{\mathrm{BAD}(e^+)} } \right],
\]
and this last quantity is independent of $c_*(e)$. This yields, using Lemma \ref{thisisboring},
\begin{align*}
&\PR_z\left[T^{X}_{\mathcal{H}_e^-(-k)}<\infty,c_*(e)\ge n, \mathrm{BAD}(e^+)\cup\mathrm{BAD}(e^-)< k/8\sqrt{d}\right]\\
\le& \sum_{x\in B_\infty(e^+,k/8\sqrt{d})}C(K)\ES_{z}\left[  P_x^{\omega}\left[\left.{  T^{X}_{\mathcal{H}_e^-(-k)} < T^{X}_{\mathrm{BAD}(e^+)} } \right]\right|e\text{ is closed}\right]{\bf P}\left[{c_*(e)\ge  n}\right],
\end{align*}
where we used the fact that $P_x^{\omega}\left[ T^{X}_{\mathcal{H}_e^-(-k)} < T^{X}_{\mathrm{BAD}(e^+)}  \right]$ depends on the value of $c_*(e)$ only through the fact that the edge $e$ is closed. In turn we obtain that 
\begin{align*}
&\PR_z\left[T^{X}_{\mathcal{H}_e^-(-k)}<\infty,c_*(e)\ge n, \mathrm{BAD}(e^+)\cup\mathrm{BAD}(e^-)< k/8\sqrt{d}\right]\\
\le& C(K)k^d\PR_{0}\left[  T^{X}_{\mathcal{H}^-(-k/2)} < \infty \right]{\bf P}\left[{c_*(e)\ge  n}\right]\le  C(K)\exp(-ck){\bf P}\left[{c_*(e)\ge  n}\right].
\end{align*}
This enables us to conclude.
\end{proof}

\subsubsection{Probability of reaching $x_e$ before regeneration for the random walk $Y^e$}

We need some control on the probability for the walk $Y^e$ to touch $x_e$ before the first regeneration time when this vertex is far away.
\begin{lemma}\label{tail_taue1}
Fix $e\in E(\Z^d)$. For any $M\in(0,+\infty)$, there exists $K_0<\infty$ such that, for any $K\ge K_0$,
\[
\PR_0^K[T^{Y^e}_{x_e} \leq \tau_1^{Y^e}]  \leq  C(K) \abs{\abs{e}}_{\infty}^{-M}.
\]
\end{lemma}
\begin{proof}
We fix $e$ such that $||e||_\infty\ge 8$ for convenience. We can notice that, under the coupling $\widehat{P}^\omega$ from Lemma \ref{coupling0}, $(Y^e_n)_{n\in \N}$ and $(X^e_n)_{n\in \N}$ are coupled with probability $1$ up to $T_e=T^{Y^e}_{x_e}$. This allows us to say that, on $\{T^{Y^e}_{x_e} \leq \tau_1^{Y^e}\}$, we have either $T_{B_{\infty}^c(0,\abs{\abs{e}}_{\infty}/3)}\le\tau_1$,
 or $T_{B_{\infty}^c(0,\abs{\abs{e}}_{\infty}/3)}>\tau_1$ and  $\tau_1^{Y^e} \geq  T^{Y^e}_{x_e}\ge T_{B^c_{\infty}(0,\abs{\abs{e}}_{\infty}/3)}=T^{Y^e}_{B^c_{\infty}(0,\abs{\abs{e}}_{\infty}/3)}$. Furthermore $T_{B_{\infty}^c(0,\abs{\abs{e}}_{\infty}/3)}>\tau_1$ implies that $Y^e_{\tau_1}$ is a potential regeneration point for $Y^e_n$ up to time $T^{Y^e}_e$ and since  $\tau_1 \neq \tau_1^{Y^e}$ we know $Y^e_{\tau_1}$ is not a regeneration point which means that after $T^{Y^e}_e$ the walk $Y^e$ has to backtrack in $\mathcal{H}^-((Y^e_{\tau_1}+e_1)\cdot \vec{\ell}) \subset\mathcal{H}^-(\abs{\abs{e}}_{\infty}/2)$. This means that there exists $y\in B_{\infty}(0,\abs{\abs{e}}_{\infty}/2)\cap\Hc^+_0$ (corresponding to $Y^e_{\tau_1}$ and where $\Hc^+_0$ is defined in \eqref{johnlennonstar}) such that $T^{Y^e}_y<\infty$, $T^{Y^e}_{x_e}\circ \theta_{T^{Y^e}_y}<\infty $ and $T^{Y^e}_{\mathcal{H}^-((y+e_1)\cdot \vec{\ell})}\circ \theta_{T^{Y^e}_{x_e}}<\infty$.
To sum up we either have
\begin{enumerate}
\item $T_{B_{\infty}^c(0,\abs{\abs{e}}_{\infty}/3)}\le\tau_1$, or
\item there exists $y\in B_{\infty}(0,\abs{\abs{e}}_{\infty}/2)\cap\Hc^+_0$  such that $T^{Y^e}_y<\infty$, $T^{Y^e}_{x_e}\circ \theta_{T^{Y^e}_y}<\infty $ and $T^{Y^e}_{\mathcal{H}^-((y+e_1)\cdot \vec{\ell})}\circ \theta_{T^{Y^e}_{x_e}}\circ \theta_{T^{Y^e}_y}<\infty$,
\end{enumerate}
so we obtain, recalling Definition \ref{defP0K} of $\PR_y^{0,K}$,
\begin{align*}
& \PR_0^K[T^{Y^e}_{x_e} \leq \tau_1^{Y^e}]
\leq  \PR_0^K[T^X_{B_{\infty}^c(0,\abs{\abs{e}}_{\infty}/3)}\le\tau_1]\\
& \qquad \qquad +\sum_{y\in B_{\infty}(0,\abs{\abs{e}}_{\infty}/2)\cap\Hc^+_0}\PR^{0,K}_y[T^{Y^e}_{x_e}<\infty, \ T^{Y^e}_{\mathcal{H}^-(y\cdot \vec{\ell})}\circ \theta_{T^{Y^e}_{x_e}}<\infty].
\end{align*}

Using Markov's property  we can see that
\begin{align}\label{awe}
 \PR_0^K[T^{Y^e}_{x_e} \leq \tau_1^{Y^e}] & \leq \PR_0^K[T^X_{B_{\infty}^c(0,\abs{\abs{e}}_{\infty}/3)}\le\tau_1] \\  \nonumber
&  +C\abs{\abs{e}}_{\infty}^d\max_{y\in B_{\infty}(0,\abs{\abs{e}}_{\infty}/2)\cap\Hc^+_0}\Bigl(\PR^{0,K}_y[T^X_e<\infty]\wedge \PR^{0,K}_{x_e}[T^{Y^e}_{\mathcal{H}^-(y\cdot \vec \ell)}<\infty] \Bigr) \\ \nonumber
& = \PR_0^K[T^X_{B_{\infty}^c(0,\abs{\abs{e}}_{\infty}/3)}<\tau_1] \\  \nonumber
&  +C\abs{\abs{e}}_{\infty}^d\max_{y\in B_{\infty}(0,\abs{\abs{e}}_{\infty}/2)\cap\Hc^+_0}\Bigl(\PR^{0,K}_y[T^X_e<\infty]\wedge \PR_{x_e}[T^{Y^e}_{\mathcal{H}^-(y\cdot \vec \ell)}<\infty] \Bigr),
\end{align}
where   $\mathcal{H}^-(y\cdot \vec \ell)$ is defined in \eqref{johnlennonstar}.\\

By Lemma~\ref{tail_chi_tau}, we can see that for any $M<\infty$, there exists $K_0<\infty$ such that, for any $K\ge K_0$,
\begin{equation}\label{awe1}
\PR_0^K[T_{B_{\infty}^c(0,\abs{\abs{e}}_{\infty}/3)}\le\tau_1] \leq C(K) \abs{\abs{e}}_{\infty}^{-M}.
\end{equation}

Moreover, notice that for $\alpha>d+3$,
\begin{enumerate}
\item if $\abs{(y-e^+)\cdot \vec{\ell}}\geq \abs{\abs{e}}^{1/\alpha}/8$ then  by Lemma \ref{thisisboring} and Lemma~\ref{thisisboringY} we have 
\begin{align*}
\PR^{0,K}_y[T^X_e<\infty]\wedge\PR_{x_e}[T^{Y^e}_{\mathcal{H}^-(y\cdot \vec \ell)}<\infty] &\leq \PR_0[T^{X}_{\mathcal{H}^-(-\frac{\abs{\abs{e}}^{1/\alpha}}{16})}<\infty]+\PR_{x_e}[T^{Y^e}_{\mathcal{H}^-(-\frac{\abs{\abs{e}}^{1/\alpha}}{16})}<\infty]  \\
&\leq C \exp(-c\abs{\abs{e}}_{\infty}^{-1/\alpha}).
 \end{align*}
\item otherwise for $y\in B_{\infty}(0,\abs{\abs{e}}_{\infty}/2)\cap\Hc^+_0$, as $\alpha>1$ and by the triangle inequality, we know that $d_{\infty}(y,e)\geq \abs{\abs{e}}_{\infty}/2$. Furthermore, notice that, as $\abs{(y-e^+)\cdot \vec{\ell}}< \abs{\abs{e}}^{1/\alpha}/8$, we have $0,e^+,e^-\notin  B_y(\abs{\abs{e}}^{1/\alpha}/2,\abs{\abs{e}}/2^\alpha)$, hence 
\begin{align*}
&\PR^{0,K}_y[T^X_{\partial B_y(\frac{\abs{\abs{e}}^{1/\alpha}}{2},\frac{\abs{\abs{e}}}{2^\alpha})}\neq T^X_{\partial^+ B_y(\frac{\abs{\abs{e}}^{1/\alpha}}{2},\frac{\abs{\abs{e}}}{2^\alpha})}]\\
=& \PR_y[T^X_{\partial B_y(\frac{\abs{\abs{e}}^{1/\alpha}}{2},\frac{\abs{\abs{e}}}{2^\alpha})}\neq T^X_{\partial^+ B_y(\frac{\abs{\abs{e}}^{1/\alpha}}{2},\frac{\abs{\abs{e}}}{2^\alpha})}].
\end{align*}
In this case 
\begin{align*}
 \PR^{0,K}_y[T^X_e<\infty]
 \leq & \PR_y[T^X_{\partial B_y(\frac{\abs{\abs{e}}^{1/\alpha}}{2},\frac{\abs{\abs{e}}}{2^\alpha})}\neq T^X_{\partial^+ B_y(\frac{\abs{\abs{e}}^{1/\alpha}}{2},\frac{\abs{\abs{e}}}{2^\alpha})}]\\
 &+\sum_{z\in \partial^+ B_y(\frac{\abs{\abs{e}}^{1/\alpha}}{2},\frac{\abs{\abs{e}}}{2^\alpha})} \PR_z[T^X_{\mathcal{H}^-(e\cdot \vec \ell\wedge0)}<\infty],
\end{align*}
noticing that for $z\in \partial^+ B_y(\frac{\abs{\abs{e}}^{1/\alpha}}{2},\frac{\abs{\abs{e}}}{2^\alpha})$ we have
\[
(z-e^+)\cdot \vec{\ell}\ge(y-e^+)\cdot \vec{\ell}+ \frac{\abs{\abs{e}}^{1/\alpha}}{2} \geq 1/4 \abs{\abs{e}}_{\infty}^{1/\alpha}.
\]
\end{enumerate}
This implies, by Lemma~\ref{tail_chi_tau} and Lemma~\ref{thisisboring},
\begin{align}\label{ended}
 &\max_{y\in B_{\infty}(0,\abs{\abs{e}}_{\infty}/2)\cap\Hc^+_0}\Bigl(\PR^{0,K}_y[T^X_e<\infty]\wedge \PR_{x_e}[T^{Y^e}_{\mathcal{H}^-(y\cdot \vec \ell)}<\infty] \Bigr)\\ \nonumber
 \leq & C  \abs{\abs{e}}_{\infty}^{d} \exp(-C\abs{\abs{e}}_{\infty}^{1/\alpha}).
 \end{align}

This last equation with~\eqref{awe} and~\eqref{awe1} implies the lemma.
\end{proof}

We will also need to control the probability to reach an edge far away, conditionally on the fact that this edge has a large conductance.

\begin{lemma}\label{tail_taue}
Fix $e\in E(\Z^d)$. For any $M<\infty$, there exists $K_0<\infty$ such that, for any $K\ge K_0$,
\[
\PR_0^K[T_{e} \leq \tau_1, c_*^{\omega}(e)\geq n]  \leq  C \abs{\abs{e}}_{\infty}^{-M} {\bf P}[c_*\geq n],
\]
where $C$ does not depend on $n$.
\end{lemma}
\begin{proof}
Let us denote $A(e)$ the event on which there exists $y \in B_{\infty}(0,\abs{\abs{e}}_{\infty}/2)$ such that, considering the trajectory up  to time $T_e$, the point $y$ is compatible with $X_{\tau_1}=y$. This means that there exists a time $n<T_e$ such that $X_n=y$ is a new maximum for the trajectory in the direction $\vec\ell$ and $X_{n+k}\cdot\vec\ell>X_n\cdot\vec\ell$ for any $1\le k\le T_e-n$. In particular, if such a vertex $y$ exists then $e\in\mathcal{H}^+(0)$ and  $y\in\mathcal{H}_{0,e}:=\mathcal{H}^+(0)\cap\mathcal{H}^-_{e^++e_1}$.\\

First let us control our event on $A(e)^c$. Let us write $\tilde{\omega}$  the environment coinciding with $\omega$ where $c_*(e)$ has been resampled according to ${\bf P}$. Notice that $A^{\omega}(e)=A^{\tilde{\omega}}(e)$ and $\tilde{\omega}$ has the same law as $\omega$.

We have 
\begin{eqnarray*}
 \PR_0^K[A(e)^c, T_{e} \leq \tau_1, c_*^{\omega}(e)\geq n]  &\leq&   \PR_0^K[A(e)^c, c_*^{\omega}(e)\geq n] \\
 &=& \PR_0^K[A(e)^c, c_*^{\tilde{\omega}}(e)\geq n]=\PR_0^K[A(e)^c] {\bf P}[c_*\geq n],
 \end{eqnarray*}
 but now we can notice that on $A(e)^c$ we necessarily have $T_{\partial  B_{\infty}(0,\abs{\abs{e}}_{\infty}/2) }<\tau_1$
 \[
  \PR_0^K[A(e)^c] 
\leq  \sum_{y \in \partial  B_{\infty}(0,\abs{\abs{e}}_{\infty}/2)} \PR_0^K[T_y<\tau_1] \leq C \abs{\abs{e}}^d_{\infty} \abs{\abs{e}}^{-M}_{\infty} ,
\]
where we used Lemma~\ref{tail_chi_tau} and we chose $K$ large enough.

Combining the last two equations we have
\begin{equation}\label{pasd}
 \PR_0^K[A(e)^c, T_{e} \leq \tau_1, c_*^{\omega}(e)\geq n]  \leq  C \abs{\abs{e}}^{-M}_{\infty}{\bf P}[c_*\geq n].
 \end{equation}
 
Now,  we can estimate $\{A(e), T_{e} \leq \tau_1\}$. On that event, we know that there exists $y \in B_{\infty}(0,\abs{\abs{e}}_{\infty}/2)\cap\mathcal{H}_{0,e}$ such that, considering the trajectory up  to time $T_e$, the point $y$ is compatible with $X_{\tau_1}=y$, but since $T_y  < T_e \leq \tau_1$, we know that $y$ is not where the regeneration occurs. This means that $T_{\mathcal{H}^-(y\cdot \vec{\ell})}\circ \theta_{T_e}<\infty$. 

This means that we obtain
\begin{align*}
& \PR_0^K[A(e), T_{e} \leq \tau_1, c_*^{\omega}(e)\geq n]
\\
&\leq  \sum_{y\in B_{\infty}(0,\abs{\abs{e}}_{\infty}/2)\cap\mathcal{H}_{0,e}}\PR_y^{0,K}[T_e\circ \theta_{T_y}<\infty \text{ and }T_{\mathcal{H}^-((y+e_1)\cdot \vec{\ell})}\circ \theta_{T_e}<\infty, c_*^{\omega}(e)\geq n].
\end{align*}

Using Markov's property we can see that
\begin{align} \label{pasds}
 \PR_0^K[A(e), T_{e} \leq \tau_1, c_*^{\omega}(e)\geq n] \le& C\abs{\abs{e}}_{\infty}^d\max_{\substack{y\in B_{\infty}(0,\abs{\abs{e}}_{\infty}/2)\\ \cap\mathcal{H}_{0,e}}}\Bigl(\PR^{0,K}_y[T^X_e<\infty, c_*^{\omega}(e)\geq n]\\ \nonumber
&\wedge \ES^{0,K}[\max_{z\in \{e^+,e^-\}}P^{\omega}_z[T^{X}_{\mathcal{H}^-((y+e_1)\cdot \vec \ell)}<\infty], c_*^{\omega}(e)\geq n] \Bigr)\\ \nonumber
\le &C\abs{\abs{e}}_{\infty}^d\max_{\substack{y\in B_{\infty}(0,\abs{\abs{e}}_{\infty}/2)\\ \cap\mathcal{H}_{0,e}}}\Bigl(\PR^{0,K}_y[T^X_e<\infty, c_*^{\omega}(e)\geq n]\\ \nonumber
&\wedge\max_{z\in \{e^+,e^-\}} \PR_z[T^{X}_{\mathcal{H}^-((y+e_1)\cdot \vec \ell\vee e_1\cdot\vec\ell)}<\infty, c_*^{\omega}(e)\geq n] \Bigr),
\end{align}
where, in the last line, we consider the hitting time of the hyperplane $\mathcal{H}^-((y+e_1)\cdot \vec \ell\vee e_1\cdot\vec\ell)$ in order to lose the dependence to the environment around the origin.

Besides, we can notice that the event $\{T^X_e<\infty,\}$ is $\PR$-independent of $c_*^{\omega}(e)$. 
Now, proceeding as for the estimate~\eqref{ended}, we obtain, for $\alpha>d+3$,
\begin{enumerate}
\item if $(e^+-y)\cdot \vec{\ell}\geq \abs{\abs{e}}^{1/\alpha}/8$ then  by  Lemma~\ref{thisisboring+} we have 
\[
\max_{z\in \{e^+,e^-\}} \PR_z[T^{X}_{\mathcal{H}^-((y+e_1)\cdot \vec \ell\vee e_1\cdot\vec\ell)}<\infty, c_*^{\omega}(e)\geq n] \leq C \exp(-c\abs{\abs{e}}_{\infty}^{-1/\alpha}){\bf P}[c_*^{\omega}(e)\geq n].
 \]
\item otherwise, for $y\in B_{\infty}(0,\abs{\abs{e}}_{\infty}/2)\cap\mathcal{H}_{0,e}$. we know that $d_{\infty}(y,e)\geq \abs{\abs{e}}_{\infty}/2$. Furthermore, notice that, as $-2\le(e^+-y)\cdot \vec{\ell}< \abs{\abs{e}}^{1/\alpha}/8$, we have $0,e^+,e^-\notin  B_y(\abs{\abs{e}}^{1/\alpha}/2,\abs{\abs{e}}/2^\alpha)$. In this case 
\begin{align*}
& \PR^{0,K}_y[T^X_e<\infty, c_*^{\omega}(e)\geq n]={\bf P}[c_*^{\omega}(e)\geq n]\PR^{0,K}_y[T^X_e<\infty]\\
 \leq &{\bf P}[c_*^{\omega}(e)\geq n]\PR_y\left[T^X_{\partial B_y(\abs{\abs{e}}^{1/\alpha}/2,\abs{\abs{e}}/2^\alpha)}\neq T^X_{\partial^+  B_y(\abs{\abs{e}}^{1/\alpha}/2,\abs{\abs{e}}/2^\alpha)}\right]\\
 &+{\bf P}[c_*^{\omega}(e)\geq n]\sum_{z\in \partial^+  B_y(\abs{\abs{e}}^{1/\alpha}/2,\abs{\abs{e}}/2^\alpha)} \PR_z[T^X_{\mathcal{H}^-((e^++e_1)\cdot \vec \ell)}<\infty],
\end{align*}
but noticing that for $z\in \partial^+  B_y(\abs{\abs{e}}^{1/\alpha}/2,\abs{\abs{e}}/2^\alpha)$ we have $(z-e^+)\cdot \vec{\ell} \geq 1/4 \abs{\abs{e}}_{\infty}^{1/\alpha}$.
\end{enumerate}
these two points imply, by Lemma~\ref{tail_chi_tau} and Lemma~\ref{thisisboring}, that
\begin{align*}
 &\max_{\substack{y\in B_{\infty}(0,\abs{\abs{e}}_{\infty}/2)\\ \cap\mathcal{H}_{0,e}}}\Bigl(\PR^{0,K}_y[T^X_e<\infty, c_*^{\omega}(e)\geq n]\wedge\max_{z\in \{e^+,e^-\}} \PR_z[T^{X}_{\mathcal{H}^-((y+e_1)\cdot \vec \ell\vee e_1\cdot\vec\ell)}<\infty, c_*^{\omega}(e)\geq n] \Bigr)\\
 \leq & C {\bf P}[c_*^{\omega}(e)\geq n] \abs{\abs{e}}_{\infty}^{d} \exp(-C\abs{\abs{e}}_{\infty}^{1/\alpha}).
 \end{align*}

Using this equation, \eqref{pasds} and~\eqref{pasd} yields the result.

\end{proof}

\subsection{Probability of events conditional on the encounter of a large trap}
In this section, we give one of the key results about the asymptotic environment seen from the particle. Indeed, Lemma \ref{big_trap_cond} provides an explicit formula for the law the walk, conditioned to meet a large trap, outside the edge with largest conductance.\\
Let us define $e^{(n)}$ as a random edge of $E(\Z^d)$ verifying 
\begin{enumerate}
\item $c_*^{\omega}(e^{(n)})\geq n$,
\item for all $i< T_e$, we have $c_*^{\omega}(e')<n$ for any $e'\in E(\Z^d)$ such that $X_i\in e'$,
\end{enumerate}
and in case of multiple possible choices we choose $e^{(n)}$ according to some predetermined order on $\Z^d$. Under $\PR_n$, defined in \eqref{defPn}, this edge is met before $\tau_1$ and, because of Proposition~\ref{not_2_traps}, there is only one possible choice for $e^{(n)}$, with high probability. Also, note that, on $ OLT_e(\delta,K,n)$ defined in \eqref{def_OLTe}, $e^{(n)}=e$.

\begin{remark}\label{rem_mes_e}
We note that, for a fixed environment $\omega$, $T_{e^{(n)}}$ is a stopping-time and the random variable $e^{(n)}$ is measurable with respect to $(X_i)_{i\leq T_{e^{(n)}}}$.
\end{remark}

We write $(A,B)\in \mathfrak{F}$ if
\begin{enumerate}
\item  $A=\{ \{e\}\times A^e, e\in E(\Z^d)\}$ where each $A^e$ belongs to the $\sigma$-field generated by the finite nearest-neighbour paths in $\Z^d_e$,
\item $B=\{ \{e\}\times B^e, e\in E(\Z^d)\}$ where each $B^e$ belongs to the $\sigma$-field generated by the functions from $E_e$ to the positive numbers, where $E_e$ is a subset of $E(\Z^d_e)$.
\end{enumerate}

For $(A,B)\in \mathfrak{F}$ and for any trajectory $\mathcal{T}^e$ on $\Z^d_e$ associated with some environment $\omega_e$, we write ${\mathcal{T}^e} \in (A,B)$, with an abuse of notation, to designate the event that
\begin{enumerate}
\item $(\mathcal{T}^e_n)_{n\leq \tau_1}\in A^e$, 
\item $\{c^{\omega_e}(e'), e'\in E(\Z^d_e) \text{ with } \mathcal{T}^e_i\in e' \text{ for } i\leq \tau_1 \text{ or } e'\sim e\}\in B^e$.
\end{enumerate}
In particular, this applies to $X^e$, $Y^e$ and  we will also use the notation ${X}^{e^{(n)}} \in (A,B)$. Recall that $e\notin E(\Z^d_e)$.


We are now going to prove one of the key propositions for understanding the behaviour of the the walk around large traps.
\begin{proposition}\label{big_trap_uncond}
There exists $K_0<\infty$ such that, for any $K\ge K_0$, there exists $\eta>0$ and a function $g(n)$ such that for any  $n\in \N$, any $(A,B)\in \mathfrak{F}$ and any borelian set $F$, we have that 
\begin{align*}
&\Big| { \PR}_0^K[c_*^\omega({e^{(n)}})\in F,\ {X}^{e^{(n)}} \in (A,B),LT(n),D=\infty]\\
 & \qquad -  {\bf P}[c_*\geq n, c_*\in F] \sum_{e\in E(\Z^d)} {\PR}_0^K[Y^e \in (A,B), T_e<\tau_1^{Y_e},D^{Y^e}=\infty]\Big| \leq g(n),
\end{align*}
where $g(n)=o(n^{-\eta}\overline{\PR}[LT(n)])$. In words, the difference on the left-hand side can be upper-bounded independently of our choice of $(A,B)$ or $F$.
\end{proposition}

The series appearing in the lemma cannot be infinite by Lemma~\ref{tail_taue1}.
\begin{proof}
Fix some $\delta\in(0,1)$ that will be chosen later. Recalling the notations~\eqref{def_exist_1_trap} and~\eqref{def_exist_2_trap}, we have
\begin{align*}
& {\PR}_0^K[{X}^{e^{(n)}} \in (A,B),\ c_*(e^{(n)})\in F,LT(n),D=\infty]\\
\le& \sum_{e\in E(\Z^d)}{\PR}_0^K[c_*(e)\ge n, T_e<\tau_1,{X}^e \in (A,B),c^{\omega}_*(e)\geq n,c^{\omega}_*(e)\in F,D=\infty]\\
&+{\PR}_0^K[SLT(\delta,K,n)],
\end{align*}
and
\begin{align*}
& {\PR}_0^K[{X}^{e^{(n)}} \in (A,B),\ c_*(e^{(n)})\in F,LT(n),D=\infty]\\
\ge& \sum_{e\in E(\Z^d)}{\PR}_0^K[c_*(e)\ge n, T_e<\tau_1,{X}^e \in (A,B),c^{\omega}_*(e)\geq n,c^{\omega}_*(e)\in F,D=\infty]\\
&-\sum_{e\in E(\Z^d)}{\PR}_0^K[c_*(e)\ge n^\delta,e\in B(2\chi,2\chi^\alpha),SLT(\delta,K,n)].
\end{align*}
Therefore, this yields
\begin{align*}
& \left|{\PR}_0^K[{X}^{e^{(n)}} \in (A,B),\ c_*(e^{(n)})\in F,LT(n),D=\infty]\right.\\
&\left. - \sum_{e\in E(\Z^d)}{\PR}_0^K[{X}^e \in (A,B),T_e<\tau_1,c^{\omega}_*(e)\geq n,c^{\omega}_*(e)\in F,D=\infty]\right|\\
 \leq  &
{\ES}_0^K[NLT(\delta,n)\1{SLT(\delta,K,n),D=\infty}].
\end{align*}
and we know by Lemma~\ref{not_2_trapsb} (applied with $\epsilon=\gamma\delta/4$) and Lemma~\ref{LB_pn} that
\[
\ES[NLT(\delta,K,n)\1{SLT(\delta,K,n)}\1{D=\infty}] \leq Cn^{-\gamma\delta/2}{\bf P}[c_*\geq n] =n^{-\gamma\delta/4}o(\overline{\PR}[LT(n)]),
\]
(which is a bound that does not depend on $(A,B)$ or $F$) and thus 
\begin{align}
& {\PR}_0^K[{X}^{e^{(n)}} \in (A,B),\ c^{\omega}_*(e^{(n)})\in F,LT(n),D=\infty]\label{approx0}\\
&= \sum_{e\in E(\Z^d)}{\PR}_0^K[{X}^e \in (A,B),T_e<\tau_1,c^{\omega}_*(e)\geq n,c^{\omega}_*(e)\in F, D=\infty]+n^{-\gamma\delta/4}o(\overline{\PR}[LT(n)]).\nonumber
\end{align}


Fix $M<\infty$, by Lemma~\ref{tail_chi_tau} and Lemma~\ref{tail_taue1}, there exists $K_0<\infty$ such that, for any $K\ge K_0$,
\[
{\PR}_0^K[T_{e} \leq \tau_1]  \leq  C \abs{\abs{e}}_{\infty}^{-M}\text{ and }{\PR}_0^K[T_{e} \leq \tau_1^{Y^e}]  \leq  C \abs{\abs{e}}_{\infty}^{-M},
 \]
 and  hence
\begin{align}\nonumber
& \left|\sum_{e\in E(\Z^d)}{\PR}_0^K[{X}^e \in (A,B),T_e<\tau_1,c^{\omega}_*(e)\geq n,c^{\omega}_*(e)\in F,D=\infty] \right.\\
&\left.- \sum_{e\in E(\Z^d)}{{\PR}}_0^K[{Y}^e \in (A,B),T^{Y^e}_{x_e}<\tau^{Y^e}_1,c^{\omega}_*(e)\geq n,c^{\omega}_*(e)\in F,D^{Y^e}=\infty]\right| \\ \nonumber
\leq & 2\sum_{e\in E(\Z^d) e\notin B_{\Z^d}(0,n^{1/M})} C \abs{\abs{e}}_{\infty}^{-2M}\\ \nonumber
&  +  \sum_{\substack{e\in E(\Z^d),\\ e\in B_{\Z^d}(0,n^{1/M})}} \left|{\PR}_0^K[{X}^e \in (A,B),T_e<\tau_1,c^{\omega}_*(e)\geq n,c^{\omega}_*(e)\in F,D=\infty] \right.\\ \nonumber
&\qquad\qquad\left.- {{\PR}}_0^K[{Y}^e \in (A,B),T^{Y^e}_{x_e}<\tau^{Y^e}_1,c^{\omega}_*(e)\geq n,c^{\omega}_*(e)\in F,D^{Y^e}=\infty]\right| \\ \nonumber
\leq & C n^{-1} \\ \nonumber
&  +  \sum_{\substack{e\in E(\Z^d),\\ e\in B_{\Z^d}(0,n^{1/M})}} \left|{\PR}_0^K[{X}^e \in (A,B),T_e<\tau_1,c^{\omega}_*(e)\geq n,c^{\omega}_*(e)\in F,D=\infty] \right.\\  \label{whereis}
&\qquad\qquad\left.- {{\PR}}_0^K[{Y}^e \in (A,B),T^{Y^e}_{x_e}<\tau^{Y^e}_1,c^{\omega}_*(e)\geq n,c^{\omega}_*(e)\in F,D^{Y^e}=\infty]\right| 
\end{align}

 Choose $\delta\in(0,1/(\gamma+3))$ and recall that
 \[
 {\bf P}[c_*(e)\ge n, \exists e'\sim e: c_*(e')\ge n^{\delta}]\le C n^{-\delta\gamma/2}{\bf P}[c_*(e)\ge n].
 \]
Thus, as $\gamma\in(0,1)$, we can choose $\delta>0$ and  $\epsilon>0$, independently of $(A,B)$ or $F$, such that, using the coupling and the results in Lemma~\ref{coupling} and Lemma~\ref{tail_regen_new}, we have
  \begin{align*}
&  \left|{\PR}_0^K[{X}^e \in (A,B),T_e<\tau_1,c^{\omega}_*(e)\geq n,c^{\omega}_*(e)\in F,D=\infty] \right.\\
&\left.- {{\PR}}_0^K[{Y}^e \in (A,B),T^{Y^e}_{x_e}<\tau^{Y^e}_1,c^{\omega}_*(e)\geq n,c^{\omega}_*(e)\in F,D^{Y^e}=\infty]\right|\\
& \le C n^{-2\epsilon}\overline{\PR}[LT(n)]).
\end{align*}

Choosing $M$ large enough (depending $\epsilon$ and $d$), using the last inequality, \eqref{approx0} and~\eqref{whereis}, we obtain
 \begin{align*}
&  {\PR}_0^K[{X}^{e^{(n)}} \in (A,B),\ c_*(e^{(n)})\geq n,c^{\omega}_*(e^{(n)})\in F,LT(n),D=\infty] \\
&= \sum_{e\in E(\Z^d)}{{\PR}}_0^K[{Y}^e \in (A,B),T^{Y^e}_{x_e}<\tau^{Y^e}_1,c^{\omega}_*(e)\geq n,c^{\omega}_*(e)\in F,D^{Y^e}=\infty] +n^{-\epsilon}o(\overline{\PR}[LT(n)]).
\end{align*}
Now, we can use the fact that the trajectory of $Y^e$ and the environment outside $e$ are independent of $c^{\omega}_*(e)$ and thus
\begin{eqnarray*}
&&\sum_{e\in E(\Z^d)}{{\PR}}_0^K[{Y}^e \in (A,B),T^{Y^e}_{x_e}<\tau^{Y^e}_1,c^{\omega}_*(e)\geq n,c^{\omega}_*(e)\in F,D^{Y^e}=\infty]\\
&&=  {\bf P}[c_*\geq n,c_*\in F]\sum_{e\in E(\Z^d)}{{\PR}}_0^K[{Y}^e \in (A,B),T^{Y^e}_{x_e}<\tau^{Y^e}_1,D^{Y^e}=\infty] ,
 \end{eqnarray*}
 and so finally
 \begin{align*}
&  {\PR}_0^K[{X}^{e^{(n)}} \in (A,B),\ c_*(e^{(n)})\geq n,c^{\omega}_*(e^{(n)})\in F,LT(n),D=\infty] \\
&= {\bf P}[c_*\geq n,c_*\in F]\sum_{e\in E(\Z^d)}{{\PR}}_0^K[{Y}^e \in (A,B),T^{Y^e}_{x_e}<\tau^{Y^e}_1,D^{Y^e}=\infty] \\
&\quad+n^{-\epsilon}o(\overline{\PR}[LT(n)]).
\end{align*}
This implies the result.
\end{proof}

The previous proposition implies the following statement.
\begin{lemma}\label{big_trap_cond}
Take $(A,B)\in \mathfrak{F}$. There exists $K_0<\infty$ such that, for any $K\ge K_0$, we have, as $n$ goes to infinity,
\begin{align*}
&{ \PR}_0^K[{X}^{e^{(n)}} \in (A,B) \mid LT(n),D=\infty]\\
 &\qquad\longrightarrow \frac{\sum_{e\in E(\Z^d)}{{\PR}}_0^K[{Y}^e \in (A,B),T^{Y^e}_{x_e}<\tau^{Y^e}_1,D^{Y^e}=\infty]  }{\sum_{e\in E(\Z^d)}{{\PR}}_0^K[T^{Y^e}_{x_e}<\tau^{Y^e}_1,D^{Y^e}=\infty]  }.
 \end{align*}
 \end{lemma}
\begin{proof}
The result follows easily once, we have proved that 
\[
\sum_{e\in E(\Z^d)}{{\PR}}_0^K[T^{Y^e}_{x_e}<\tau^{Y^e}_1,D^{Y^e}=\infty] \geq c .
\]

For $e_0=[2e_1,3e_1]$, we can prove a lower bound 
\[
{\PR}_0^K[T_{x_{e_0}}^{Y^{e_0}}<\tau_1^{Y^{e_0}},D^{Y^{e_0}}=\infty] \geq c,
\]
 by noticing that the event $\{T_{x_{e_0}}^{Y^{e_0}}<\tau_1^{Y^{e_0}},D^{Y^{e_0}}=\infty\}$ occurs if
\begin{enumerate}
\item $0$, $e_1$, $2e_1$, $3e_1$ and $4e_1$ are good;
\item $(Y^{e_0}_1,Z^{Y^{e_0}}_1)=(e_1,1)$, $(Y^{e_0}_2,Z^{Y^{e_0}}_2)=(x_{e_0},1)$ and $Y^{e_0}_3=4e_1$;
\item $D^{Y^{e_0}}\circ \theta_3=\infty.$
\end{enumerate}
and then we can do a computation very similar to the proof of Lemma~\ref{LB_pn} and use the fact that 
\[
\ES^K_0\left[\1{4e_1\text{ is good}}{P}^\omega_{4e_1}[D^{Y^{e_0}}=\infty]\right]=\ES^K_0\left[\1{4e_1\text{ is good}}{P}^\omega_{4e_1}[D=\infty]\right]>0.
\]
\end{proof}

%
%
%
%


\section{Approximation of the time spent in a large trap.}

The goal of this section is to find an asymptotic approximation of the time spent in a trap when the edge associated to this trap has a large conductance. More specifically, we want to show that the time is roughly the conductance of the large edge times an independent random variable $W_{\infty}$.

\subsection{Number of return to a large trap}

We define the random variable 
\[
V^{x_e}=\text{card}\{i \geq 0,\ Y_i^e=x_e,\ i\leq \tau_1^{Y^e}\}.
\]
 Besides, let us define 
 \begin{equation}\label{sandro}
 V_n:=\text{card}\{i< \tau_1,\ X_i\notin e^{(n)} \text{ and } X_{i+1} \in e^{(n)}\}.
 \end{equation}
 
Note that, chosen under the measure $\PR[\cdot \mid LT(n)]$, $V_n\ge 1$, we can see that the previous random variable verifies
 \begin{equation}\label{sandro}
 V_n=\text{card}\{i\ge0,\ X_i\notin e^{(n)} \text{ and } X_{i+1} \in e^{(n)}\}.
 \end{equation}
Finally, let us define, for any $e\in E(\Z^d)$,
\begin{equation}\label{traindemerde}
\overline{\pi}^{\omega_e}(x_e):=e^{-(e^++e^-)\cdot \ell}{\pi}^{\omega_e}(x_e).
\end{equation}

 A direct application of the Lemma~\ref{big_trap_cond} implies that
\begin{lemma}\label{cvg_vpi}
There exists $K_0<\infty$ such that, for any $K\ge K_0$, there exists a couple of random variables $(V_{\infty},\overline{\pi}^{\infty})$ such that under ${\PR}_n$ we have
\[
(V_n,\overline{\pi}^{\omega_{e^{(n)}}}(x_{e^{(n)}}))\xrightarrow{(d)} (V_{\infty},\overline{\pi}^{\infty}).
\]
Denoting $\PR^\infty$ the probability associated to $(V_{\infty},\overline{\pi}^{\infty})$, the distribution of this couple of variables is given by
\[
\PR^\infty[(V_{\infty},\overline{\pi}^{\infty})\in\cdot]=\frac{\displaystyle{\sum_{e\in E(\Z^d)}} {\PR}_0^K[(V^{x_e}, \overline{\pi}^{\omega_e}(x_e))\in\cdot, T_e<\tau_1^{Y^e},D^{Y^e}=\infty] }{\displaystyle{\sum_{e\in E(\Z^d)}} {\PR}_0^K[T^{Y^e}_{x_e}<\tau^{Y^e}_1,D^{Y^e}=\infty] }.
\]
\end{lemma}

\begin{proof}
Fix some sets $A\subset \N$ and $B\subset \R$ which are measurable and define
\begin{eqnarray*}
\widetilde{A}&=\left\{\{e\}\times A^e, e\in E(\Z^d)\right\},\\
\widetilde{B}&=\left\{\{e\}\times B^e, e\in E(\Z^d)\right\},
\end{eqnarray*}
where $A^e$ a the set of trajectories on $\Z^d_e$ such that $V^{x_e}\in A$, and $B^e$ is the event $\{\overline{\pi}^{\omega_{e^{(n)}}}(x_{e^{(n)}}) \in B\}$. Note that $(\widetilde{A},\widetilde{B})\in\mathfrak{F}$. Therefore,
 we have by Lemma~\ref{big_trap_cond}
\begin{align*}
& {\PR}_n[V_n \in A, \overline{\pi}^{\omega_{e^{(n)}}}(x_{e^{(n)}}) \in B] \\
=& {\PR}_n[X^{e^{(n)}}\in(\widetilde{A},\widetilde{B})]\\
\to & \frac{\displaystyle{\sum_{e\in E(\Z^d)}{{\PR}}_0^K[{Y}^e \in (\widetilde{A},\widetilde{B}),T^{Y^e}_{x_e}<\tau^{Y^e}_1,D^{Y^e}=\infty] } }{\displaystyle{\sum_{e\in E(\Z^d)}{{\PR}}_0^K[T^{Y^e}_{x_e}<\tau^{Y^e}_1,D^{Y^e}=\infty]  }}\\
=& \frac{\displaystyle{\sum_{e\in E(\Z^d)}} {\PR}_0^K[V^{x_e}\in A, \overline{\pi}^{\omega_e}(x_e)\in B, T_e<\tau_1^{Y^e},D^{Y^e}=\infty] }{\displaystyle{\sum_{e\in E(\Z^d)}} {\PR}_0^K[T^{Y^e}_{x_e}<\tau^{Y^e}_1,D^{Y^e}=\infty] },
\end{align*}
and notice that the right-hand side is a probability distribution corresponding to our limiting random variables.
\end{proof}

 \subsection{Time in excursions in the large edge $e$}

 Let us define the exit time of an edge $e\in E(\Z^d)$ once it has been hit $T_e^{\mathrm{ex}}=T_{\{e^+,e^-\}^c}\circ \theta_{T_{\{e^+,e^-\}}}$ with the convention that $T_e^{\mathrm{ex}}=0$ if $T_{\{e^+,e^-\}}=\infty$. Moreover, we define, for any $i\ge 1$, $T_e^{\mathrm{ex},i}$ as the time spent in $e$ during the $i$-th excursion, again with the convention that $T_e^{\mathrm{ex},i}=0$ if there are not $i$ excursions to $e$, i.e.~$V_n<i$.

 \begin{lemma}\label{coupl_exp}
Fix $\delta>0$. Take  $e\in E(\Z^d)$ such that $c^{\omega}_*(e)\geq n$ and $c^{\omega}_*(e')\leq n^{\delta}$ for all $e'\sim e$. Then, there exists a coupling of $({\bf e}_i)$ and $(T^{\mathrm{ex},i}_e)$ where $({\bf e}_i)$ are i.i.d.~exponential random variables with mean $1$ such that, for any $i\ge1$ such that $V_n\ge i$, $P^\omega$-almost surely,
\[
(1-C(d)n^{\delta-1})\frac{c^\omega(e)}{\pi^{\omega_e}(x_e)}2{\bf e}_i \le T^{\mathrm{ex},i}_e\le (1+C(d)n^{\delta-1})\frac{c^\omega(e)}{\pi^{\omega_e}(x_e)}2{\bf e}_i+1,
\]
 where $({\bf e}_i)$ is independent of $c^\omega(e)$, $\pi^{\omega_e}(x_e)$ and $(X^e_n)_n$.
\end{lemma}
 
\begin{proof}
{\it Step 1: Probability distribution estimates }

\vspace{0.5cm}

We do the proof for $T_e^{\mathrm{ex}}=T_e^{\mathrm{ex},1}$, and the result will follow by induction. Let us apply Markov's property at time $T_e$, the time when the edge $e$ is hit. For any $k\in\N$, we have, almost surely,
\begin{align*}
& P_{X_{T_e}}^{\omega}\left[\left\lfloor \frac{T_e^{\mathrm{ex}}}{2}\right\rfloor >k\right] \\
=& P\left[\mathrm{Geom}\left(1-\frac{(c^\omega(e))^2}{\pi^\omega(e^+)\pi^\omega(e^-)}\right)> k\right]\\
=& P\left[\frac{{\bf e}}{\ln\left(\frac{\pi^\omega(e^+)\pi^\omega(e^-)}{(c^\omega(e))^2}\right)}> k\right]
\end{align*}
where ${\bf e}$ is an exponential random variable ${\bf e}$ with mean $1$, which does not depend on the environment. These last equalities show in particular that $\lfloor {T_e^{\mathrm{ex}}}/{2}\rfloor$ does not depend on $X_{T_e}$. Using that $c^\omega(e)\ge n$ and $c^\omega(e')\le n^\delta$ for any $e'\sim e$, a straightforward computation yields
\begin{align}\nonumber
 P\left[(1-C(d)n^{\delta-1})\frac{c^\omega(e)}{\pi^{\omega_e}(x_e)}{\bf e}> k\right] \le & P_{X_{T_e}}^{\omega}\left[\left\lfloor \frac{T_e^{\mathrm{ex}}}{2}\right\rfloor > k\right] \\ \label{buuullshit}
\le & P\left[(1+C(d)n^{\delta-1})\frac{c^\omega(e)}{\pi^{\omega_e}(x_e)}{\bf e} > k\right].
\end{align}
We give an explicit construction of the coupling in order to emphasize the dependencies. First, note that, for any $k\in \N$ and any $y\sim e$, $y\notin e$,
\begin{align*}
&P^\omega_{X_{T_e}}\left[\left\lfloor \frac{T_e^{\mathrm{ex}}}{2}\right\rfloor=k, X_{ T_e^{\mathrm{ex}}}=y\right]\\
=& P^\omega_{X_{T_e}}\left[\left\lfloor \frac{T_e^{\mathrm{ex}}}{2}\right\rfloor=k\right]P_{X_{T_e}}^\omega\left[X_{T_e^{\mathrm{ex}}}=y\left| \left\lfloor \frac{T_e^{\mathrm{ex}}}{2}\right\rfloor=k\right.\right]\\
=& P^\omega_{X_{T_e}}\left[\left\lfloor \frac{T_e^{\mathrm{ex}}}{2}\right\rfloor=k\right]P_{X_{T_e}}^\omega\left[X_{T_e^{\mathrm{ex}}}=y\right],
\end{align*} 
where the last equality can be checked with a straightforward computation.
This means that, conditioned on $T_e$, $\lfloor T_e^{\mathrm{ex}}/2\rfloor$ and $X_{T_e^{\mathrm{ex}}}$ are independent.\\

\vspace{0.5cm}

{\it Step 2: The coupling}

\vspace{0.5cm}

Let us now explain how we couple the exponential variables with the time spent in the edge $e$. 

Assume that $X$ hits $e$ for the $i$-th time. Then, pick an exponential random variable ${\bf e}_i$ independently of anything else. First we define ${\bf e}_{i,-}:=(1-C(d)n^{\delta-1})\frac{c^\omega(e)}{\pi^{\omega_e}(x_e)}{\bf e}_i$ and ${\bf e}_{i,+}:=(1+C(d)n^{\delta-1})\frac{c^\omega(e)}{\pi^{\omega_e}(x_e)}{\bf e}_i$, where $C(d)$ is the same constant as in \eqref{buuullshit}. Let us denote $G_-$, $G_+$ and $F$ the respective cumulative distribution functions of  ${\bf e}_{i,-}$, ${\bf e}_{i,+}$ (conditioned on the environment) and the quenched cumulative distribution function of $\lfloor T_e^{\mathrm{ex}}/2\rfloor$ (recall  that this does not depend on the starting point). Note that, given the environment, $G_+$ and $G_-$ are continuous. 

Now, given ${\bf e}_i$, we will introduce the coupling for the time spent on the edge by setting
\[
\left\lfloor \frac{T_e^{\mathrm{ex},i}}{2}\right\rfloor=F^{-1}\left(G_-({\bf e}_{i,-})\right).
\]

Once we notice that $G_-(\cdot)=G_+(\cdot\times(1+C(d)n^{\delta-1})/(1-C(d)n^{\delta-1}))$ and that ${\bf e}_{i,+}={\bf e}_{i,-}(1+C(d)n^{\delta-1})/(1-C(d)n^{\delta-1})$, it is easy to prove that, by construction, ${\bf e}_{i,-}\le \lfloor T_e^{\mathrm{ex},i}/2\rfloor \le {\bf e}_{i,+}$, which implies the inequality of the statement. The independence of ${\bf e}_i$ and $X_{ T_e^{\mathrm{ex},i} }$ comes from the independence of $\lfloor T_e^{\mathrm{ex},i}/2\rfloor$ and $X_{ T_e^{\mathrm{ex},i} }$. Indeed, we use ${\bf e}_i$ only to determine the value of $\lfloor T_e^{\mathrm{ex},i}/2\rfloor$ and this is independent of $X_{ T_e^{\mathrm{ex},i} }$.

\end{proof}

\subsection{The time spent in the largest edge described using a random variable $W_n$}

We know that  under ${\PR}_n$, defined in \eqref{defPn}, the random variable $T=\text{card}\{i\leq \tau_1,\ X_i \in e^{(n)}\}$ measuring the time spent in $e^{(n)}$ can be written
\[
T=\sum_{i=1}^{V_n} T^{(i)}_{e^{(n)}},
\]
where $T^{(i)}_{e^{(n)}}$ is the time spent during the $i$-th excursion in ${e^{(n)}}$ .
\begin{remark}\label{whysoserious}
The random variables $T^{(i)}_{e^{(n)}}$  are distributed as $T^{\mathrm{ex}}_{e^{(n)}}$  chosen under $P^{\omega}_{x}$ for some  $x \in {e^{(n)}}$. The law of the $T^{(i)}_{e^{(n)}}$ typically depends on $V_n$ and may not be the same for all $i$, as it depends on which vertex the walker enters and exits the edge $e^{(n)}$. Nevertheless, in the previous proof, we showed that the random variables $\lfloor T^{(i)}_{e^{(n)}}/2\rfloor$ are independent of $V_n$.
\end{remark}

As we stated in Section \ref{sketch}, we aim to prove that the time spent in the trap, when this trap is large, is the product of the conductance $c^{\omega}_*(e^{(n)})$, with a random variable which is almost independent. For this reason we introduce
\begin{equation}\label{def_Wblabla}
W_n:=\frac{T}{c^{\omega}_*(e^{(n)})},
\end{equation}
and we will now prove that this random variable admits a limit under ${\PR}_n$ and that as $n$ gets large $W_n$ and $c^{\omega}_*(e^{(n)})$ are asymptotically independent.


Recall the definitions \eqref{def_OLT} and \eqref{def_OLTe} of $OLT(\delta,K,n)$ and $OLT_e(\delta,K,n)$. 
Using Lemma~\ref{not_2_traps} and that $LT(n)\cap SLT(\delta,K,n)^c\subset OLT(\delta,K,n)$ by Remark \ref{elodie}, it is easy to conclude that there exists $\epsilon>0$ such that
\begin{equation}\label{thomas}
\PR_n[OLT(\delta,K,n)^c]=o(n^{-\epsilon}).
\end{equation}

We have $ OLT(\delta,K,n)=\bigcup_{e\in E(\Z^d)}  OLT_e(\delta,K,n)$, and on $ OLT_e(\delta,K,n)$ we have $e^{(n)}=e$

Let us first prove a uniform estimate on the moments of $W_n$ which will be useful to prove limiting results.
\begin{lemma}\label{lem_momWn}
Fix $\delta>0$ and $\epsilon\in[0,1-\gamma)$. There exists $K_0<\infty$ such that, for any $K\ge K_0$, there exists a constant $C(K)$ such that, for any $n>K$, we have
\begin{align*}
\ES_n\left[W_n^{\gamma+\epsilon}\1{OLT(\delta,K,n)}\right]&<C(K),\\
\ES_n\left[\left(\frac{V_n}{\overline{\pi}^{\omega_{e^{(n)}}}(x_{e^{(n)}})}\right)^{\gamma+\epsilon}\1{OLT(\delta,K,n)}\right]&<C(K).
\end{align*}
\end{lemma}

\begin{proof}
Recall that under $\PR_n$, defined in \eqref{defPn}, $0$ is open and thus $0\notin e^{(n)}$.
Fix some $\epsilon\in[0,1-\gamma)$. 

\vspace{0.5cm}

{\it Step 1: Relating $W_n$ and $V_n$}

\vspace{0.5cm}

Firstly, we want to show that $ W_n \1{OLT(\delta,K,n)}$ has, under ${\PR}_n$, a moment $\gamma + \epsilon$ which is uniform in $n$.\\
Conditionally on the fact that $e^{(n)}$ is hit for the $i$-th time, the random variables $T^{(i)}_{e^{(n)}}$ are distributed as $T^{\text{ex}}_{e^{(n)}}$ under $P^{\omega}_{x}$ for some random $x\in e^{(n)}$ (see Remark~\ref{whysoserious}). Fix some edge $e\in E(\Z^d)$. We can use Lemma~\ref{coupl_exp} to see that on  $OLT_e(\delta,K,n)$  for all $i$
\begin{align}\label{abel}
T^{(i)}_{e} \leq 
(1+C(d)n^{\delta-1})\frac{c^\omega(e)}{\pi^{\omega_e}(x_e)}2{\bf e}_i+1,
\end{align}
where ${\bf e}_i$ are i.i.d.~exponential random variables of mean $1$ that are independent of $({X}^{e^{(n)}})$, $e^{(n)}$, ${\pi^{\omega_{e^{(n)}}}(x_{e^{(n)}})}$ and $c_*(e^{(n)})$.
Thus, under ${\PR}_n$, we have
 \[
 W_n \1{OLT(\delta,K,n)}  \leq (1+C(d)n^{\delta-1})\frac{ \1{OLT(\delta,K,n)}}{\overline{\pi}^{\omega_{e^{(n)}}}(x_{e^{(n)}})} \sum_{i=1}^{V_n}(2{\bf e}_i+1)\1{OLT(\delta,K,n)},
 \]
 so taking the expectation with respect to the randomness of the exponentials, we see that
 \[
E_{\text{exp}}\left[W_n\1{OLT(\delta,K,n)}\right] \leq C  \frac{V_n}{\overline{\pi}^{\omega_{e^{(n)}}}(x_{e^{(n)}})}\1{OLT(\delta,K,n)}.
 \]

\vspace{0.5cm}

{\it Step 2: Relating $V_n$ to the number of visits of $e^+$ and $e^-$}

\vspace{0.5cm}

 On the event $ OLT_e(\delta,K,n)$, we have
\[
V_n= \text{card}\{i\ge 1: X_i\in e\text{ and } X_{i+1}\notin e\},
\]
thus
\begin{eqnarray}
&&E^\omega_0[V_n\1{OLT_e(\delta,K,n),D=\infty}]\nonumber\\
&&\le E_0^\omega\Bigg[ \sum_{i=1}^\infty (\1{T_0^+>i,X_i=e^-}(1-\1{X_{i+1}=e^+})\nonumber\\
&&\qquad +\1{T_0^+>i,X_i=e^+}(1-\1{X_{i+1}=e^-}))\Bigg]\nonumber\\
&=& E_0^\omega\left[ N_0(e^-)\right]\frac{\sum_{y\sim e^-, y\neq e^+} c([y,e_y])}{\pi(e^-)}\nonumber\\
&&+E^\omega_0\left[N_0(e^+)\right]\frac{\sum_{y\sim e^+, y\neq e^-} c([y,e_y])}{\pi(e^+)},\label{snoop}
\end{eqnarray}
where we used Markov's property and where $N_0(e^\pm):=\sum_{i=1}^\infty\1{T_0^+>i,X_i=e^\pm}$.

\vspace{0.5cm}

{\it Step 3: The number of visits to $e^+$ (resp.~$e^-$) are related to the size of the surrounding bad area and$\pi^{\omega}(e^+)$ (resp.~$\pi^{omega}(e^-)$)}

\vspace{0.5cm}

 We now want to estimate the expectations appearing in previous equation. For this purpose, note that, as $n>K$, $e^+$ and $e^-$ are bad vertices so that $\mathrm{BAD}(e^+)=\mathrm{BAD}(e^-)$ and, for any $y\in\Z^d$ and stopping-time $T$, define $N(y,T)=\left| \left\{ 0\le n\le T:X_n=y\right\}\right|$. We have that
\begin{eqnarray}
N_0(e^+)&\le& \1{0\in\mathrm{BAD}(e^+)}N(e^+,T^+_{\mathrm{GOOD}\cup\{0\}})\nonumber\\
&&+\sum_{x\in\partial\mathrm{BAD}(e^+)} \sum_{i=0}^\infty \1{X_i=x}N(e^+,T^+_{\mathrm{GOOD}})\circ\theta_i,\label{snoop2}
\end{eqnarray}
and a similar inequality holds for $N_0(e^-)$.\\
Now, as in the proof of Lemma \ref{timetrap}, we can consider the finite graph $\omega_\delta$, obtained by merging all the points of $\partial \mathrm{BAD}(e^+)$ (or $\{0\}\cup\partial \mathrm{BAD}(e^+)$ if $0\in\mathrm{BAD}(e^+)$) into one point $\delta$.\\
%
In the case $0\in\mathrm{BAD}(e^+)$ and by merging $\{0\}\cup\partial \mathrm{BAD}(e^+)$ into $\delta$, as $0$ is open, we have by Lemma \ref{psyche} and Remark \ref{fakeUE}, for any $y\in\mathrm{BAD}(e^+)$, $y\sim 0$,
\begin{align*}
&E^{\omega}_y[N(e^+,T^+_{\mathrm{GOOD}\cup\{0\}})]=E^{\omega_\delta}_y[N(e^+,T^+_{\delta})]\le C(K)\frac{\pi^\omega(e^+)}{c^\omega([0,y])}\\
&\le C\exp\left(2\lambda\left|\max_{z\in\{0\}\cup\partial \mathrm{BAD}(e^+)} z \cdot\vec{\ell
}-\min_{z\in\{0\}\cup\partial \mathrm{BAD}(e^+)}
z \cdot\vec{\ell}       \right|\right)      e^{-(e^++e^-)\cdot\ell}\pi^\omega(e^+),
\end{align*}
where we used the fact that the number of visits to $e^+$ is upper bounded by the number of times incident edges have been crossed.\\
Using that
\[
\max_{z\in\{0\}\cup\partial \mathrm{BAD}(e^+)} z \cdot\vec{\ell
}-\min_{z\in\{0\}\cup\partial \mathrm{BAD}(e^+)}
z \cdot\vec{\ell}\leq W\left(\mathrm{BAD}(e^+)\right),
\]
we obtain
\begin{eqnarray*}
E^{\omega_\delta}_y[N(e^+,T^+_{\mathrm{GOOD}\cup\{0\}})]  \le C\exp\left(2\lambda W\left(\mathrm{BAD}(e^+)\right)\right) e^{-(e^++e^-)\cdot\ell}\pi^\omega(e^+).
\end{eqnarray*}
This yields
\begin{align}\nonumber
E^\omega_0[N(e^+,T^+_{\mathrm{GOOD}\cup\{0\}})] & \le 1+ C(K)\sum_{y\sim0} E^{\omega}_y[N(e^+,T^+_{\mathrm{GOOD}\cup\{0\}})]\\ \label{eminem}
& \le C\exp\left(2\lambda W\left(\mathrm{BAD}(e^+)\right)\right)e^{-(e^++e^-)\cdot \ell}\pi(e^+),
\end{align}
where we used that $0$ is open and, as we are on $OLT_e(\delta,K,n)$, $e^{-(e^++e^-)\cdot \ell}\pi(e^+)\ge C(K)$.
A similar inequality holds for $e^-$.\\
In the same way, by merging only $\partial \mathrm{BAD}(e^+)$ into $\delta$, we have, for any $x\in\partial \mathrm{BAD}(e^+)$, we have 
\begin{equation}\label{eminem2}
E^\omega_x[N(e^+,T^+_{\mathrm{GOOD}})]\le C\exp\left(2\lambda W\left(\mathrm{BAD}(e^+)\right)\right)e^{-(e^++e^-)\cdot l}\pi(e^+),
\end{equation}
and a similar inequality holds for $e^-$.\\
Using Lemma \ref{backbone}, \eqref{snoop} and \eqref{snoop2}, we have
\begin{equation}\label{dr.dre}
E^\omega_0[V_n\1{OLT_e(\delta,K,n),D=\infty}]\le  C(K)\exp\left(2\lambda W\left(\mathrm{BAD}(e^+)\right)\right)\overline{\pi}^{\omega_e}(x_e).
\end{equation}

\vspace{0.5cm}

{\it Step 4: The size of bad areas is too small to have a significant effect}

\vspace{0.5cm}

Fix some constant $M\ge  4d+4d/(1-\gamma-\epsilon)$.
Now, for $k\in\{0,...,\lfloor||e||_\infty^{M}\rfloor\}$, by Lemma \ref{tail_taue}, there exists $K_0$ such that, for any $K\ge K_0$,
\begin{eqnarray}\nonumber
&&\PR_0^K\left[\left(\frac{V_n}{ \overline{\pi}^{\omega_e}(x_e)} \right)^{\gamma+\epsilon}\1{OLT_e(\delta,K,n),D=\infty}\ge k \right]\\ \label{matelas1}
&&\qquad \le \PR_0^K[T_e<\tau_1,c_*(e)\ge n]\le C(K)||e||_\infty^{-2M-4d}{\bf P}[c_*(e)\ge n],
\end{eqnarray}

Another more delicate upper bound can be obtained, for an integer $k>\lfloor||e||_\infty^{M}\rfloor$, Markov's inequality and \eqref{dr.dre} yield
\begin{align}\nonumber
&\PR_0^K\left[\left(\frac{V_n}{ \overline{\pi}^{\omega_e}(x_e)} \right)^{\gamma+\epsilon}\1{OLT_e(\delta,K,n),D=\infty}\ge k \right]\\ \nonumber
 \le & \ES_0^K\left[P_0^\omega\left[\left(\frac{V_n}{ \overline{\pi}^{\omega_e}(x_e)} \right)^{\gamma+\epsilon}\1{OLT_e(\delta,K,n),D=\infty}\ge k\right]\1{c_*(e)\ge n}\right]\\ \nonumber
\le & \frac{C(K)}{k^{1/(\gamma+\epsilon)}} \ES_0^K\left[\exp\left(2\lambda W\left(\mathrm{BAD}(e^+)\right)\right)\1{c_*(e)\ge n}\right]\\ \nonumber
\le & \frac{C(K)}{k^{1/(\gamma+\epsilon)}} \ES_0^K\left[\left.\exp\left(2\lambda W\left(\mathrm{BAD}(e^+)\right)\right)\right|{c_*(e)\ge n}\right]{\bf P}\left[{c_*(e)\ge n}\right]\\ \nonumber
\le &  \frac{C(K)}{k^{1/(\gamma+\epsilon)}} \ES_0^K\left[\left.\exp\left(2\lambda W\left(\mathrm{BAD}(e^+)\right)\right)\right|e\text{ is closed}\right]{\bf P}\left[{c_*(e)\ge n}\right]\\ 
&\le \frac{C(K)}{k^{1/(\gamma+\epsilon)}} {\bf P}[{c_*(e)\ge n}], \label{matelas2}
\end{align}
where we used Lemma \ref{BLsizeclosedbox} with $K$ large enough (depending on $\lambda$ and $d$).\\

\vspace{0.5cm}

{\it Step 5: Conclusion}

\vspace{0.5cm}

Summing over $k\ge0$, we obtain, by \eqref{matelas1} and \eqref{matelas2},
\begin{eqnarray*}
&&\ES_0^K\left[\left(\frac{V_n}{ \overline{\pi}^{\omega_e}(x_e)} \right)^{\gamma+\epsilon}\1{OLT_e(\delta,K,n),D=\infty} \right]\\
&&\le C(K)||e||_\infty^{-4d}{\bf P}[c_*(e)\ge n]+C(K)||e||_\infty^{-(1-\gamma-\epsilon)M/(\gamma+\epsilon)}{\bf P}[c_*(e)\ge n]\\
&&\le C(K)||e||_\infty^{-4d}{\bf P}[c_*(e)\ge n].
\end{eqnarray*}
Finally, as $\text{card}\{e\in E(\Z^d):||e||_\infty=k\}\le ck^{d-1}$, we have
\begin{eqnarray*}
&&\ES_0^K\left[\left(\frac{V_n}{ \overline{\pi}^{\omega_{e^{(n)}}}(x_{e^{(n)}})} \right)^{\gamma+\epsilon}\1{OLT(\delta,K,n),D=\infty} \right]\\
&&=\sum_{e\in E(\Z^d)} \ES_0^K\left[\left(\frac{V_n}{ \overline{\pi}^{\omega_e}(x_e)} \right)^{\gamma+\epsilon}\1{OLT_e(\delta,K,n),D=\infty} \right]\\
&&\le C(K){\bf P}[c_*(e)\ge n]\\
&&\le C(K)\PR_0^K[LT(n),D=\infty],
\end{eqnarray*}
using Lemma \ref{LB_pn} in the last line. This implies the result.
\end{proof}


\subsection{Limit in law of the random variables $W_n$}

The random variables $W_n$ introduced at \eqref{def_Wblabla} to understand the time spent in large trap have a limit in law.

\begin{lemma}\label{cvg_wn}
Define
\[
W_{\infty}= \frac 1 {\overline{\pi}^{\infty}} \sum_{i=1}^{V_{\infty}}2{\bf e}_i,
\]
where ${\bf e}_i$ are some i.i.d.~exponential random variables with mean $1$, independent of $V_{\infty}$ and $\overline{\pi}^{\infty}$ (defined in Lemma \ref{cvg_vpi}).\\
Then, for the random variables $W_n$ chosen under ${\PR}_n$, we have
\[
W_n \xrightarrow{(d)} W_{\infty}\text{ as }n\text{ goes to infinity}.
\]
\end{lemma}

\begin{proof}
Consider $({\bf e}_i)_i$ a sequence of exponential random variables with mean $1$. If, for some $e\in E(\Z^d)$,  $c^{\omega}_*(e)\geq n$ and $c^{\omega}_*(e')\leq n^{\delta}$ for all $e'\sim e$, then using Lemma~\ref{coupl_exp} there exists a coupling of $({\bf e}_i)_i$ and $(\lfloor T^{(i)}_{e}/2\rfloor)_i$ such that, for all $i$, we have
\[
(1-C(d)n^{\delta-1})\frac{c^\omega(e)}{\pi^{\omega_e}(x_e)}2{\bf e}_i \leq T^{(i)}_{e} \leq 
(1+C(d)n^{\delta-1})\frac{c^\omega(e)}{\pi^{\omega_e}(x_e)}2{\bf e}_i+1,
\]
and where ${\bf e}_i$ are independent of each other and independent of $({X}^{e})$, $e$, ${\pi^{\omega_{e}}(x_{e})}$ and $c_*(e)$. Note that, on the other hand,  $T_e^{(i)}$ depends on these quantities and on ${\bf e}_i$. 

Using that $OLT(\delta,K,n)=\bigcup_{e\in E(\Z^d)} OLT_e(\delta,K,n)$, we have, on $OLT(\delta,K,n)$,
\begin{align}\label{decorwn}
& (1-C(d)n^{\delta-1})\frac{1}{\overline{\pi}^{\omega_{e^{(n)}}}(x_{e^{(n)}})}\sum_{i=1}^{V_n} 2{\bf e}_i  \\ \nonumber
\leq& W_n \leq   (1+C(d)n^{\delta-1})\frac{1}{\overline{\pi}^{\omega_{e^{(n)}}}(x_{e^{(n)}})}\sum_{i=1}^{V_n} (2{\bf e}_i+C(d)n^{\delta-1}),
\end{align}
where $V_n$ and $\pi^{\omega^{e^{(n)}}}(x_{e^{(n)}})$ are defined in \eqref{sandro} and \eqref{traindemerde}.\\

Now recall that by Lemma \ref{cvg_vpi}, under $\PR_n$, $(V_n,\pi^{\omega^{e^{(n)}}}(x_{e^{(n)}}))$ converges in law to  $(V_{\infty},\pi^{\infty})$ and that the exponential random variables are independent of $V_n$ and $\pi^{\omega^{e^{(n)}}}(x_{e^{(n)}})$. Using Markov's inequality and Lemma~\ref{lem_momWn} with $\epsilon < \delta$, it is easy to prove that, under $\PR_n$,
\[
n^{\delta-1}\frac{V_n}{\overline{\pi}^{\omega_{e^{(n)}}}(x_{e^{(n)}})}\1{OLT(\delta,K,n)}\to 0 \text{ in probability,}
\]
and recalling the definition of $W_n$ at~\eqref{def_Wblabla}
\[
n^{\delta-1}\frac{1}{\overline{\pi}^{\omega_{e^{(n)}}}(x_{e^{(n)}})}\sum_{i=1}^{V_n} 2{\bf e}_i \1{OLT(\delta,K,n)}\to 0 \text{ in probability.}
\]
Thus, we have the convergence in law, under $\PR_n$, of $W_n\1{OLT(\delta,K,n)}$ to $W_\infty$ and, as $\PR_n[OLT(\delta,K,n)^c]=o(1)$, the result follows.
\end{proof}

\subsection{Asymptotic independence of the conductance and $W_n$ }


On the probability space $\PR^{\infty}$, we define the random variables $(V_{\infty},\pi^{\infty})$ and independently of this couple a variable $c_*^{\mathrm{max}}$ which has the $\overline{\PR}$-law of the largest conductance met during the first regeneration period, that is
\[
c_*^{\mathrm{max}}=\max\left\{c^\omega_*(e) \text{, with } e\in E(\Z^d), T_e<\tau_1\right\},
\]
where $\overline{\PR}[\ \cdot \ ]=\PR_0^K[\ \cdot\ | D=\infty]$ is defined in \eqref{romeo}.

\begin{remark}\label{tail2}
Note that, on $OLT(\delta,K,n)$, we have $c_*^{\mathrm{max}}=c_*^\omega(e^{(n)})$. Moreover, it is plain to see that that $\{c_*^{\mathrm{max}} \geq n\}=LT(n)$, hence $\overline{\PR}[c_*^{\mathrm{max}} \geq n]=\overline{\PR}[LT(n)]$.
\end{remark}

Define the constant
\begin{equation}\label{def_C1}
C_1=\frac{1}{\PR_0^K[D=\infty]} \sum_{e\in E(\Z^d)} {\PR}_0^K[ T_e<\tau_1^{Y^e},D^{Y^e}=\infty]\in(0,+\infty).
\end{equation}

We obtain the probability of encountering at least one trap in a regeneration period as a direct consequence of Proposition~\ref{big_trap_uncond}. 
\begin{lemma}\label{tailpn}
There exists $K_0<\infty$ such that, for any $K\ge K_0$, there exists $\eta>0$ such that
\[
\overline{\PR}[LT(n)] \times\left(C_1 {\bf P}[c_*\geq n]\right)^{-1}\in(1-n^{-\eta},1+n^{-\eta}).
\]
\end{lemma}

Note that the constant $C_1$ depends on $K$, and so does $\overline{\PR}$.

It is also important to know that it is unlikely to encounter more than two large traps when you know you are encountering one in the regeneration time.
\begin{lemma}\label{equiv_pnqn}
Fix $\delta>0$. There exists $K_0<\infty$ such that, for any $K\ge K_0$, there exists $\eta>0$ such that
\[
\frac{\overline{\PR}[LT(n)] }{\overline{\PR}[OLT(\delta,K,n)]}\in[1,1+n^{-\eta}).
\]
\end{lemma}
\begin{proof}
Obviously, we have that for any $\delta>0$ and $K<\infty$, we have
\begin{align*}
\overline{\PR}[OLT(\delta,K,n)] &
 \leq \overline{\PR}[LT(n)]\\
& \leq  \overline{\PR}[LT(n)\cap OLT(\delta,K,n)^c] +\overline{\PR}[OLT(\delta,K,n)],
\end{align*}
and the conclusion follows by~\eqref{thomas}.
\end{proof}

Using Proposition \ref{big_trap_uncond} and~\eqref{thomas} a simple computation yields the following statement.
\begin{lemma}\label{tom}
There exists $\eta>0$ such that, for any borelian set $F$,
\[
\PR_0^K[c_*^{\max}\in F,c_*^{\max}\ge n| D=\infty]=\frac{{\bf P}[c_*\in F,c_*\ge n]}{{\bf P}[c_*\ge n]}\overline{\PR}[LT(n)]+n^{-\eta}o(\overline{\PR}[LT(n)]),
\]
which implies, by Remark \ref{tail2},
\[
\PR_0^K[c_*^{\max}\in F|c_*^{\max}\ge n, D=\infty]=\frac{{\bf P}[c_*\in F,c_*\ge n]}{{\bf P}[c_*\ge n]}+o(n^{-\eta}).
\]
\end{lemma}

The following result is a simple consequence of Lemma~\ref{tom}, Proposition \ref{big_trap_uncond} and Lemma \ref{cvg_vpi}.
\begin{lemma}\label{kevin}
There exists $K_0<\infty$ such that, for any $K\ge K_0$, there exists $\eta>0$ such that for any  $n\in \N$, we have
\begin{align*}
& \max_{A,B,F} \left|\PR^\infty[V_{\infty} \in A, \overline{\pi}^{\infty}\in B,c_*^{\max}\in F|c_*^{\max}\ge n]\right.\\
&\left.-{\PR}_n[c_*^\omega({   e^{(n)} })\in F,V_{n} \in A, \overline{\pi}^{\omega_{e^{(n)}}}(x_{e^{(n)}}) \in B]\right| \\
=& o(n^{-\eta}),
\end{align*}
where the maximum is taken over all borelians. It implies that
\begin{align*}
& \max_{A,B} \left|\PR^\infty[W_\infty c_*^{\max}\1{c_*^{\max}\ge n}\in A, \tilde{W}_\infty c_*^{\max}\1{c_*^{\max}\ge n}\in B]\right.\\
&\left.-\overline{\PR}\left[c^\omega_*(e^{(n)})\frac{\sum_{i=1}^{V_n} 2{\bf e}_i}{\overline{\pi}^{\omega_{e^{(n)}}}(x_{e^{(n)}})}\1{LT(n)}\in A, c^\omega_*(e^{(n)})\frac{2{V_n} }{\overline{\pi}^{\omega_{e^{(n)}}}(x_{e^{(n)}})}\1{LT(n)}\in B\right]\right| \\
=& n^{-\eta}o(\overline{\PR}[LT(n)]).
\end{align*}
\end{lemma}
\begin{proof}
The proof comes easily once we have noticed that 
\begin{eqnarray*}
&&\PR^\infty[V_{\infty} \in A, \overline{\pi}^{\infty}\in B,c_*^{\max}\in F|c_*^{\max}\ge n]\\
&&=\PR_0^K[c_*^{\max}\in F|c_*^{\max}\ge n, D=\infty]\PR^\infty[V_{\infty} \in A, \overline{\pi}^{\infty}\in B],
\end{eqnarray*}
and that ${\PR}_0^K[LT(n)|D=\infty]=\PR_0^K[c_*^{\max}\ge n|D=\infty]$. Besides,  $\frac 1 {\overline{\pi}^{\infty}} \sum_{i=1}^{V_{\infty}}2{\bf e}_i$ has the same law as $W_{\infty}$ and $\frac{2V_{\infty}}{\overline{\pi}^{\infty}}$ has the same law as $\tilde{W}_{\infty}$. 
\end{proof}

\begin{remark}\label{one_more_damn_coupling}
The previous proposition states that the total variation distance between the law of $(c_*^{\max},V_{\infty},\pi^{\infty})$ chosen under $\PR^\infty[\cdot|c_*^{\max}\ge n]$ and $(c_*^\omega({   e^{(n)} }),V_{n},\overline{\pi}^{\omega_{e^{(n)}}}(x_{e^{(n)}}))$ chosen under ${\PR}_n$ is $o(n^{-\eta})$. This allows us to produce a coupling such that those two triplets of random variables do not coincide with probability  at most $o(n^{-\eta})$. In the same way, we can couple the $4$-uplets $(c_*^{\max},V_{\infty},\overline{\pi}^{\infty}, ({\bf e}_i))$ chosen under $\PR^\infty[\cdot|c_*^{\max}\ge n]$ and $(c_*^\omega({   e^{(n)} }),V_{n},\overline{\pi}^{\omega_{e^{(n)}}}(x_{e^{(n)}}) , ({\bf e'}_i))$ chosen under ${\PR}_n$ such that they do not coincide with probability  at most $o(n^{-\eta})$, and where $({\bf e}_i)$ and $({\bf e'}_i)$ are two i.i.d.~sequences of mean $1$ exponential variables which are independent of the three other quantities involved in their respective $4$-uplets.
\end{remark}

The following result will be used to prove the main theorem and gives an estimate of the distance between $c^\omega_*(e^{(n)})W_n$ and $c_*^{\mathrm{max}}W_{\infty}$.

\begin{proposition}\label{final_coupling}
Fix $\delta\in(0,1)$. There exist  $\eta>0$ and a constant $C'(d)$ such that, for any $n>K$, there exists a coupling $P^{n,\infty}$ of $(c^\omega_*(e^{(n)}),W_n,2V_n/\overline{\pi}^{\omega_{e^{(n)}}}(x_{e^{(n)}}))$ under $\overline{\PR}$ and $(c_*^{\mathrm{max}},W_{\infty},\tilde{W}_{\infty})$ under $\PR^\infty$ such that we have
\begin{align*}
& P^{n,\infty}\Bigl[ \left|c^\omega_*(e^{(n)})W_n\1{OLT(\delta,K,n)}- c_*^{\mathrm{max}}W_{\infty}\1{c_*^{\max}\ge n}\right|\\
&\qquad\qquad\qquad\qquad> C'(d) n^{\delta-1}c_*^{\mathrm{max}}(W_{\infty}+\tilde{W}_{\infty})\1{c_*^{\max}\ge n}\Bigr] \\
&\leq  n^{-\eta} o(\overline{\PR}[LT(n)]).
\end{align*}
\end{proposition}

\begin{proof}
Recall that, by~\eqref{thomas}, there exists $\epsilon>0$ such that $\PR_n[OLT(\delta,K,N)^c]=o(n^{-\epsilon})$. On $OLT(\delta,K,N)$, $c^\omega_*(e^{(n)})=c_*^{\mathrm{max}}$, and recalling~\eqref{decorwn} we can see that there exist exponential random variables ${\bf e'}_i$, independent of ${\pi^{\omega_{e^{(n)}}}(x_{e^{(n)}})}$ and $V_n$, such that
\begin{align}\label{fuck}
& (1-C(d)n^{\delta-1})\frac{1}{\overline{\pi}^{\omega_{e^{(n)}}}(x_{e^{(n)}})}\sum_{i=1}^{V_n} 2{\bf e'}_i  \\ \nonumber
\leq& W_n \leq   (1+C(d)n^{\delta-1})\frac{1}{\overline{\pi}^{\omega_{e^{(n)}}}(x_{e^{(n)}})}\sum_{i=1}^{V_n} (2{\bf e'}_i+C(d)n^{\delta-1}).
\end{align}
Now, by Lemma~\ref{kevin} we know that we can find
\[
(W_\infty c_*^{\max}\1{c_*^{\max}\ge n},\tilde{W}_\infty c_*^{\max}\1{c_*^{\max}\ge n})\]
which is coupled with high probability with
\[
(\frac{\sum_{i=1}^{V_{n}}2{\bf e}_i}{\overline{\pi}^{\omega_{e^{(n)}}}(x_{e^{(n)}})} \1{LT(n)},\frac{2V_n}{\overline{\pi}^{\omega_{e^{(n)}}}(x_{e^{(n)}})}\1{LT(n)}).
\] 
 Thus, we have built an adequate coupling and the conclusion follows easily.
\end{proof}

\subsection{Tail estimate of  the random variable $W_{\infty}c_*^{\mathrm{max}}$}
The goal of this section is to compute the tail of $W_\infty$. For this purpose, we need to give a moment on this variable. To bound the error terms from Proposition \ref{final_coupling}, we will also need to compute the tail the following random variable:

\[
\tilde{W}_{\infty}:=2\frac{V_{\infty}}{\overline{\pi}_{\infty}}.
\]

\begin{lemma}\label{moment_winf}
For any $\epsilon\in[0,1-\gamma)$, we have
\[
\ES^\infty[W_{\infty}^{\gamma+\epsilon}]<\infty,
\]
and
\[
\ES^\infty[\tilde{W}_{\infty}^{\gamma+\epsilon}]<\infty.
\]
\end{lemma}
\begin{proof}
Let us emphasize that we do not know yet that $W_{\infty}$ is almost surely finite. We will prove it by a moment estimate.\\
Define $\epsilon'=\epsilon+(1-\gamma-\epsilon)/2$. Using Lemma~\ref{lem_momWn}, for any $n>K$ and for any $R$,
\begin{eqnarray*}
&&{\ES}_n[(W_n\1{OLT(\delta,K,n)})^{\gamma+\epsilon}\1{(W_n\1{OLT(\delta,K,n)})^{\gamma+\epsilon}\ge R}]\\
&\le& {\ES}_n[(W_n\1{OLT(\delta,K,n)})^{\gamma+\epsilon'}]^{ \frac{\gamma+\epsilon}{\gamma+\epsilon'} }{\PR}_n[{(W_n\1{OLT(\delta,K,n)})^{\gamma+\epsilon}\ge R}]^{ \frac{\epsilon'-\epsilon}{\gamma+\epsilon'} }\\
&\le& C(K,\epsilon,\gamma)R^{-r},
\end{eqnarray*}
where $r$ is a constant which does not depend on $n$ and where we used  H\"older's inequality and Markov's inequality. This means that  the quantity
\[
{\ES}_n[(W_n\1{OLT(\delta,K,n)})^{\gamma+\epsilon}\1{(W_n\1{OLT(\delta,K,n)})^{\gamma+\epsilon}\le R}]
\]
converges uniformly in $n$ to ${\ES}_n[(W_n\1{OLT(\delta,K,n)})^{\gamma+\epsilon}]$ as $R$ goes to infinity. Moreover, using Lemma \ref{not_2_trapsb} and Lemma \ref{moment_winf}, it is easy to prove that $W_n\1{OLT(\delta,K,n)}$ under ${\PR}_n$ converges in law to $W_\infty$, hence we have
\begin{eqnarray*}
&&\lim_{n\to\infty}{\ES}_n[(W_n\1{OLT(\delta,K,n)})^{\gamma+\epsilon}\1{(W_n\1{OLT(\delta,K,n)})^{\gamma+\epsilon}\le R}]\\
&&={\ES}^\infty[W_\infty^{\gamma+\epsilon}\1{W_\infty^{\gamma+\epsilon}\le R}].
\end{eqnarray*}
Finally, we conclude, by switching the limits in $n$ and in $R$,
 \begin{eqnarray*}
&&{\ES}^\infty[W_\infty^{\gamma+\epsilon}]=\lim_{n\to\infty}{\ES}_n[(W_n\1{OLT(\delta,K,n)})^{\gamma+\epsilon}]\le C(K).\\
\end{eqnarray*}
A similar result now holds easily for $\tilde{W}_\infty$.
%
\end{proof}

Finally, we are able to estimate de tail of $W_{\infty}c_*^{\mathrm{max}}$.

\begin{lemma}\label{final_tail}
We have
\[
\PR^{\infty}[W_{\infty}c_*^{\mathrm{max}} \geq t]\sim C_1 \ES^\infty[W_{\infty}^{\gamma}]L(t)t^{-\gamma},
\]
and
\[
\PR^{\infty}[\tilde{W}_{\infty}c_*^{\mathrm{max}} \geq t]\sim C_1 \ES^\infty[\tilde{W}_{\infty}^{\gamma}]L(t)t^{-\gamma},
\]
\end{lemma}
\begin{proof}
Using Lemma~\ref{tailpn} and Remark~\ref{tail2} we see that 
\[
{\PR}^\infty[c_*^{\mathrm{max}} \geq t]=\overline{\PR}[c_*^{\mathrm{max}} \geq t]\sim C_1L(t)t^{-\gamma}.
\]
Hence, by using Breiman's Theorem (which is proved for example in~\cite{CS}, see Corollary 3.6 $(iii)$), we obtain 
\[
\PR^\infty[c_*^{\max}W_{\infty}\geq t] \sim  C_1L(t)t^{-\gamma} \ES^\infty[W_{\infty}^{\gamma}],
\]
since $\ES^\infty[W_{\infty}^{\gamma+\epsilon}]<\infty$ for $\epsilon\in[0,1-\gamma)$, by Lemma~\ref{moment_winf}.\\
The proof for $\tilde{W}_\infty$ is the same.
%
%

\end{proof}

\section{Limit theorems}\label{sect_limth}

We now have all the necessary results to conclude the main results.

Recall that we assumed that ${\bf P}[c_*(e)\ge t ]= L(t) t^{-\gamma}$ for some $\gamma\in(0,1)$ and where $L$ is a slowly varying function. Note that
\[
\tau_n=\tau_1+\sum_{i=1}^{n-1} \left(\tau_{i+1}-\tau_{i}\right),
\]
where, using Remark \ref{rem_regensup} and Theorem \ref{thindep},  the quantities $\left(\tau_{i+1}-\tau_{i}\right)$, $i\ge 1$, under $\PR_0$, are all independent, independent of $\tau_1$ and distributed as $\tau_1$ under $\PR_0^K[\cdot|D=\infty]$.\\

Let us define the generalized inverse of the tail of $c_*$, composed with $t\mapsto 1/t$, that is
\begin{align}\label{defInv}
\mathrm{Inv}(t):=\inf\left\{x:{\bf P}[c_*>x]\le t^{-1}\right\}.
\end{align}
Using that $\mathrm{Inv}(\cdot)$ is nondecreasing and that $t^rL(t)\to+\infty$ for any $r>0$, one can easily prove that
\begin{align} \label{retard}
\frac{n^{\frac{1}{\gamma}-r}}{\mathrm{Inv}(n)}\longrightarrow 0\text{, for any }r>0.
\end{align}

\subsection{Identifying the terms that do not contribute}

We are interested in the scaling limit of $\tau_n/   \mathrm{Inv}(n)$ under $\PR_0$. Recall that the quantities $\tau_{i+1}-\tau_{i}$ for $i\ge 1$ are distributed like $\tau_1$ under $\PR_0^K[\cdot|D=\infty]$, are independent of each other and independent of $\tau_1$.\\
As $\tau_1<\infty$ $\PR$-a.s.~and $\PR_0^K[\cdot|D=\infty]$-a.s., it is equivalent to look for the scaling limit of
\[
\sum_{i=1}^{n}\frac{\tau_1^{(i)}}{   \mathrm{Inv}(n)},
\]
where the $\tau_1^{(i)}$'s are i.i.d.~copies of $\tau_1$ under $\PR_0^K[\cdot|D=\infty]$. We will keep all previous notations adding naturally a superscript or subscript $(i)$ to notify that the quantity is related to the $i$-th copy of the first regeneration block under $\PR_0^K[\cdot|D=\infty]$. We will still denote $\PR_0^K[\cdot|D=\infty]$ their common probability measure.

We are thus interested in the time spent during one regeneration period under $\PR_0^K[\cdot|D=\infty]$. As explained before, the time spent by the walker during one regeneration period is not negligible only if he meets an edge with large conductance and, in this case, he spends essentially all of his time on this edge.

First we prove the following result.
\begin{lemma}\label{gladiator}
For any $\delta>0$, there exists $K_0<\infty$ such that, for any $K>K_0$,
\[
\ES_0^K[\tau_1^{\gamma-\delta}|D=\infty]<\infty.
\]
\end{lemma}
\begin{proof}
Recalling the definitions \eqref{def_tausup}, \eqref{def_tauinf} and \eqref{def_Wblabla}, note that, for any $\epsilon>0$, under $\PR_0^K[\cdot|D=\infty]$,
\begin{align}\nonumber
\tau_1=&\tau_1^{\ge n^\frac{1-\epsilon}{\gamma}}\1{OLT(1/2,K,n^\frac{1-\epsilon}{\gamma})}+\tau_1^{< n^{\frac{1-\epsilon}{2\gamma}}}\1{OLT(1/2,K,n^\frac{1-\epsilon}{\gamma})}\\
\nonumber&+\tau_1\1{LT(n^\frac{1-\epsilon}{\gamma})^c}+\tau_1\1{LT(n^\frac{1-\epsilon}{\gamma})\cap OLT(1/2,K,n^\frac{1-\epsilon}{\gamma})^c}\\
\nonumber=& W_{  n^\frac{1-\epsilon}{\gamma} }c_*^\omega(e^{(n^\frac{1-\epsilon}{\gamma})}) \1{OLT(1/2,K,n^\frac{1-\epsilon}{\gamma})}+\tau_1^{< n^{\frac{1-\epsilon}{2\gamma}}}\1{OLT(1/2,K,n^\frac{1-\epsilon}{\gamma})}\\
\label{furio}&+\tau_1^{< n^\frac{1-\epsilon}{\gamma}   }\1{LT(n^\frac{1-\epsilon}{\gamma})^c}+\tau_1\1{LT(n^\frac{1-\epsilon}{\gamma})\cap OLT(1/2,K,n^\frac{1-\epsilon}{\gamma})^c}.
\end{align}
Now, for any $k\in\N$, we can use Lemma~\ref{decaytau} and Lemma~\ref{tailpn} and prove that
\begin{align*}
\PR_0^K[\tau_1^{\gamma-\delta}>k|D=\infty] \le& \PR_0^K[\tau_1^{< k^{\frac{1}{\gamma-\delta^2}}}>k^{\frac{1}{\gamma-\delta}}|D=\infty] + \PR_0^K[LT(k^{\frac{1}{\gamma-\delta^2}})|D=\infty] \\
\le& Ck^{-1-\frac{\delta}{\gamma-\delta}}+ C k^{-1-\frac{\delta^2}{2(\gamma-\delta^2)}},
\end{align*}
which is summable. This implies the result.
\end{proof}

Now let us prove that the time is overwhelmingly spent on large edges. The quantity $1/2$ in the following result as to be seen as some number strictly less than $1$. We choose this specific value only to avoid useless notation.
\begin{proposition}\label{justin}
For any $\epsilon\in(0,1/6)$, there exists $K_0<\infty$ such that, for any $K\ge K_0$, we have
\[
\frac{\tau_n-\sum_{i=1}^{n}W_{  n^\frac{1-\epsilon}{\gamma} }^{(i)}c_*^\omega(e_{(i)}^{(n^\frac{1-\epsilon}{\gamma})}) \1{OLT^{(i)}(\frac{1}{2},K,n^\frac{1-\epsilon}{\gamma})}}{   \mathrm{Inv}(n)}\xrightarrow{(d)}0.
\]
\end{proposition}
\begin{proof}
Using that $\tau_1<\infty$ $\PR$-a.s. and $\PR_0^K[\cdot|D=\infty]$-a.s., and using \eqref{retard} and \eqref{furio}, it is enough to prove that, for some $r>0$,
\begin{align}\label{radio}
n^{-\frac{1}{\gamma}+r}\sum_{i=1}^{n}\left[{\tau_1^{(i)}}-W_{  n^\frac{1-\epsilon}{\gamma} }^{(i)}c_*^\omega(e_{(i)}^{(n^\frac{1-\epsilon}{\gamma})}) \1{OLT^{(i)}(1/2,K,n^\frac{1-\epsilon}{\gamma})}\right]\longrightarrow 0,
\end{align}
in probability.\\

For notational simplicity, let us write 
\[
f_{i,n}:= {\tau_1^{(i)}}-W_{  n^\frac{1-\epsilon}{\gamma} }^{(i)}c_*^\omega(e_{(i)}^{(n^\frac{1-\epsilon}{\gamma})}) \1{OLT^{(i)}(1/2,K,n^\frac{1-\epsilon}{\gamma})}
\]
for the $i$-th term of the sum in~\eqref{radio}. Note that by Lemma~\ref{gladiator}, the $f_{i,n}$'s, $1\le i \le n$, are i.i.d., have a moment $\gamma-\delta$ for any $\delta>0$ and their tails can be upper bounded using \eqref{furio}.\\

The rest of the proof is made of three main steps. Before detailing them, let us define constants that will be useful:
\begin{align}\label{poinconneur}
M:= 4\frac{2+\gamma}{\eta_1},\
\eta_1:= (1-\gamma)\eta_2,\
\eta_2:= \frac{1-\gamma}{3+\gamma}\epsilon,
\end{align}
and note that each of these constants depends only on $\gamma$ and $\epsilon$.\\
Moreover, the following inequality will be used several times. For any $a\in\R$,
\begin{align}\nonumber
   \ES_0^K\left[\left.\mathrm{card}\{1\leq j\leq n,f_{j,n}\geq a\}\right|D=\infty\right]
     =& \ES_0^K\left[\left.\sum_{{1\leq j\leq n}}  \1{f_{j,n}\geq a}\right|D=\infty\right] \\ 
 \leq & n \PR_0^K\left[\left.f_{1,n}\geq a\right|D=\infty\right].\label{serge}
\end{align}

\vspace{0.5cm}

{\it Step 1: controlling terms with small or medium conductances }

\vspace{0.5cm}

Using Lemma~\ref{gladiator} and Markov's inequality, we have, for any $i\in\{0,...,M\}$,
\[
\PR_0^K[f_{1,n}\geq n^{i/(M\gamma)}|D=\infty]\leq C(K,M) n^{-(\gamma-1/M)(i/(M\gamma))}.
\]

Using \eqref{serge} and Markov's inequality, we obtain for $0\le i\leq M$
\begin{align*}
&\PR_0^K\left[\left.\text{card}\{1\leq j \leq n,f_{j,n}\geq n^{i/(M\gamma)}\}\geq\frac{n^{\frac{1-\eta_1}{\gamma} +\frac{2+\gamma}{M\gamma}-\frac{i+1}{M\gamma}}}{2M} \right|D=\infty\right] \\
\leq & C(K,M) n n^{-(\gamma-\frac{1}{M})\frac{i}{M\gamma}}n^{-\frac{1-\eta_1}{\gamma} -\frac{2+\gamma}{M\gamma}+\frac{i+1}{M\gamma}} \\
\leq & C(K,M) n^{\frac{\eta_1}{\gamma}-(\frac{1}{\gamma}-1)(1-\frac{i}{M})-\frac{1}{M}}.
\end{align*}

For any $0\le i \le M$, we define the event 
\begin{align*}
&B(n,i, M,\eta_1)\\
:=& \Bigl \{\text{card} \{1\leq j \leq n, f_{j,n}\in (n^{i/(M\gamma)},n^{(i+1)/(M\gamma)}]\}    \geq\frac 1 {2M} n^{\frac{1-\eta_1}{\gamma} +\frac{2+\gamma}{M\gamma}-\frac{i+1}{M\gamma}}\Bigr\}.
\end{align*}
Together with the fact that $i\leq M (1-\eta_1/(1-\gamma))$ is equivalent to $\eta_1/\gamma-(1/\gamma-1)(1-i/M)\leq 0$, this yields that, for any fixed $M$,
\begin{align}
\label{clyde}
\PR_0^K[B(n,i,M,\eta_1)|D=\infty]\leq& C(K,M) n^{-1/M}=o(1),
\end{align}
for any  $i\leq \left\lfloor M \left(1-\frac{\eta_1}{1-\gamma}\right)\right\rfloor=: i_{\max}$.

\vspace{0.5cm}

{\it Step 2: ruling out terms with large conductances }

\vspace{0.5cm}

Let us define the event
\begin{align*}
B'(n,\eta_2) & =\left\{\text{card}\{1\leq j \leq n,f_{j,n} \geq n^{\frac{1-\eta_2}{\gamma}}\}\geq 1\right\}.
\end{align*}

Recalling \eqref{serge} and using \eqref{furio}, we see that
\begin{align*}
 \PR_0^K[B'(n,\eta_2)|D=\infty] \leq &  n \PR_0^K[f_{1,n}\geq n^{\frac{1-\eta_2}{\gamma}}|D=\infty] \\
 \leq & n \left(   \PR_0^K\left[\left.\tau_1^{< n^{ \frac{1-\epsilon}{2\gamma} } } \ge \frac{1}{3}n^{\frac{1-\eta_2}{\gamma}}\right|D=\infty\right]\right.      \\
 &\qquad +\PR_0^K\left[\left.\tau_1^{< n^{ \frac{1-\epsilon}{\gamma} } } \ge \frac{1}{3}n^{\frac{1-\eta_2}{\gamma}}\right|D=\infty\right]\\
 &\left.\qquad +\PR_0^K\left[\left. LT(n^{ \frac{1-\epsilon}{\gamma} })\cap OLT(1/2,K,n^{ \frac{1-\epsilon}{\gamma} })^c\right|D=\infty\right]\right)\\
 \le& Cn\left( 2n^{ \frac{1-\eta_2}{\gamma}\left[ -\gamma-\frac{ \left(1-\frac{1-\epsilon}{1-\eta_2}\right)(1-\gamma)   }{2    }   \right]   } + n^{\frac{\epsilon}{2}+\frac{1-\epsilon}{2}-(1-\epsilon)}\right),
 \end{align*} 
where we used Lemma~\ref{decaytau} twice  and Lemma~\ref{not_2_trapsb}. Finally, using \eqref{poinconneur} and the fact that $\eta_2< \epsilon< 1/6$, we have
\begin{align}\label{bonnie}
& \PR_0^K[B'(n,\eta_2)|D=\infty]\le Cn\times n^{-1-\eta_2}\le Cn^{-\eta_2}=o(1).
\end{align}

\vspace{0.5cm}

{\it Step 3: conclusion}

\vspace{0.5cm}

Recall that $i_{\max}=\left\lfloor M \left(1-\eta_2\right)\right\rfloor$ so that $(i_{\max}+1)/(M\gamma)\ge (1-\eta_2)/\gamma$. Now, define the event
\[
B(n,M)=B'(n,\eta_2)\cup\bigcup_{j=0}^{i_{\max}} B(n,i,M,\eta_1).
\]
Using \eqref{clyde} and \eqref{bonnie}, we have
\[
\PR_0^K[B(n,M)|D=\infty]=o(1).
\]
This implies that 
\[
n^{-\frac{1}{\gamma}+\frac{\eta_1}{4}}\1{B(n,M)}\sum_{j=1}^{n} f_{j,n}\longrightarrow 0
\]
in probability.\\
On the other hand, on $B(n,M)^c$, we can give an upper bound
\begin{align*}
\sum_{j=1}^{n} f_{j,n}\le& \sum_{j=1}^{i_{\max}}\frac{1}{2M} n^{\frac{1-\eta_1}{\gamma}+\frac{2+\gamma}{M\gamma}}\le n^{\frac{1-\eta_1}{\gamma}+\frac{2+\gamma}{M\gamma}}=n^{\frac{1}{\gamma} -\frac{3\eta_1}{4}}.
\end{align*}
This implies that
\[
n^{-\frac{1}{\gamma}+\frac{\eta_1}{4}}\1{B(n,M)^c}\sum_{j=1}^{n} f_{j,n} \longrightarrow 0
\]
in probability. This concludes the proof.
\end{proof}


\subsection{Scaling limits of the asymptotic environment}
In this section, we state the following results about the scaling limits of i.i.d.~copies of the variable $c_{*}^{\mathrm{max}}W_{\infty}$. In Section \ref{jaienviededormir}, we will then prove that $\tau_n$ has the same limit.
 \begin{proposition}\label{jackwhite}
 Under $\PR^\infty$, we have
 \begin{align*}
&\frac{\sum_{i=1}^{n} c_{*,(i)}^{\mathrm{max}}W^{(i)}_{\infty}}{   \mathrm{Inv}(n)}\xrightarrow{(d)} C_\infty S_\gamma\text{ and }\frac{\sum_{i=1}^{n} c_{*,(i)}^{\mathrm{max}}\tilde{W}^{(i)}_{\infty}}{   \mathrm{Inv}(n)}\xrightarrow{(d)} \tilde{C}_\infty S_\gamma,
\end{align*}
where $S_\gamma$ has a completely asymmetric stable law of index $\alpha$ and where the constants are defined by:
\begin{align}\label{agraffes}
C_\infty:=\left(C_1\ES^\infty[W_\infty^\gamma]\right)^{1/\gamma} \text{ and }
\tilde{C}_\infty:=\left(C_1\ES^\infty[\tilde{W}^\gamma_\infty]\right)^{1/\gamma},
\end{align}
with $C_1$ being the constant defined in \eqref{def_C1}.
 \end{proposition}

 \begin{proof}
Let us explain it for the first case, the second being similar.
We have to deal with a sum of i.i.d.~random variables whose tails are heavy and are, by Lemma~\ref{final_tail}, such that
\[
\PR^\infty[c_*^{\max}W_\infty\ge t]\sim C_\infty^\gamma L(t)t^{-\gamma}.
\]
This is equivalent to say that there exists a slowly-varying function $\tilde{L}$ such that, for any $t\ge 0$,
\[
\PR^\infty[c_*^{\max}W_\infty\ge t]= \tilde{L}(t)t^{-\gamma}.
\]
Let us denote $\widetilde{\mathrm{Inv}}(\cdot)$ the generalized inverse function of this tail, composed with $t\mapsto1/t$, as in \eqref{defInv}. Using classical results about sums of i.i.d.~heavy-tailed random variables (see Theorem 3.7.2, p.161 of \cite{Durrett}), we have that
\[
\frac{\sum_{i=1}^{n} c_{*,(i)}^{\mathrm{max}}W^{(i)}_{\infty}}{   \widetilde{\mathrm{Inv}}(n)}\xrightarrow{(d)} S_{\gamma}.
\]
Now, using the properties of slowly-varying functions and using the monotonicity of $\mathrm{Inv}$ and $\widetilde{\mathrm{Inv}}$, one can easily show  that, for any $\delta>0$,  $(1-\delta)C_\infty\mathrm{Inv}(n)\le \widetilde{\mathrm{Inv}}(n) \le (1+\delta)C_\infty \mathrm{Inv}(n)$, as soon as $n$ is large enough, hence
\[
\frac{\widetilde{\mathrm{Inv}}(n)}{   \mathrm{Inv}(n) }\to C_\infty.
\]
Finally, as we deal with sums of non-negative random variables, $S_\gamma$ is necessarily completely asymmetric (i.e.~supported by the non-negative real numbers).
 \end{proof}
 
 The last Proposition obviously implies the following lemma.
  \begin{lemma}\label{ggggg}
 Under $\PR^\infty$, we have, for any $\delta>0$,
 \begin{align*}
&n^{ - \delta}\frac{\sum_{i=1}^{n} c_{*,(i)}^{\mathrm{max}}W^{(i)}_{\infty}}{   \mathrm{Inv}(n)}\xrightarrow{(d)} 0\text{ and }n^{ - \delta}\frac{\sum_{i=1}^{n} c_{*,(i)}^{\mathrm{max}}\tilde{W}^{(i)}_{\infty}}{   \mathrm{Inv}(n)}\xrightarrow{(d)} 0.
\end{align*}
 \end{lemma}

The following result shows that only terms associated to large conductances contribute to the limits stated in Proposition \ref{jackwhite}.

\begin{lemma}\label{neglect_small_cond}
For any $\epsilon\in(0,1/6)$, we have 
\[
\frac{\sum_{i=1}^{n} c_{*,(i)}^{\mathrm{max}}W^{(i)}_{\infty}\1{c_{*,(i)}^{\max}\ge n^{\frac{1-\epsilon}{\gamma}}}-\sum_{i=1}^{n} c_{*,(i)}^{\mathrm{max}}W^{(i)}_{\infty}}{   \mathrm{Inv}(n)}\xrightarrow{(d)}0,
\]
and
\[
\frac{\sum_{i=1}^{n} c_{*,(i)}^{\mathrm{max}}\tilde{W}^{(i)}_{\infty}\1{c_{*,(i)}^{\max}\ge n^{\frac{1-\epsilon}{\gamma}}}-\sum_{i=1}^{n} c_{*,(i)}^{\mathrm{max}}\tilde{W}^{(i)}_{\infty}}{   \mathrm{Inv}(n)}\xrightarrow{(d)}0.
\]
 \end{lemma}
\begin{proof}
Using \eqref{retard}, we can conclude if we prove that, for some $r>0$,
  \begin{align}\label{dontbelieve}
n^{-\frac{1}{\gamma}+r}\sum_{i=1}^{n} c_{*,(i)}^{\max} W^{(i)}_{\infty}\1{c_{*,(i)}^{\max}< n^{\frac{1-\epsilon}{\gamma}}}\xrightarrow{(d)}{} 0,
 \end{align}
 and
\[
n^{-\frac{1}{\gamma}+r}\sum_{i=1}^{n} c_{*,(i)}^{\max} \tilde{W}^{(i)}_{\infty}\1{c_{*,(i)}^{\max}<  n^{\frac{1-\epsilon}{\gamma}}}\xrightarrow{(d)}{} 0.
 \]
The proof is very close to the proof of Proposition~\ref{justin}.
Let us do the proof of the first equation, the second being similar.
As in \eqref{poinconneur}, let us define the following constants which depend only on $\gamma$ and $\epsilon$:
\begin{align}\label{poinconneur2}
M:= 4\frac{2+\gamma}{\eta_1},\
\eta_1:= (1-\gamma)\eta_2,\
\eta_2:=  \frac{1-\gamma}{6(1+\gamma)}\epsilon.
\end{align}
Let us also define the following shorthand notation for the $i$-th term of the sum in \eqref{dontbelieve}
\[
f_{i,n}:=c_{*,(i)}^{\max} W^{(i)}_{\infty}\1{c_{*,(i)}^{\max}< n^{\frac{1-\epsilon}{\gamma}}},
\]
so that the $f_{i,n}$'s, $1\le i\le n$, are i.i.d.~and, using Lemma~\ref{final_tail}, we have 
\[
\ES^\infty[(f_{1,n})^{\gamma-\frac{1}{M}}]\leq \ES^\infty[(c_{*,(i)}^{\max} W^{(i)}_{\infty})^{\gamma-\frac{1}{M}}] < C(\gamma,M)<\infty.
\]

\vspace{0.5cm}

{\it Step 1: controlling terms with small or medium conductances }

\vspace{0.5cm}

%

For any $0\le i \le M$, we define the event 
\begin{align*}
&B(n,i, M,\eta_1)\\
:=& \Bigl \{\text{card} \{1\leq j \leq n, f_{j,n}\in (n^{i/(M\gamma)},n^{(i+1)/(M\gamma)}]\}    \geq\frac 1 {2M} n^{\frac{1-\eta_1}{\gamma} +\frac{2+\gamma}{M\gamma}-\frac{i+1}{M\gamma}}\Bigr\}.
\end{align*}
Proceeding as in the Step $1$ of the proof of Proposition~\ref{justin}, we obtain that, for any fixed $M$,
\begin{align}
\label{clyde2}
\PR^\infty[B(n,i,M,\eta_1)]\leq& C(M) n^{-1/M}=o(1),
\end{align}
for any  $i\leq \left\lfloor M \left(1-\frac{\eta_1}{1-\gamma}\right)\right\rfloor=: i_{\max}$.

\vspace{0.5cm}

{\it Step 2: ruling out terms with large conductances }

\vspace{0.5cm}

Denote, for $\eta_2>0$,
\begin{align*}
B'(n,\eta_2) & =\{\text{card}\{1\leq j \leq n, f_{j,n} \geq n^{(1-\eta_2)/\gamma}\}\geq 1\}.
\end{align*}

Recalling \eqref{serge}, the fact that, by Lemma~\ref{final_tail}, $c_{*}^{\max}W_\infty$ has a slowly-varying tail such that $\PR^\infty[c_{*}^{\max}W_\infty  \geq t]\leq t^{-\gamma+\gamma\eta_2}$ asymptotically, and using Breiman's Theorem (see~\cite{CS}, see Corollary 3.6 $(iii)$), we have
\begin{align*}
& \PR^\infty[B'(n,\eta_2)]\\
& \leq   n \PR^\infty[c_{*}^{\max} {W}^{}_{\infty}\geq n^{(1-\eta_2)/\gamma}, c_{*}^{\max}< n^{\frac{1-\epsilon}{\gamma}}] \\
&  \leq   n \PR^\infty[c_{*}^{\max} {W}^{}_{\infty}\1{W_{\infty} \geq n^{\frac{\epsilon-\eta_2}{\gamma}}}\geq n^{(1-\eta_2)/\gamma}] \\
& \leq C n  n^{-(1-\eta_2)^2} \PR^\infty[ W_{\infty} \geq n^{\frac{\epsilon-\eta_2}{\gamma}} ] \\
& \le  C n^{2\eta_2} n^{-(\epsilon-\eta_2)}\\
&\le Cn^{-\epsilon/2},
 \end{align*}
%
%
%
where we used that $\eta_2\le \epsilon/6$. We have thus proved that
\begin{equation}\label{alma}
\PR^\infty[B'(n,\eta_2)]=o(1).
\end{equation}

\vspace{0.5cm}

{\it Step 3: conclusion}

\vspace{0.5cm}
We conclude in the exact same manner as the step $3$ in the proof of Proposition~\ref{justin} that
\begin{align*}
n^{-\frac{1}{\gamma}+\frac{\eta_1}{4}}\sum_{j=1}^{n} c_{*,(j)}^{\max} {W}^{(j)}_{\infty}\1{c_{*}^{\max}< n^{\frac{1-\epsilon}{\gamma}}}&\longrightarrow 0
\end{align*}
in probability.

\end{proof}

\subsection{Coupling and conclusion}\label{jaienviededormir}
We finally are able to state and prove the scaling limit of $\tau_n$.
\begin{proposition}\label{justin2}
We have
\[
\frac{\tau_n}{   \mathrm{Inv}(n)}\xrightarrow{(d)}  C_\infty S_\gamma,
\]
where $S_\gamma$ has a completely asymmetric stable law of index $\alpha$, and where $C_\infty$ is the constant defined in \eqref{agraffes}.
\end{proposition}
\begin{proof}

By Lemma~\ref{justin}, Proposition \ref{jackwhite} and Lemma \ref{neglect_small_cond}, we only need to prove that, under some coupling,
\begin{align}\label{bougetoi}
&\left|\frac{\sum_{i=1}^{n}W_{  n^\frac{1-\epsilon}{\gamma} }^{(i)}c_*^\omega(e_{(i)}^{(n^\frac{1-\epsilon}{\gamma})}) \1{OLT^{(i)}(\frac{1}{2},K,n^\frac{1-\epsilon}{\gamma})}}{   \mathrm{Inv}(n)}\right.\\ \nonumber
&\qquad\qquad\qquad\left.-\frac{\sum_{i=1}^{n} c_{*,(i)}^{\mathrm{max}}W^{(i)}_{\infty}\1{c_{*,(i)}^{\max}\ge n^{\frac{1-\epsilon}{\gamma}}}}{   \mathrm{Inv}(n)}\right|\to0,
\end{align}
in probability.\\
By Proposition~\ref{final_coupling}, there exists $\eta>0$ such that, for any $\epsilon$, there exists a coupling $P^{n,\infty}$ such that
\begin{align*}
& P^{n,\infty}\Bigl[ \left|c^\omega_*(e^{(n^{  \frac{1-\epsilon}{\gamma} })})W_n\1{OLT(\frac{1}{2},K,n^{  \frac{1-\epsilon}{\gamma}})}- c_*^{\mathrm{max}}W_{\infty}\1{c_*^{\max}\ge n^{  \frac{1-\epsilon}{\gamma}}}\right|\\
&\qquad\qquad\qquad\qquad> C'(d) n^{ - \frac{1-\epsilon}{2\gamma}}c_*^{\mathrm{max}}(W_{\infty}+\tilde{W}_{\infty})\1{c_*^{\max}\ge n^{  \frac{1-\epsilon}{\gamma}}}\Bigr] \\
&\leq n^{ -\eta \frac{1-\epsilon}{\gamma}} o(\overline{\PR}[LT(n^{  \frac{1-\epsilon}{\gamma}})]).
\end{align*}
By choosing $\epsilon$ small enough compared to $\eta$ and using that $\overline{\PR}[LT(n^{  \frac{1-\epsilon}{\gamma}})]\le Cn^{-1+2\epsilon}$, this implies 
\begin{align}\label{ggggg2}
& P^{n,\infty}\Bigl[ U_n\Bigr] \leq Cn^{ -\frac{\eta}{2\gamma}},
\end{align}
where we define the event
\begin{align*}
U_n&:=\Bigg\{\exists i\in\{1,...,n\}: \left|c^\omega_{*}(e_{(i)}^{(n^{  \frac{1-\epsilon}{\gamma} })})W^{(i)}_n\1{OLT(\frac{1}{2},K,n^{  \frac{1-\epsilon}{\gamma}})}\right.\\
&\left.- c_{*,(i)}^{\mathrm{max}}W^{(i)}_{\infty}\1{c_{*,(i)}^{\max}\ge n^{  \frac{1-\epsilon}{\gamma}}}\right|> C'(d) n^{ - \frac{1-\epsilon}{2\gamma}}c_{*,(i)}^{\mathrm{max}}(W^{(i)}_{\infty}+\tilde{W}^{(i)}_{\infty})\1{c_{*,(i)}^{\max}\ge n^{  \frac{1-\epsilon}{\gamma}}}\Bigg\}.
\end{align*}
Let us denote $G_n$ the left-hand side of \eqref{bougetoi}, we have that
\begin{align*}
G_n\le &G_n\1{U_n}+C'(d) n^{ - \frac{1-\epsilon}{2\gamma}}\frac{\sum_{i=1}^{n} c_{*,(i)}^{\mathrm{max}}W^{(i)}_{\infty}\1{c_{*,(i)}^{\max}\ge n^{\frac{1-\epsilon}{\gamma}}}}{   \mathrm{Inv}(n)}  \\
&+C'(d) n^{ - \frac{1-\epsilon}{2\gamma}}\frac{\sum_{i=1}^{n} c_{*,(i)}^{\mathrm{max}}\tilde{W}^{(i)}_{\infty}\1{c_{*,(i)}^{\max}\ge n^{\frac{1-\epsilon}{\gamma}}}}{   \mathrm{Inv}(n)},                        
\end{align*}
and the quantities on the right-hand side go to $0$ in probability by \eqref{ggggg2} and Lemma \ref{ggggg}.
\end{proof}

%

\section{Process convergence}
We here give functional statements of scaling limit results. The strategy is quite classical and consists of two main ingredients. First, we prove the joint-convergence of the trajectory (or its deviation) and the clock process, see Lemma \ref{onnaitonvitonmeurt} . Second, we use an inversion argument to conclude the main results, see Theorem \ref{seinfeld}.\\
In order to prove Lemma \ref{onnaitonvitonmeurt}, we first need to prove that the position of the walker depends mostly on regeneration blocks without large traps.\\
For this purpose, we denote $OLT_i$ the event $OLT$, defined in \eqref{def_OLT}, associated with the $i$-th regeneration block. Besides, as before, we write $(OLT^{(i)})_i$ a sequence of i.i.d.~events distributed as $OLT$ under $\PR_0^K[\cdot|D=\infty]$.

\begin{lemma}\label{approx_displ}
Fix $\epsilon\in(0,1/12)$. There exists $K_0<\infty$, such that for any $K\ge K_0$, for any $t\ge0$, we have
\begin{align*}
&\lim_{n\to\infty}\PR_0\left[\abs{\abs{X_{\tau_{\lfloor tn \rfloor}} -\sum_{i=0}^{\lfloor tn \rfloor-1} (X_{\tau_{i+1}}-X_{\tau_i})\1{(OLT_{i+1}(1/2,K,(tn)^{ \frac{1-\epsilon}{\gamma} }))^c}}}_{\infty}\geq n^{1/4}\right]\\
&\quad =0.
\end{align*}
\end{lemma}
\begin{proof}
Firstly, we see that  
\begin{align*}
&\PR_0\left[\abs{\abs{X_{\tau_{\lfloor tn \rfloor}} -\sum_{i=0}^{\lfloor tn \rfloor-1} (X_{\tau_{i+1}}-X_{\tau_i})\1{(OLT_{i+1}(1/2,K,(tn)^{ \frac{1-\epsilon}{\gamma} }))^c}}}_{\infty}\geq n^{1/4}\right]\\
\le& \PR_0\left[\abs{\abs{X_{\tau_1}\1{OLT(1/2,K,(tn)^{ \frac{1-\epsilon}{\gamma} })}}}_\infty\ge \frac{n^{1/4}}{2}\right]\\
&+ \PR_0^K\left[\left.\abs{\abs{\sum_{i=1}^{\lfloor tn \rfloor-1} X^{(i)}_{\tau^{(i)}_1}\1{OLT^{(i)}(1/2,K,(tn)^{ \frac{1-\epsilon}{\gamma} })}}}_\infty\ge \frac{n^{1/4}}{2}\right|D=\infty\right],
\end{align*}
where the variables $X^{(i)}_{\tau^{(i)}_1}\1{OLT^{(i)}(1/2,K,(tn)^{ \frac{1-\epsilon}{\gamma} })}$ are i.i.d.~copies of the variable $X_{\tau_1}\1{OLT(1/2,K,(tn)^{ \frac{1-\epsilon}{\gamma} })}$ under $\PR_0^K[\cdot|D=\infty]$.\\
Using Theorem \ref{tailtau}, we have
\[
\PR_0\left[\abs{\abs{X_{\tau_1}\1{OLT(1/2,K,(tn)^{ \frac{1-\epsilon}{\gamma} })}}}_\infty\ge \frac{n^{1/4}}{2}\right]\to0.
\]
In order to take care of the second term, notice that, on the one hand,
\begin{align*}
&\PR_0^K\left[\left.\text{Card}\{1\le i \leq \lfloor tn\rfloor-1, OLT^{(i)}(1/2,K,(tn)^{ \frac{1-\epsilon}{\gamma} })\} \geq n^{\frac{1}{4}-\epsilon}\right|D=\infty\right]  \\
 \leq  &n^{-\frac{1}{4}+\epsilon}\ES_0^K\left[\left.\text{Card}\{ 1\le i \leq \lfloor tn\rfloor-1, OLT^{(i)}(1/2,K,(tn)^{ \frac{1-\epsilon}{\gamma} })\} \right|D=\infty\right] \\
 \leq  &n^{-\frac{1}{4}+\epsilon}\ES_0^K\left[\left.\sum_{i=1}^{\lfloor tn\rfloor -1} \1{OLT^{(i)}(1/2,K,(tn)^{ \frac{1-\epsilon}{\gamma} })} \right|D=\infty\right] \\
 \leq &n^{-\frac{1}{4}+\epsilon} \times tn\PR_0^K[OLT(1/2,K, (tn)^{ \frac{1-\epsilon}{\gamma} })|D=\infty] \le Cn^{-\frac{1}{4}+3\epsilon}=o(1),
\end{align*}
where we used Lemma~\ref{tailpn} and Lemma~\ref{equiv_pnqn} in the last line.

On the other hand, we can see that 
\begin{align*}
& \PR_0^K\left[\left.\max_{\substack{A\subset \{1,\ldots,\lfloor tn\rfloor-1\}, \\ \abs{A} \leq n^{ \frac{1}{4}-\epsilon }}} \sum_{j\in A}  X_{\tau^{(i)}_{1}} \geq n^{1/4}/2\right|D=\infty\right] \\
\leq & \PR_0^K\left[\left.\max_{1\le j\leq \lfloor tn\rfloor -1} X_{\tau^{(i)}_{1}} \geq n^{\epsilon}/2\right|D=\infty\right] \\
\leq & tn \PR_0^K\left[\left.X_{\tau_1} \geq n^{\epsilon}/2\right|D=\infty\right] =o(1),
\end{align*}
where we used Theorem~\ref{tailtauK}, with $K$ large enough compared to $\epsilon$ (which is fixed).

Since on $\left\{\abs{\abs{\sum_{i=1}^{\lfloor tn \rfloor-1} X^{(i)}_{\tau^{(i)}_1}\1{OLT^{(i)}(1/2,K,(tn)^{ \frac{1-\epsilon}{\gamma} })}}}_\infty\ge \frac{n^{1/4}}{2}\right\}$, we have either
\begin{enumerate}
\item $\text{Card}\{1\le i \leq \lfloor tn\rfloor-1, OLT^{(i)}(1/2,K,(tn)^{ \frac{1-\epsilon}{\gamma} })\} \geq n^{\frac{1}{4}-\epsilon}$,
\item or $\max_{\substack{A\subset \{1,\ldots,\lfloor tn\rfloor-1\}, \\ \abs{A} \leq n^{ \frac{1}{4}-\epsilon }}} \sum_{j\in A}  X_{\tau^{(i)}_{1}} \geq n^{1/4}/2$,
\end{enumerate}
the result follows.
\end{proof}

Define 
\[
Y_n(t)=\frac{X_{\tau_{\lfloor tn \rfloor}}}{n},\ Z_n(t)= \frac{X_{\tau_{\lfloor tn \rfloor}}- vnt}{ n^{1/2}} \text{ and } S_n(t)=\frac{\tau_{\lfloor nt\rfloor}}{\mathrm{Inv}(n)},
\]

where we also define the $d$-dimensional vector
\begin{align}\label{olaf}
v=\ES_0^K[X_{\tau_1}|D=\infty].
\end{align}

To obtain our limiting result it will be enough to prove the joint convergence of
\[
\left(Y_n(t),S_n(t)\right)_{0\le t \le T}\text{ and } \left(Z_n(t),S_n(t)\right)_{0\le t \le T}.
\]

Using the basic properties of slowly varying functions and the monotonicity of $\mathrm{Inv}(\cdot)$, one can prove that, for any constant $c>0$,
\begin{align*}
\frac{\mathrm{Inv}(\lfloor cn\rfloor)}{\mathrm{Inv}(n)}\to c^{\frac{1}{\gamma}}.
\end{align*}
By Proposition \ref{justin2}, the law of large numbers and the central limit theorem, we thus have, for any fixed $t\ge 0$,
\begin{align}\nonumber
S_n(t)&\xrightarrow{(d)} t^{1/\gamma}C_\infty S_\gamma,\\ \nonumber
Y_n(t)&\xrightarrow{a.s.} vt,\\ \label{chinois}
Z_n(t)&\xrightarrow{(d)}  \sqrt{\Sigma}B_t,
\end{align}
where $S_\gamma$ has a completely asymmetric stable law of index $\alpha$, $B_\cdot$ is a standard $d$-dimensional Brownian motion and $\sqrt{\Sigma}$ is some nonsingular $d\times d$ matrix such that $\Sigma:=\sqrt{\Sigma}^t \sqrt{\Sigma}$ is the covariance matrix of $X_{\tau_1}$ under $\PR_0^K[\cdot|D=\infty]$. The invertibility of $\sqrt{\Sigma}$ can be proved using an argument similar to \cite{SznPerco}, right after display $(3.40)$.\\

In the following results of process convergence, we use the uniform topology, denoted $U$, and two classical Skorokhod's topologies $J_1$ and $M_1$, see \cite{Whitt} for details on these topologies.\\
Recall that $D$ (resp.~$D^d$) is the space of $\mathbb{R}$-valued (resp.~$\mathbb{R}^d$-valued) c\`adl\`ag functions.

\begin{lemma} \label{onnaitonvitonmeurt} 
Fix some $T\ge0$.
The joint distribution of $(Y_n(t),S_n(t))_{0\leq t \leq T}$ converges to the distribution of $(v t, C_\infty \mathcal{S}_{\gamma}(t))_{0\leq t\leq T}$  in $D^d\times D$ in the $U\times M_1$-topology, where $\mathcal{S}_{\gamma}(\cdot)$ is a stable subordinator of index $\gamma$.

Moreover the joint distribution of $(Z_n(t),S_n(t))_{0\leq t \leq T}$ converges to the distribution of $( \sqrt{\Sigma}B_t, C_\infty \mathcal{S}_{\gamma}(t))_{0\leq t\leq T}$  in $D^d\times D$ in the $J_1\times M_1$-topology, where $B_\cdot$ is a standard Brownian motion independent of $\mathcal{S}_\gamma(\cdot)$.
\end{lemma}

\begin{proof} By Theorem 11.6.6 of \cite{Whitt}, we only have to prove the finite-dimensional convergence and the tightness of the sequences. Let us proceed in three steps.

\vspace{0.5cm} 

{\it Step 1: Joint convergence for one fixed time $t\ge0$}

\vspace{0.5cm} 
If $t=0$, the result is immediate, we then assume $t>0$.
Firstly, as $Y_n(t)$ converges to a constant, the joint convergence in distribution of $(Y_n(t),S_n(t))$ comes at once as soon as $S_n(t)$ converges.

Secondly, for $(Z_n(t),S_n(t))$, notice that, by Lemma \ref{approx_displ} and Proposition \ref{justin}, we only need to consider the joint limit of
\begin{align*}
J_n(t):=
\left(\frac{\sum_{i=1}^{\lfloor nt\rfloor }X_{\tau_{1}}^{(i)}\1{OLT^{(i)}(1/2,K,(tn)^{\frac{1-\epsilon}{\gamma}})^c}-vtn}{\sqrt{n}},\right.\\
\left. \frac{\sum_{i=1}^{\lfloor nt\rfloor }\tau_{1}^{(i)}\1{OLT^{(i)}(1/2,K,(tn)^{\frac{1-\epsilon}{\gamma}})}}{\mathrm{Inv}(n)}\right),
\end{align*}
where the variables $(X_{\tau_1}^{(i)},\tau^{(i)}_1)$ are i.i.d.~copies of $(X_{\tau_1},\tau_1)$ under $\PR_0^K[\cdot|D=\infty]$, and where $\epsilon\in(0,1/12)$ is a fixed constant.

The two sums occurring in the previous display are not independent but we will show that this couple has the same limit as a couple of independent random variables.\\
Let us define, in some probability space $\PR_H$, two independent sequences of i.i.d.~random variables $(H_{n,i}^1)_i$ and $(H_{n,i}^2)_i$ respectively distributed as $X_{\tau_{1}}$ under $\PR_0^K[\cdot|D=\infty, OLT(1/2,K,(tn)^{\frac{1-\epsilon}{\gamma}})^c]$ and ${\tau_{1}}$ under $\PR_0^K[\cdot|D=\infty, OLT(1/2,K,(tn)^{\frac{1-\epsilon}{\gamma}})]$. In the space probability, we independently define a binomial random variable $B_n$ of parameters $\lfloor nt\rfloor$ and $p_n:=\PR_0^K[OLT(1/2,K,(tn)^{\frac{1-\epsilon}{\gamma}})|D=\infty]$.\\
We claim that
\begin{align}\label{blockparty}
J_n(t)\stackrel{(d)}{=} \left(\frac{\sum_{i=1}^{\lfloor nt\rfloor -B_n}H_{n,i}^1-vtn}{\sqrt{n}},\frac{\sum_{i=1}^{B_n}H_{n,i}^2}{\mathrm{Inv}(n)}\right).
\end{align}

In order to prove that, let us define the random vector
\[
I_n:=(\1{OLT^{(i)}(1/2,K,(tn)^{\frac{1-\epsilon}{\gamma}})})_{1\le i \le \lfloor nt \rfloor},
\]
whose coordinates are i.i.d.~Bernoulli random variables with parameter $p_n$, and thus $|I_n|$ is a binomial random variable with parameters $\lfloor nt\rfloor$ and $p_n$. For any $0\le k \le \lfloor nt\rfloor$ and any $\bar{i}_k\subset \{i_1,...,i_k\}$ with $1\le i_1<\dots< i_k\le \lfloor nt \rfloor$, we denote ${\bf 1}_{n,\bar{i}_k}$ the $\lfloor nt\rfloor$-dimensional vector with its $i$-th component being $1$ if $i\in\bar{i}_k$ and $0$ otherwise.\\
For any measurable set $A$, we have
\begin{align*}
&\PR_0^K[J_n(t)\in A|D=\infty]=\sum_{k=0}^{\lfloor nt\rfloor}\sum_{\substack{\bar{i}_k=\{i_1,...,i_k\},\\ 1\le i_1< \dots < i_k\le \lfloor nt\rfloor}} \PR_0^K[I_n={\bf 1}_{n,\bar{i}_k}|D=\infty]\\
&\times \PR_0^K\left[\left.\left(\frac{\sum_{i\in \{1,...,\lfloor nt\rfloor\}\setminus \bar{i}_k }X_{\tau_{1}}^{(i)}-vtn}{\sqrt{n}}, \frac{\sum_{i\in\bar{i}_k }\tau_{1}^{(i)}}{\mathrm{Inv}(n)}\right)\in A\right|D=\infty, I_n={\bf 1}_{n,\bar{i}_k}\right]\\
&=\sum_{k=0}^{\lfloor nt\rfloor}\sum_{\substack{\bar{i}_k=\{i_1,...,i_k\},\\ 1\le i_1< \dots < i_k\le \lfloor nt\rfloor}} p_n^k(1-p_n)^{\lfloor nt\rfloor-k} \PR_H\left[\left(\frac{\sum_{i=1}^{\lfloor nt\rfloor -k}H_{n,i}^1-vtn}{\sqrt{n}},\frac{\sum_{i=1}^{k}H_{n,i}^2}{\mathrm{Inv}(n)}\right)\in A\right]\\
&=\sum_{k=0}^{\lfloor nt\rfloor} \PR_H\left[B_n=k,\left(\frac{\sum_{i=1}^{\lfloor nt\rfloor -k}H_{n,i}^1-vtn}{\sqrt{n}},\frac{\sum_{i=1}^{k}H_{n,i}^2}{\mathrm{Inv}(n)}\right)\in A\right]\\
&= \PR_H\left[\left(\frac{\sum_{i=1}^{\lfloor nt\rfloor -B_n}H_{n,i}^1-vtn}{\sqrt{n}},\frac{\sum_{i=1}^{B_n}H_{n,i}^2}{\mathrm{Inv}(n)}\right)\in A\right],
\end{align*}
which proves the claim. Besides, note that the marginal laws converge.\\
Moreover, it is clear that the following couple of random variables is independent:
\begin{align}\label{classique}
\left(\frac{\sum_{i=1}^{\lfloor nt\rfloor }H_{n,i}^1-vtn}{\sqrt{n}},\frac{\sum_{i=1}^{B_n}H_{n,i}^2}{\mathrm{Inv}(n)}\right).
\end{align}
So, if we prove that the distance between this couple and the right-hand side of \eqref{blockparty} goes to $0$ in probability, we will be allowed to conclude.
Recalling that the $H_{n,i}^1$'s have the law of $X_{\tau_1}$ under $\PR_0^K[\cdot|D=\infty,OLT(1/2,K,(tn)^{\frac{1-\epsilon}{\gamma}})^c]$, we have, for $n$ large enough,
\begin{align*}
\PR_H\left[\abs{\abs{\frac{\sum_{i=\lfloor nt\rfloor -B_n+1}^{\lfloor nt\rfloor}H_{n,i}^1}{\sqrt{n}}         }}_\infty>\epsilon\right]\le& \PR_H\left[\abs{\abs{\frac{\sum_{i=1}^{(1+t)n^{2\epsilon}}H_{n,i}^1}{\sqrt{n}}         }}_\infty>\epsilon\right]\\
&+\PR_H\left[B_n\ge p_n\lfloor nt\rfloor + n^{2\epsilon}\right]\\
\le & Cn^{-1/3}+\lfloor nt\rfloor p_n/n^{-4\epsilon}\le Cn^{-2\epsilon}=o(1),
\end{align*}
recalling that $\epsilon\in(0,1/12)$ is a constant and $t\in[0,T]$ where $T$ is also a fixed constant.\\
Hence, the marginal laws of \eqref{classique} converge in distribution respectively to $\sqrt{\Sigma} B_t$ and $t^{1/\gamma}C_\infty S_\gamma$, and, as the coordinates are independent, the couple converges jointly to a couple of independent random variables. This finally implies that
\[
(Z_n(t),S_n(t))\xrightarrow{(d)} ( \sqrt{\Sigma}B_t,t^{1/\gamma}C_\infty S_\gamma),
\]
where $B_t$ and $S_\gamma$ are independent.

\vspace{0.5cm} 

{\it Step 2: Finite-dimensional convergence}

\vspace{0.5cm}

Fix a positive integer $k$ and $k+1$ times $t_0=0\le t_1<...<t_k\le T$. Consider the vectors
\begin{align}\label{usure}
\left((Y_n(t_{i})-Y_n(t_{i-1}),S_n(t_i)-S_n(t_{i-1})\right)_{1\le i \le t_k}
\end{align}
and
\begin{align}\label{usure2}
\left((Z_n(t_{i})-Z_n(t_{i-1}),S_n(t_i)-S_n(t_{i-1})\right)_{1\le i \le t_k}.
\end{align}
As soon as $n$ is large enough, the variables $(Y_n(t_{i})-Y_n(t_{i-1}),S_n(t_i)-S_n(t_{i-1}))$, ${1\le i \le t_k}$, are independent and have the same limit as $(Y_n(t_{i}-t_{i-1}),S_n(t_i,t_{i-1}))$, ${1\le i \le t_k}$. The same holds for \eqref{usure2} and this implies the finite-dimensional convergence using basic properties on the increments of a stable subordinator.

\vspace{0.5cm} 

{\it Step 3: Tightness}

\vspace{0.5cm} 

Firstly, the process $(Y_n(t))_{t\in[0,T]}$ converges almost surely and uniformly to a constant. Besides, by Donsker's Theorem, the process $(Z_n(t))_{t\in[0,T]}$ converges to a Brownian motion on $D^d$ in the $J_1$-topology, which implies the tightness of the sequence by Prohorov's Theorem.\\

Secondly, we will need a criterion for tightness of probability measures on $D$, the space of $\R$-valued c\`adl\`ag functions. To this
end we define several moduli of continuity,
\begin{equation}
  \begin{split}
    w_f(\delta )&=
    \sup \big\{ \inf_{\alpha \in [0,1]}|f(t)-(\alpha f(t_1)+(1-\alpha )f(t_2))|:
      t_1\le t\le t_2\le T,t_2-t_1\le \delta \big\},
    \\
    v_f(t,\delta )&=
    \sup\big\{|f(t_1)-f(t_2)|:t_1,t_2\in [0,T]\cup (t-\delta ,t+\delta )\big\}.
  \end{split}
\end{equation}
The following result is a restatement of Theorem 12.12.3 of \cite{Whitt}.
\begin{theorem}[Theorem 12.12.3 of \cite{Whitt}]
  \label{t:tight}
  The sequence of probability measures $\{P_n\}$ on $D$ 
  is tight in the $M_1$-topology if 
  \begin{enumerate}
    \item[(i)] For each positive $\varepsilon $ there exist $c$ such that
    \begin{equation}
      P_n [f:\sup_{t\in[0,T]}|f(t)|>c]\le \varepsilon , \qquad n\ge 1.
    \end{equation}
    \item[(ii)] For each $\varepsilon >0$ and $\eta >0$, there exist 
    a $\delta $, $0<\delta <T$, and an integer $n_0$ such that
    \begin{equation}
      \label{e:Jcond}
      P_n[f:w_f (\delta )\ge \eta ]\le \varepsilon , \qquad n\ge n_0,
    \end{equation}
    and
    \begin{equation}\label{flammesdelenfer}
      P_n[f:v_f (0,\delta )\ge \eta ]\le \varepsilon \text{ and }
      P_n[f:v_f (T,\delta )\ge \eta ]\le \varepsilon , \qquad n\ge n_0.
    \end{equation}
  \end{enumerate}

\end{theorem}

Let us check that $S_n(\cdot)$ satisfies the two conditions of this theorem. For condition $(i)$, as $S_n(\cdot)$ is a.s.~nondecreasing, we just have to check the tightness of $S_n(T)$ which is easily obtained by the finite-dimensional convergence. For condition $(ii)$, note first that $w_f(\delta)$ is equal to $0$ when $f$ is nondecreasing. We then have to check the conditions \eqref{flammesdelenfer}. Using again the fact that $S_n(\cdot)$ is nondecreasing, we just need that, for any $\epsilon>0$ and $\eta>0$, there exists $\delta>0$ such that $\PR_0[S_n(\delta)\ge\eta]\le\epsilon$ and $\PR_0[S_n(T)-S_n(T-\delta)\ge\eta]\le\epsilon$. This is easily obtained using the finite-dimensional convergence.

Finally, by Theorem 11.6.7 of \cite{Whitt}, the tightness of $Y_n(\cdot)$ and $S_n(\cdot)$ implies that $(Y_n,S_n)$ and $(Z_n,S_n)$ are tight on the product space $D^d\times D$ in the $U\times M_1$-topology and $J_1 \times M_1$-topology respectively. This concludes the proof.
\end{proof}

Let us now introduce the inverse map on $D_u$ the subset of $D$ of functions $x$ that are unbounded above and such that $x(0)\ge 0$.  For $x\in D_u$, the inverse map of $x$ is defined as
\[
x^{-1}(t)=\inf\{s\ge 0:x(s)>t\},\quad\forall t\ge0.
\]
Besides, we define the subset $D_{u,\uparrow}$ of nondecreasing functions of $D_u$. We also define the subset $D_u^*$ of functions $x\in D_u$ such that $x^{-1}(0)=0$. We denote $D_\uparrow$ the subset of nondecreasing functions of $D$ and $C_{\uparrow\uparrow}$ the set of increasing (strictly) continuous functions.\\
We denote $S^{-1}_n(\cdot)$ the inverse map of $S_n(t)$. Also, recall the definition of the matrix $\sqrt{\Sigma}$ introduced in \eqref{chinois} and define ${v_0}:=v/||v||$ where $v$ is the vector defined in \eqref{olaf}. Finally, denote $I_d$ the $d\times d$ identity matrix, $P_{v_0}$ the projection matrix on $v_0$, that is the matrix such that, for any $x\in\Z^d$, $P_{v_0}x=(x\cdot v_0)v_0$, and let
\begin{equation}\label{water}
M_d:=C_\infty^{-\gamma/2}(I_d-P_{v_0})\sqrt{\Sigma}.
\end{equation}
Since $\sqrt{\Sigma}$ is invertible (see~\eqref{chinois}) and that $(I_d-P_{v_0})$ has rank $d-1$, it is clear $M_d$ has rank $d-1$. Also, note that $P_{v_0}M_d$ is the null matrix.

\begin{theorem}\label{seinfeld}
We have
\begin{align}\label{clavier1}
\Bigl(\frac{X_{\lfloor nt\rfloor}}{n^{\gamma}/L(n)}\Bigr)_{0\leq t \leq T} &\to (vC_\infty^{-\gamma} \mathcal{S}_\gamma^{-1}(t))_{t\in [0,T]},\\ \label{clavier2}
\left(\frac{X_{\lfloor nt\rfloor}-v\frac{n^\gamma}{L(n)} S^{-1}_{n^\gamma/L(n)}\left(\frac{nt}{\mathrm{Inv}(n^\gamma/L(n))}\right)]}{\sqrt{n^{\gamma}/L(n)}}\right)_{0\leq t \leq T} &\to \left(C_\infty^{-\gamma/2}\sqrt{\Sigma}B_{\mathcal{S}_\gamma^{-1}(t)}\right)_{t\in [0,T]},
\end{align}
and
\begin{equation}\label{clavier3}
\left(\frac{X_{\lfloor nt\rfloor}-\left(X_{\lfloor nt\rfloor}\cdot v_0\right)v_0}{\sqrt{n^{\gamma}/L(n)}}\right)_{0\leq t \leq T} \to \left(M_dB_{\mathcal{S}_\gamma^{-1}(t)}\right)_{t\in [0,T]},
\end{equation}
on $D^d$ in the uniform topology for \eqref{clavier1} and in the $J_1$-topology for \eqref{clavier2} and \eqref{clavier3}. The process $B_\cdot$ is a standard Brownian motion and $S^{-1}_{\gamma}(\cdot)$ is the inverse of a stable subordinator with index $\gamma$, independent of $B_\cdot$.
\end{theorem}

\begin{remark}
In \eqref{clavier3}, if we recenter with the projection on any other unit vector than $v_0$, then this quantity will diverge to infinity.
\end{remark}

\begin{proof}
Here, we will use two results from \cite{Whitt}. Firstly, Theorem 13.2.1 of \cite{Whitt} states that if $(x_n,y_n)\to(x,y)$ in $D^d\times D_\uparrow$ with $(x,y)\in C^d\times C_{\uparrow}$, then $x_n\circ y_n\to x\circ y$ in the uniform topology, hence in the Skorokhod's topologies. Secondly, Corollary 13.6.4 of \cite{Whitt} states that the inverse map from $(D_{u,\uparrow\uparrow},M_1)$ to $(C,U)$ is continuous.\\
Note that $(tn/\mathrm{Inv}(n^\gamma/L(n)))_{t\in[0,T]}$ converges uniformly to $(t)_{t\in[0,T]}$. Note also that $\mathcal{S}_\gamma^{-1}(\cdot)$ is a.s.~continuous and strictly increasing (See Lemma III.17 in \cite{BertoinLP}). Using these results and Lemma \ref{onnaitonvitonmeurt}, we have that
\begin{align*}
\left(S^{-1}_{n^\gamma/L(n)}\left(\frac{nt}{\mathrm{Inv}(n^\gamma/L(n))}\right)\right)_{t\in [0,T]}&\to (C_\infty^{-\gamma}\mathcal{S}_\gamma^{-1}(t))_{t\in [0,T]}\\
\left(Y_{n^\gamma/L(n)}\left(S^{-1}_{n^\gamma/L(n)}\left(\frac{nt}{\mathrm{Inv}(n^\gamma/L(n))}\right)\right)\right)_{t\in [0,T]}&\to (vC_\infty^{-\gamma} \mathcal{S}_\gamma^{-1}(t))_{t\in [0,T]}\\
\left(Z_{n^\gamma/L(n)}\left(S^{-1}_{n^\gamma/L(n)}\left(\frac{nt}{\mathrm{Inv}(n^\gamma/L(n))}\right)\right)\right)_{t\in [0,T]}&\to (C_\infty^{-\gamma/2}B_{\mathcal{S}_\gamma^{-1}(t)}\sqrt{\Sigma})_{t\in [0,T]},
\end{align*}
in the uniform topology (and thus $J_1$) for the two first limits, and  in the $J_1$-topology for the last one.

Now, we will be able to conclude if we prove that $X_{\lfloor nt\rfloor}$ is uniformly close to $X_{\tau_{ \left\lfloor \frac{n^\gamma}{L(n)}S^{-1}_{n^\gamma/L(n)}(t n/\mathrm{Inv}(n^\gamma/L(n))\right\rfloor  }}$.\\
It is elementary to verify that  ${\tau_{ \left\lfloor \frac{n^\gamma}{L(n)}S^{-1}_{n^\gamma/L(n)}(tn/\mathrm{Inv}(n^\gamma/L(n))\right\rfloor  }}$ is the smallest $\tau_i$ such that $\tau_i>tn$. It means that $X_{\tau_{ \left\lfloor \frac{n^\gamma}{L(n)}S^{-1}_{n^\gamma/L(n)}(t n/\mathrm{Inv}(n^\gamma/L(n))\right\rfloor  }}-X_{\lfloor nt\rfloor}$ is less than the size of a regeneration block, plus one.\\
Besides, we have that 
\begin{align*}
M_n&:=\max_{t\in[0,T]}\abs{\abs{    X_{\tau_{ \left\lfloor \frac{n^\gamma}{L(n)}S^{-1}_{n^\gamma/L(n)}(t n/\mathrm{Inv}(n^\gamma/L(n))\right\rfloor  }}-X_{\lfloor nt\rfloor}                  }}_\infty\\
&\le C \max_{k=1,...,\lfloor nT\rfloor +1}\{\abs{\abs{     X_{\tau_i}(t)-X_{\tau_{i-1 }}                   }}_\infty\},
\end{align*}
where it should be noticed that $X_{\tau_1}\ge1$ almost surely.
Using Lemma \ref{tail_chi_tau}, we have that
\[
\PR_0\left[\frac{M_n}{n^{\gamma/4}}\ge \delta\right]\le CnT\times n^{-2}=o(1).
\]
This concludes the proof of \eqref{clavier1} and \eqref{clavier2}.\\
For \eqref{clavier3}, notice that we just have to apply the linear combination $(I_d-P_{v_0})$ to \eqref{clavier2}, as $(I_d-P_{v_0})v=0$. All linear combinations are continuous in the $J_1$-topology at continuous functions, see Section 3.3 of \cite{Whitt}. This concludes the proof.

\end{proof}

 \appendix
\section{}

\subsection{Proofs of Theorems \ref{tailtau} and \ref{tailtauK}}
We give here the proof of Theorem \ref{tailtau}, which is very close the the proof of Theorem 4.1 in \cite{Fri11}, and Theorem \ref{tailtauK}. Before giving the core of the proofs, we need some defitions and  some lemmas.\\
Let us define the hitting time of the \emph{level} $n$ by
\[
\Delta_n:=T_{\Hc^+(n)},
\]
and the ladder times
%
\begin{equation}
\label{defW} W_0=0 \quad\mbox{and}\quad W_{k+1}=
\inf\{n \geq0, X_n\cdot\vec{\ell} >X_{W_k} \cdot\vec{\ell}
\}.
\end{equation}

We introduce the event
%
\begin{eqnarray}
\label{defM} M^{(K)}(n)&=&M(n)=\bigl\{\mbox{for $k$ with
$W_k \leq\Delta_n$,}\nonumber\\[-8pt]\\[-8pt]
&&\hspace*{5.6pt}\mbox{we have } X_{\mathcal{M}^{(K)}\circ\theta
_{W_k}+W_k}\cdot\vec{
\ell} -X_{W_k}\cdot\vec{\ell}\leq n^{1/2}\bigr\}.\nonumber
\end{eqnarray}
Moreover,  we introduce
the event
%
\begin{equation}
\label{defSn} S(n)=\bigl\{\mbox{for all $i$ with $S_i
\leq\Delta_n$ and $M_i<\infty$, } M_i-X_{S_i}
\cdot\vec{\ell}\leq n^{1/2}\bigr\}.
\end{equation}
As the definition of $M^{(K)}(n)$ depends on $\mathcal{M}^{(K)}$, we work in the exact same context as in \cite{Fri11} and the following Lemma holds.
\begin{lemma}[Lemma $6.6$ of \cite{Fri11}]
\label{Mn}
For any $M<\infty$, there exists $K_0$ such that, for any $K\geq K_0$
we have
\[
\PR\bigl[M^{(K)}(n)^c\bigr] \leq Cn^{-M}.
\]
\end{lemma}

Recalling the definition of $M$ at \eqref{defM}, $M_k$ at (\ref{defrealM}) an $S(n)$ at \eqref{defSn}, let us quote the following results.

\begin{lemma}[Analog of Lemma $7.1$ of \cite{Fri11}]
\label{M1}
We have
\[
\PR[M\geq n \mid D<\infty]\leq C\exp(-cn).
\]
\end{lemma}

\begin{lemma}[Analog of Lemma $7.2$ of \cite{Fri11}]
\label{Sn}
We have
\[
\PR\bigl[S(n)^c\bigr] \leq\exp\bigl(-n^{1/2}\bigr).
\]
\end{lemma}

The proofs of these Lemmas are exactly the same as in \cite{Fri11}: even though the definition of $D$ is not the same, it only uses the fact that if $M_i<\infty$ then $D\circ\theta_{S_i}+S_i <\infty$, that $\PR(D<\infty)>0$, together with Theorem \ref{BL} (tagged $5.1$ in \cite{Fri11}). Thus, we do not rewrite these proofs here.\\

We recall that $N$ was defined at (\ref{deftau1}). The proof of the following lemma is almost the same as Lemma $7.5$ in \cite{Fri11}, but as some non-trivial details change, we provide a complete proof.
%
\begin{lemma}
\label{Kn}
We have
\[
\PR[N\geq n]\leq\exp(-cn).
\]
\end{lemma}

\begin{proof}
We introduce the event
\[
C(n):=\{\mbox{for all }k \leq n\mbox{ such that }S_k<\infty\mbox{,
we have } D\circ S_k+S_k<\infty\},
\]
which verifies
%
\begin{equation}
\label{fa} \{N\geq n\}\subseteq C(n).
\end{equation}

Because of the way our regeneration times are constructed, we can see
that $C(n)$ is $P^{\omega}$-measurable with respect to $\sigma\{(X_k,Z_k)
\mbox{ with } k \leq S_{n+1}\}$, see the construction in Section \ref{sect_altcons}. Using Markov's property at $S_{n+1}$,
\begin{eqnarray*}
\PR\bigl[C(n+1)\bigr] &\leq& \sum_{x\in\Z^d} {\mathbf E}
\bigl[P^{\omega}\bigl[X_{S_{n+1}}=x,C(n)\bigr] P^{\omega}_x[D<
\infty]\bigr]
\\
&=& \sum_{x\in\Z^d} {\mathbf E}\Bigl[P^{\omega}
\bigl[X_{S_{n+1}}=x, C(n)\bigr] P^{\omega_x^K}_x[D<\infty]\Bigr],
\end{eqnarray*}
where we used Proposition \ref{prop_idconf}, the fact that $X_{S_{n+1}}$ is open, and where $\omega_x^K$ coincides with $\omega$ everywhere except that we fix $c_*(x,x+e_j)=K$ for all $1\le j\le 2d$.
Furthermore:

\begin{itemize}

\item[(1)] $\!\!P^{\omega}[X_{S_{n+1}}\,{=}\,x, C(n)] $ is measurable with
respect to 
$$\sigma\{c_*(e) : e^+\cdot\lv\vee e^-\cdot\lv\le x\cdot\lv\text{ or } x\in\{e^+,e^-\} \};$$
\item[(2)] $P^{\omega_x^K}_x[D<\infty]$ is
measurable with respect to 
$$\sigma\{c_*(e) : e^+\cdot\lv\vee e^-\cdot\lv> x\cdot\lv\text{ and } x\notin\{e^+,e^-\} \}.$$
\end{itemize}
So we have ${\mathbf P}$-independence between the random variables in $(1)$
and in $(2)$. Moreover, the ${\mathbf P}$-expectation of the second random variable does not depend on $x$, by translation invariance. Hence
\begin{align*}
 \PR\bigl[C(n+1)\bigr]
&\leq  \PR\bigl[C(n)\bigr] {\mathbf P} \Bigl[P^{\omega_0^K}_0[D<\infty] \Bigr]
\leq  \PR\bigl[C(n)\bigr](1-c),
\end{align*}
where we used Lemma \ref{posescape} and Remark \ref{suissenergie}. Therefore,
\[
\PR\bigl[C(n+1)\bigr]\leq(1-c) \PR\bigl[C(n)\bigr]\leq\cdots
\leq(1-c)^n,
\]
and we obtain the result by (\ref{fa}).
\end{proof}

The proof of Theorem \ref{tailtau} is now exactly the same as in \cite{Fri11} but we rewrite it here as a self-contained argument in order to make it easier to the reader.

\begin{proof}[Proof of Theorem \ref{tailtau}]
Recalling definitions (\ref{defS}), the definition of $M_k$ at (\ref{defrealM}), the definition \eqref{defW} of $W_k$ and the definition \eqref{defM} of $M(n)$, we may see that
\mbox{$\{T_{\mathcal{H}^+(M_k)},\break k\geq0\} \subset\{W_k, k\geq0\}$.} This
means that on $M(n)$, 
\[
\mbox{for $k$ such that $S_{k} \leq\Delta_n$}\qquad \mbox{we have }
X_{S_{k+1}}\cdot\vec{\ell} -X_{T_{\mathcal{H}^+(M_k)}}\cdot\vec{\ell
}\leq
n^{1/2}.
\]

Recalling the definition of \eqref{defSn} and noticing that $X_{T_{\mathcal{H}^+(M_k)}}\cdot\vec{\ell} \leq M_k+1$,
we may see that, on $S(n)\cap M(n)$
\[
X_{S_{k+1}}\cdot\vec{\ell}-X_{S_k}\cdot\vec{\ell}
\leq2n^{1/2}+1
\]
for any $k$ with $S_{k}<\Delta_n$ and $M_k<\infty$. By induction, this
means that if $k\leq n^{1/2}/3$, $S_{k}<\Delta_n$ and $M_k<\infty$, then
\[
X_{S_{k+1}}\cdot\vec{\ell} \leq k\bigl(2n^{1/2}+1\bigr)<n
\quad\mbox{and}\quad S_{k+1}\leq\Delta_n,
\]
and the second part following from the fact that $X_{S_{k+1}}$ is a new
maximum for the random walk in the direction $\vec{\ell}$. In
particular, if $N\leq n^{1/2}/3$, then we can apply\vspace*{1pt} the previous
equation to $k=N-1$. Recalling (\ref{deftau1}) we see that, if $\{N
\leq n^{1/2}/3\}$ and $M(n)\cap S(n)$, then for $n$ large enough,
\[
X_{\tau_1}\cdot\vec{\ell} \leq\bigl(n^{1/2}/3\bigr)
\bigl(2n^{1/2}+1\bigr) < n.
\]

Thus
\begin{eqnarray*}
\PR[X_{\tau_1}\cdot\vec{\ell} \geq n] &\leq& \PR\bigl[N \geq
n^{1/2}/3\bigr] +\PR\bigl[M(n)^c\bigr] +\PR
\bigl[S(n)^c\bigr]
\\
&\leq& 2\exp\bigl(-cn^{1/2}\bigr) + n^{-M}
\leq3n^{-M}
\end{eqnarray*}
by Lemmas \ref{Kn}, \ref{Mn} and \ref{Sn}. This completes
the proof.
\end{proof}

Finally, we give a proof of Theorem \ref{tailtauK}.

\begin{proof}[Proof of Theorem \ref{tailtauK}]
 Note that, if $X_{\tau_1}\cdot\vec{\ell} \geq n$, then $(X_{\tau_1}\cdot\vec{\ell})\circ \theta_{T_{\Hc^+(n/2)}} \geq n/4$.
 Therefore, we have
 \begin{align*}
& \PR_0^K\left[X_{\tau_1}\cdot\vec{\ell} \geq n\right] \le  \PR_0^K\left[T_{\partial B(n/2,(n/2)^{\alpha})}\neq T_{ \partial^+ B(n/2,(n/2)^{\alpha})}\right]\\
 &+\sum_{z\in \partial^+ B(n/2,(n/2)^{\alpha})}\PR_0^K\left[X_{T_{\partial B(n/2,(n/2)^{\alpha})}}=z, T_{\mathcal{H}^-_{e_1}}\circ \theta_{T_{z}}<\infty\right]\\
 &+\sum_{z\in \partial^+ B(n/2,(n/2)^{\alpha})}\PR_0^K\left[X_{T_{\partial B(n/2,(n/2)^{\alpha})}}=z, T_{\mathcal{H}^-_{e_1}}\circ \theta_{T_{z}}=\infty, (X_{\tau_1}\cdot\vec{\ell})\circ \theta_{T_{z}} \geq n/4\right]
\\
\le & \PR_0^K\left[T_{\partial B(n/2,(n/2)^{\alpha})}\neq T_{ \partial^+ B(n/2,(n/2)^{\alpha})}\right]\\
 &+\sum_{z\in \partial^+ B(n/2,(n/2)^{\alpha})}\PR^{0,K}_z\left[ T_{\mathcal{H}^-_{e_1}}<\infty\right]\\
 &+\sum_{z\in \partial^+ B(n/2,(n/2)^{\alpha})}\PR^{0,K}_z\left[ T_{\mathcal{H}^-_{e_1}}=\infty, X_{\tau_1}\cdot\vec{\ell}\geq n/4\right],
\end{align*}
using Markov's property and where $\PR_z^{0,K}$ is defined in Definition \ref{defP0K}.\\
Now, for any environment $\omega$, $P^\omega_z\left[T_{\mathcal{H}^-_{e_1}}<\infty\right]$ and $P^\omega_z\left[ T_{\mathcal{H}^-_{e_1}}=\infty, X_{\tau_1}\cdot\vec{\ell}\geq n/4\right]$ do not depend on the conductances around $0$. Hence, we obtain
\begin{align*}
&\PR_0^K\left[X_{\tau_1}\cdot\vec{\ell} \geq n\right]\le  \PR_0^K\left[T_{\partial B(n/2,(n/2)^{\alpha})}\neq T_{ \partial^+ B(n/2,(n/2)^{\alpha})}\right]\\
 &+\sum_{z\in \partial^+ B(n/2,(n/2)^{\alpha})}\PR_z\left[ T_{\mathcal{H}^-_{e_1}}<\infty\right]+\sum_{z\in \partial^+ B(n/2,(n/2)^{\alpha})}\PR_z\left[ X_{\tau_1}\cdot\vec{\ell}\geq n/4\right]
\\
\le & Ce^{-cn}+Cn^{c(d)}n^{-M},
 \end{align*}
for any $M<\infty$, as soon as $K$ is large enough. Here we used  Theorem \ref{BLK}, Theorem \ref{thisisboring} and Theorem \ref{tailtau}.
 \end{proof}

\subsection{A result on electrical networks}
\begin{lemma} \label{psyche}
Let $(G,(c(e))_{e\in E})$ be a finite network with set of vertices $V$ and set of edges $E$ and consider a random walk $X$ on this network. Fix some subset of edges $E_0\subset E$ and fix $\delta\in V$. We have, for any $y\in V$, $y\sim \delta$,
\[
E_{y}[T_{\delta,E_0}^+] \leq \frac{2}{c(e_{y\delta})}  \sum_{e\in E_0} c(e),
\]
with $e_{y\delta}$ being any edge linking $y$ and $\delta$, and where $T_{\delta,E_0}^+$ is defined in \eqref{kurt}.
\end{lemma}

\begin{proof}
Let us denote, for $x,y\in V$, $\mathcal{G}^G_{\delta}(x,z)=E^G_x[\sum_{i=0}^{T_\delta^+} \1{X_i=z}]$ the standard Green function killed at $\delta$. For $x\sim z$, we introduce 
\[
S^\delta_{\vec{xz}}:=\mathrm{Card}\{i < T_\delta,\ X_i=x \text{ and } X_{i+1}=z\}.
\]

%

It is possible to prove (see proof of Propositions 2.1 and 2.2. of~\cite{LP}) that, for $y\neq \delta$,
\[
E_y[S^\delta_{\vec{xz}}]=c(x,z) v_{y\to \delta}(x),
\]
where $v_{y\to \delta}(x)$ is the potential (or voltage) at $x$ when a unit current flows from $y$ to $\delta$ that verifies $v_{y\to \delta}(\delta)=0$. Furthermore, as the voltage function is harmonic on $V\setminus \{\delta,y\}$, using the Maximum Principle (see Section 2.1 of \cite{LP}), we have
\[
v_{y\to \delta}(x)\leq v_{y\to \delta}(y)= \mathcal{R}^G(y\leftrightarrow \delta),
 \]
 where $\mathcal{R}^G(y\leftrightarrow \delta)$ is the effective resistance between $y$ and $\delta$ and where we used Proposition 2.1 together with display $(2.5)$ from \cite{LP}.
 
 By Rayleigh's monotonicity principle, we see that, when $y\sim \delta$, for any edge $e_{y\delta}$ linking $y$ and $\delta$, we have $\mathcal{R}^G(y\leftrightarrow \delta) \leq 1/c(e_{y\delta})$. Using the previous equations this leads to the upper-bound
 \begin{align*}
   E_{y}[T_{\delta,E_0}^+]  
 &=   \sum_{\substack{x,z\in V:\\ [x,z]\in E_0}} E_y[S^\delta_{\vec{xz}}] \leq  \sum_{\substack{x,z\in V:\\ [x,z]\in E_0}} \frac{c(x,z)}{c(e_{y\delta})} \leq   \frac{2}{c(e_{y\delta})} \sum_{e\in E_0} c(e).
 \end{align*}
 \end{proof}


{\bf Acknowledgement}

\vspace{0.5cm}

\noindent This work was done while DK was Postdoc at the Ecole Polytechnique F\'ed\'erale de Lausanne (EPFL). The authors would like to thank the EPFL for allowing DK to travel to Universit\'e de Montr\'eal and AF to travel to the EPFL. We also would like to thank the anonymous Referee who helped to shorten and improve some of the proofs.

\end{document}